\documentclass[twoside, 14pt]{article}
\usepackage{amsmath,amssymb,amsthm,mathrsfs}
\usepackage{ulem}
\usepackage[margin=2.0cm]{geometry} 
\usepackage{lipsum}
\usepackage{titlesec,hyperref}
\usepackage{fancyhdr}
\usepackage[numbers,sort&compress]{natbib}
\usepackage{color}
\usepackage[titletoc]{appendix}
\usepackage{enumitem}  

\fancypagestyle{plain}
{
	\fancyhf{}
	
}

\titleformat{\subsection}{\it}{\thesubsection.\enspace}{1.5pt}{}
\titleformat{\subsubsection}{\it}{\thesubsubsection.\enspace}{1.5pt}{}

\newtheorem{theorem}{Theorem}[section]

\newtheorem{lemma}[theorem]{Lemma}
\newtheorem{proposition}[theorem]{Proposition}
\newtheorem{remark}[theorem]{Remark}
\newtheorem{corollary}[theorem]{Corollary}

\numberwithin{equation}{section}

\newcommand\tu{{\tilde{u}}}
\newcommand\tv{{\tilde{v}}}
\newcommand\tw{{\tilde{w}}}

\newcommand\tue{{\tilde{u}^{\epsilon}}}

\newcommand\twe{{\tilde{w}^{\epsilon}}}
\newcommand\use{{(u^s+\tilde{u}^{\epsilon})}}

\newcommand\uve{{\big(\partial_x\tilde u^{\epsilon}+K\partial_{y}\tilde u^{\epsilon}+\partial_{y}K(u^s+\tilde u^{\epsilon})\big)}}
\newcommand\uveo{{\big(\partial_x\tilde u_{0}+\partial_{y}(K(u^s_{0}+\tilde u_{0}))\big)}}
\newcommand\uvmuo{{(\partial_x\mu^{\epsilon}+\partial_{y}\mu^{\epsilon})}}

\newcommand\tvae{{\tilde{\varphi}^{\epsilon}}}

\newcommand\ir{\int_{{\mathbb{R}^3_+}}}

\newcommand\useo{{(u^s_{0}+\tilde{u}_{0})}}

\newcommand\KK{\big(K(u^s + \tilde{u}^\epsilon)\big)}

\allowdisplaybreaks

\begin{document}
	\title{Long time well-posedness for the 3D Prandtl boundary
		layer equations with a special structure \hspace{-4mm}}
	\author{Yuming Qin$^{\ast}$ \quad Junchen Liu \\[10pt]
		\small {School of Mathematics and Statistics, Institute for Nonlinear Sciences, Donghua University,}\\
		\small {201620, Shanghai, P. R. China}\\[5pt]
	}

	\footnotetext{Email: \it $^{\ast}$   yuming@dhu.edu.cn, yuming\_qin@hotmail.com}
	\date{}
	
	\maketitle

	\begin{abstract}
		{This paper is concerned with  existence,
			uniqueness and stability  of  the  solution for the 3D   Prandtl equation in a polynomial weighted  Sobolev space. The main novelty of this paper is to directly  prove the long time well-posedness to  3D   Prandtl equation under monotonicity condition $\partial_{z} u >0$ and  a special structural assumption $v=Ku$ $\big(\partial_{z}\big(\frac{v}{u}\big) \equiv 0\big)$  by the energy method. Moreover, the solution's lifespan can be extended to any large $T$, provided that the initial data with a perturbation lie in the monotonic shear profile of small size $e^{-T}$.	
			This result extends the local well-posedness results established 
			by Liu-Wang-Yang  \cite{Liu-Wang-Yang-1-2017} (Adv. Math. 308 (2017) 1074-1126) 
			and Qin-Wang \cite{Qin-Wang-2024} (J. Math. Pure.  Appl. 194 (2025)  103670) for the 3D Prandtl equations to long-time well-posedness.
		
		}

		\vspace*{5pt}
		
	\vspace*{5pt}
	\end{abstract}

	\tableofcontents

\section{Introduction}
In this paper, we consider the long time well-posedness of the 3D Prandtl  equations in domain
 $\mathbb{R}^{3}_{+} \overset{def}{=} \{(x, y, z)|(x, y, z) \in \mathbb{R}^{2}\times \mathbb{R}^+; z>0 \}$, which read a fluid flow
\begin{equation}
\left\{\begin{aligned}
	&\partial_{t}u+u\partial_{x}u+v\partial_{y}u+w\partial_{z}u+\partial_{x}P=\partial^{2}_{z}u,\\
	&\partial_{t}v+u\partial_{x}v+v\partial_{y}v+w\partial_{z}v+\partial_{y}P=\partial^{2}_{z}v,\\
	&\partial_{x}u+\partial_{y}v+\partial_{z}w=0,\\
	&(u, v, w)|_{z=0}=(0, 0, 0), \ \ \underset{z\rightarrow +\infty}{\lim}(u,v)=\big(U(t, x, y), V(t, x, y)\big),\\
	&(u,v)|_{t=0}=\big(u_{0}(x,y,z), v_{0}(x,y,z)\big),
\end{aligned}\right.	
\label{3d prandtl}
\end{equation}
where $(u,v)$ and $w$ are  the tangential component and the normal  component of the velocity field respectively. $\big(U(t, x, y), V(t, x, y)\big)$ and $P(t, x, y, z)$ are the boundary traces of the tangential velocity field and pressure of the outer flow, satisfying Bernoulli equations
\begin{equation}
	\left\{\begin{aligned}
		\partial _{t}U+(U\partial _{x} +V\partial _{y} )U+\partial _{x}P=0,\\
		\partial _{t}V+(U\partial _{x} +V\partial _{y}  )V+\partial _{y}P=0.
	\end{aligned}\right.	
	\label{Bernoulli}
\end{equation}

The Prandtl boundary layer equations were first introduced by Ludwig Prandtl in 1904 to describe fluid flow near a solid boundary.      When a fluid flows over a solid surface, such as the wing of an airplane, the viscosity of the fluid leads to the formation of a thin layer near the boundary, called the boundary layer, where the effects of viscosity are significant.     Outside this thin boundary layer, the viscosity is negligible, and the flow can be approximated by the inviscid Euler equations.   One of the central problems in fluid mechanics is the rigorous justification of the inviscid limit of the Navier-Stokes equations with no-slip boundary conditions.     In this limit, the viscosity tends to zero, and the solutions of the Navier-Stokes system are expected to converge to the solutions of the Euler equations, except near the boundary, where the Prandtl boundary layer plays a crucial role.     A key step in proving the inviscid limit is to establish the well-posedness of the Prandtl system, which governs the boundary layer flow. Without this well-posedness, the connection between solutions of the Navier-Stokes and Euler equations in such limits would remain elusive.

Early studies of the Prandtl equations primarily focused on two-dimensional (2D) flows, where considerable progress was made, firstly by Oleinik \cite{Oleinik-1966}, who  proved
the local existence and uniqueness in H\"older
spaces for the 2D Prandtl equations under the monotonicity condition on the tangential velocity.
This result, along with an expanded introduction to boundary layer theory, was presented in the classical  book \cite{Oleinik-Samokhin-1999} by Oleinik and Samokhin. By using  a so-called  Crocco transformation developed in \cite{Oleinik-1966,Oleinik-Samokhin-1999},  Xin and Zhang \cite{Xin-Zhang-2004} obtained a global existence of BV weak solutions to the 2D unsteady
Prandtl system  with the addition of favorable condition $(\partial_{x}P \leq 0)$ on pressure. Motivated
by a direct energy method, instead of considering Crocco transformation, which can  recover Oleinik's well-posedness results,  Alexandre et al. \cite{Alexandre-Wang-Xu-Yang-2015}
proved that the solution exists locally with respect to time in the weight Sobolev spaces via 
applying Nash-Moser iteration, 
when the initial datum is a small
perturbation of a monotonic shear flow,  but the life span of the solution is very short.
This is a bit different from \cite{Alexandre-Wang-Xu-Yang-2015}. 
Masmoudi and Wong \cite{Masmoudi-Wong-2015}  obtained
a prior estimate of the  regularized Prandtl equations  by using classical energy method, and then proved the local existence to the two-dimensional Prandtl equations  by using weak convergence method. The key
observation is that  
a cancellation
property in the convection terms to overcome the loss of $x$-derivative in the tangential direction, which is valid under the monotonicity assumption.
Based on the works \cite{Masmoudi-Wong-2015,Alexandre-Wang-Xu-Yang-2015},
the first result of global existence of  solutions to the  2D Prandtl equations  in the Sobolev space with a polynomial weight is traced back to Xu and Zhang \cite{Xu-Zhang-2017}, who obtained the long time well-posedness  on the half plane,  and proved that the lifespan 
$T$ of solutions can be arbitrarily large when the initial datum is a small perturbation around the monotonic shear profile.  Moreover, there
are some results on the two-dimensional  Prandtl boundary layer equations  under the monotonicity assumption, see \cite{Kukavica-Masmoudi-Vicol-Wong-2014,Fan-Ruan-Yang-2021,Wang-Xie-Yang-2015,Chen-Wang-Zhang-2018}.

In violation of Oleinik's monotonicity setting, some instability and ill-posedness mechanisms are unfiltered out. Grenier \cite{Grenier-2000},
Hong and Hunter \cite{Hong-Hunter-2003}
gave  nonlinearly unstable solutions of Prandtl boundary layer equations. 
Some results of ill-posedness to the 2D  Prandtl  equations were established 
for linear cases in \cite{Gerard-Varet-Dormy-2010,Liu-Yang-2017} and 
nonlinear cases in \cite{Gerard-Varet-Nguyen-2010,Guo-Nguyen-2011}. 
Finite-time blow-up of smooth solutions with certain class of initial data was showed by E-Engquist \cite{E-Engquist-1997}. Recently, Dalibard et al. \cite{Dalibard-Dietert-GerardVaret-Marbach-2018}
considered a unsteady interactive boundary layer  model, which is a famous extension of the Prandtl equation, and then studied
linear well-posedness 
and  strong unrealistic instabilities.

Without Oleinik's monotonicity assumption for the  2D case, the solutions  and data are
desired to be in the analytic or Gevrey classes.
For the framework of the analyticity, by using a Cauchy-Kowalewski 
argument, Sammartino and Caflisch \cite{Sammartino-Caflisch-1-1998,Sammartino-Caflisch-2-1998} first proved
local well-posedness to the 2D  Prandtl boundary (also holds on the 3D case)
with initial data that are analytic in $x$-variable and  $y$-variable. 
The result in \cite{Sammartino-Caflisch-1-1998,Sammartino-Caflisch-2-1998}
was later improved by Lombardo, Cannone and Sammartino \cite{Lombardo-Cannone-Sammartino-2003} via removing the requirement of analyticity
in $y$ variable on the initial data. This improvement relies on the regularizing effect of the diffusion operator $\partial_{t}-\partial_{y}^{2}$.
For a complete survey on the analyticity hypothesis for the data,
we refer the readers to \cite{Kukavica-Vicol-2013,Cannone-Lombardo-Sammartino-2001,Cannone-Lombardo-Sammartino-2013,Kukavica-Masmoudi-Vicol-Wong-2014,Igntova-Vicol-2016,Zhang-Zhang-2016,Paicu-Zhang-2021} and the references therein.
For the framework of the Gevrey class,
G\'erard-Varet and Masmoudi \cite{Gerard-Varet-Masmoudi-Gevrey-2015}
first  proved
the local well-posedness of the two-dimensional Prandtl  equations 
for the initial data without analyticity or monotonicity that belong to the Gevrey class $\frac{7}{4}$.
The Gevrey index $\sigma=\frac{7}{4}$  in \cite{Gerard-Varet-Masmoudi-Gevrey-2015}  was extended to  $\sigma \in [\frac{3}{2}, 2]$ in  \cite{Li-Yang-2020}, 
for data that are small perturbations of a shear flow with a single
non-degenerate critical point,  where $\sigma=2$ is optimal by combining
with the ill-posedness results in  \cite{Gerard-Varet-Dormy-2010}.
The well-posedness for the linearized Prandtl equation around a non-monotonic shear flow was   obtained by Chen, Wang and Zhang \cite{Chen-Wang-Zhang-2018} in Gevery class $2-\theta$ for any $\theta>0$.
After that, Dietert and G\'ervard-Varet \cite{Dietert-Gerard-Varet-2019}
achieved the local well-posedness  for the initial data with Gevrey class $2$ in the horizontal variable $x$ and Sobolev regularity in normal variable $y$, and further improved the result of \cite{Gerard-Varet-Masmoudi-Gevrey-2015}, which is due to the removal of single non-degenerate critical points on the Gevrey setting. 
Inspired by aforementioned works of local existence  especially  \cite{Dietert-Gerard-Varet-2019},
Wang, Wang and Zhang \cite{Wang-Wang-Zhang-2024} proved global existence of Gevrey-2 small solutions,
which is an extension from small  analytic data in \cite{Paicu-Zhang-2021} to optimal Gevrey regular data.

Compared to the 2D case, the results of the three-dimensional boundary layer equations were very few. 
A well-posedness theory for the
three-dimensional Prandtl equations was first studied by Sammartino and Caflisch \cite{Sammartino-Caflisch-1-1998,Sammartino-Caflisch-2-1998}
in the analytic case. 
 Qin and  Wang \cite{Qin-Wang-2024}, Liu, Yang and Wang  \cite{Liu-Wang-Yang-1-2017} obtained the local existence 
of solutions to the 3D Prandtl equations with a special structure by the energy method. 
Later on, Liu, Yang and Wang  \cite{Liu-Wang-Yang-1-2017} also \cite{Liu-Wang-Yang-3-2017}
gave an ill-posedness criterion  which means that 3D Prandtl equations can be linearly unstable around the shear flow  even under the monotonic conditions. 
Without any structual assumption, the local well-posedness  was solved by 
Li, Masmoudi and Yang \cite{Li-Masmoudi-Yang-2021} based on the establishment of  a novel  cancellation in Gevrey spaces with the optimal class of regularity 2.
Gevrey well-posedness with Gevrey index$\leq2$ of the 2D and 3D Prandtl equations of degenerate hyperbolic type was  proved in \cite{Li-Xu-2021}.
 Recently, in \cite{CRY2025}, the local well-posedness of the 3D compressible boundary layer equation is obtained when the initial datum is real-analytic in the tangential direction and has Sobolev regularity in the normal direction.

In the aforementioned works, only local-posedness in three-dimensional case are achieved. Global existence of weak solutions to 3D  Prandtl  equations and 3D axially
symmetric Prandtl  equations was obtained by Liu et al. \cite{Liu-Wang-Yang-2-2017} and Pan et al. \cite{Pan-Xu-2024}, respectively. In  the analytical framework,
Zhang and Zhang \cite{Zhang-Zhang-2016} showed that 
the Prandtl system  in $\mathbb{R}_{+}\times \mathbb{R}^{d-1} (d=2, 3)$ has a unique solution with the lifespan $T_{\epsilon} \geq \epsilon^{-\frac{4}{3}}$. 
When initial
datum is real-analytic with respect to the tangential variable, 
Lin and Zhang \cite{Lin-Zhang-2020} got an almost global existence solution by introducing new linearly-good unknowns  for the 3D Prandtl system whose lifespan is greater than $(\epsilon^{-1}/\log(\epsilon^{-1}))$. The analytical results  are extended to Gevrey-2 spaces by Pan and Xu \cite{Pan-Xu-2022}. Moreover,
the lifespan of the Gevrey-2 solution is at least of size $\epsilon^{M}$ if the initial data are  with size of $\epsilon$.


To our best knowledge, so far there is no result  concerning on the long time
behavior of solutions for the 3D Prandtl equations in Sobolev framework. This is our preliminary interest of this paper. The main purpose
of this paper is to achieve that  the long time well-posedness, which improve the result of \cite{Xu-Zhang-2017} to the 3D  setting.
Since the appearance of the secondary flow in the 3D Prandtl equations, the monotonicity assumption  is
insufficient to ensure the long time  well-posedness of Prandtl equations in Sobolev space. Inspired by  \cite{Liu-Wang-Yang-3-2017,Liu-Wang-Yang-1-2017,Qin-Wang-2024}, we need to impose an additional structural assumption
\begin{eqnarray}
	(u(t,x,y,z),K(t, x,y)u(t,x,y,z),w(t,x,y,z)).
\end{eqnarray}
Correspondingly,
 the outer Euler flow takes the following form on the boundary $\{z=0\}$,
$$(U(t,x,y), K(t, x,y)U(t,x,y),0 ). $$
In what follows, we shall consider the following the equivalent system of Prandtl equations \eqref{3d prandtl} (see \cite{Liu-Wang-Yang-1-2017} or \cite{Qin-Wang-2024}  for the specific proof),
\begin{equation}
	\left\{\begin{aligned}
		&\partial_{t}u+u\partial_{x}u+Ku\partial_{y}u+w\partial_{z}u+\partial_{x}P=\partial^{2}_{z}u,\\
		&\partial_{x}u+\partial_{y}(Ku)+\partial_{z}w=0,\\
		&(u, w)|_{z=0}=(0,  0), \ \ \underset{z\rightarrow +\infty}{\lim}u=U(t, x, y),\\
		&u|_{t=0}=u_{0}(x,y,z).
	\end{aligned}\right.	
	\label{new1}
\end{equation}
Based on the above equation, we can consider the following condition (H):
\begin{description}
	\item[H1:] the function \( K \) only depends on \((x,y)\) and satisfies the Burgers equation in \(\mathbb{R}^2\)
	\[
	\partial_x K + K \partial_y K = 0. 
	\]
	
	\item[H2:] the initial-boundary data  $\eqref{3d prandtl}_{3}$ and $\eqref{3d prandtl}_{4}$   has the following form
	\[
	(U, KU) \text{ and } (u_{0}, Ku_{0}),
	\]
	 and the following holds
	\[
			\partial_t U+ U\partial_x U+ K U\partial_y U+\partial_x P=0.
	\]
\end{description}

Furthermore, for the sake of convenience, we choose the uniform outflow $U=1$ which implies $\partial_{x}P=0$. The Prandtl equations \eqref{new1} degenerate to 
\begin{equation}
	\left\{\begin{aligned}
		&\partial_{t}u+u\partial_{x}u+Ku\partial_{y}u+w\partial_{z}u=\partial^{2}_{z}u,\\
&\partial_{x}u+\partial_{y}(Ku)+\partial_{z}w=0,\\
&(u, w)|_{z=0}=(0,  0), \ \ \underset{z\rightarrow +\infty}{\lim}u=1,\\
&u|_{t=0}=u_{0}(x,y,z).
	\end{aligned}\right.	
\label{3d prandtl constant outflow }
\end{equation}

Let us first introduce some  notations and weighted Sobolev spaces for later use.  \\
\textbf{Notations} Throughout the paper, we always use $L_{xy}^{\infty}(L_{z}^{2})=L^{\infty}\big(\mathbb{R}^{2}; L^{2}(\mathbb{R}_{+})\big)$ to stand for the classical
Sobolev space, as does the Sobolev space
$L_{xy}^{2}(L_{z}^{\infty})$;  we also denote 
$\partial^{\alpha}= \partial^{\alpha_{1}}_{x}
\partial^{\alpha_{2}}_{y} \partial^{\alpha_{3}}_{z}$
with each multi-index $\alpha=(\alpha_{1}, \alpha_{2}, \alpha_{3}) \in
\mathbb{Z}_{+}^{3}$. \\
\textbf{Weighted Sobolev spaces }For any $\lambda>0$ and $m \in \mathbb{N}^{+}$, 
weighted Sobolev spaces are defined as follows:
\begin{align*}
&\left\| f\right\| _{L^2_{\lambda}(\mathbb{R}_{+}^{3})}^{2} = \int_{\mathbb{R}_{+}^{3}}
\langle z \rangle^{2\lambda+2\alpha_{3}}\left|   f(x, y, z) \right| ^{2} dxdydz, \\ 
&\left\| f \right\| _{H^{m, m-1}_{\lambda}(\mathbb{R}_{+}^{3})}^{2}
=\sum_{|\alpha| \leq m, \alpha_{1}+\alpha_{2} \leq m-1} \int_{\mathbb{R}_{+}^{3}}
\langle z \rangle^{2\lambda+2\alpha_{3}}\left|  \partial^{\alpha}
f(x, y, z) \right| ^{2} dxdydz
, \\
&\left\| f \right\| _{H^{m}_{\lambda}(\mathbb{R}_{+}^{3})}^2=\left\| f \right\| _{H^{m, m-1}_{\lambda}(\mathbb{R}_{+}^{3})}^2+\left\| \partial_{xy}^{m} f \right\| _{L^{2}_{\lambda}(\mathbb{R}_{+}^{3})}^2,
\end{align*}
where the weight is defined by $\langle z \rangle =(1+|z|^2)^{\frac{1}{2}}$, $\partial_{xy}^{m}$ is written as
$\partial_{xy}^{m}=\sum_{i=0}^{m}\partial_{x}^{i}\partial_{y}^{m-i}$.

We denote the shear flow by  $u^{s}$. Subsequently, we consider solutions to the Prandtl equations as perturbations about some shear flow. That is to say,
 \begin{align*}
	u(t, x, y, z) = u^s(t, z) + \tilde{u}(t, x, y,z ),\,\, t\geq 0.
\end{align*} 
Suppose the  initial shear flow that $u^{s}_{0}$ satisfies the following conditions:
\begin{align} 
\begin{cases}
u^{s}_{0} \in C^{m+4} ([0, +\infty[), \lim\limits_{y \rightarrow +\infty} u^{s}_{0}(z)=1; \\
\partial_{z}^{2p}u^{s}_{0}(0)=0, \quad 0 \leq 2p \leq m+4; \\
c_{1} \langle z \rangle^{-k} \leq \partial_{z} u^{s}_{0}(z) \leq c_{2} \langle z \rangle^{-k}, \forall z \geq 0, \\
\left| \partial_{z}^{p} u^{s}_{0}(z)\right|  \leq c_{2}
\langle z \rangle^{-k-p+1}, \forall z \geq 0, \quad 1 \leq p \leq m+4,
\end{cases}
\label{uv0}
\end{align}
for some constants $c_{1}, c_{2} >0$ and even integer $m$.

Now, we can state the main result as follows.
\begin{theorem}\label{e-u-s}
Assume the condition (H) holds and consider the system \eqref{3d prandtl constant outflow }. Let $m \geq 6$ be  an even integer , and   the  real numbers 
$k >1$ , $ -\frac{1}{2} < \nu < 0$.
Suppose that   the  initial shear flow
$u^{s}_{0}$   satisfies \eqref{uv0}, the initial data
$\tu_{0}=u_{0}-u_{0}^{s} \in H^{m+3}_{k+\nu}(\mathbb{R}_{+}^{3})$, and the compatibility conditions up to ${(m+3)}^{th}$ order. 
And $K(x, y)$ is supposed to satisfy that
$$\big\|K\big\|_{W^{m+1, \infty}(\mathbb{R}^2)} < \infty.$$
Then there exists a sufficiently small constant $\delta_{0}$,  such that if
\begin{align}
\left\| \tu_{0}\right\| _{H^{m+1}_{k+\nu}(\mathbb{R}_{+}^{3})}
\leq \delta_{0},
\label{uvdelta}
\end{align}
then  the initial-boundary value problem \eqref{3d prandtl constant outflow } admits a unique solution $(u, w)$ with
\begin{align}
 u-u^{s}  \in L^{\infty}\big( [0, T]; H^{m}_{k+\nu-\delta^{'}} (\mathbb{R}_{+}^{3}) \big),  \quad w \in 
L^{\infty}\big( [0, T]; H^{
\infty} (\mathbb{R}_{z, +}; H^{m-1}(\mathbb{R}^{2}_{xy}) \big),
\end{align}
where $\delta'>0$ satisfies  $\nu+\frac 12<\delta'<\nu+1$ and $k+\nu-\delta'>\frac 12$.

Moreover,  the classical solution to \eqref{3d prandtl constant outflow } is stable with respect to the initial data in the following sense: for any given two initial data

$$
u^1_{0}=u^{s}_{0}+\tu^{1}_{0}
$$
and
$$
u^2_{0}=u^{s}_{0}+\tu^{2}_{0},
$$
if $u^{s}_{0}$ satisfies \eqref{uv0} and $\tu^{1}_{0}, \tu^{2}_{0}$ satisfy \eqref{uvdelta},  the corresponding solutions $u^1, u^2$ of the 3D Prandtl system \eqref{3d prandtl constant outflow } satisfy
$$
\|u^1-u^2\|_{L^\infty([0, T]; H^{m-3}_{k+\nu-\delta'}(\mathbb{R}^3_+))} \le C\|u^1_{0}-u^2_{0} \|_{ H^{m+1}_{k +\nu}(\mathbb{R}^3_+)}, 
$$
where the constant $C$ depends on the norm of $u^{1}_{z}, u^{2}_{z}$ in $L^\infty([0, T]; H^m_{k+\nu-\delta'+1}(\mathbb{R}^3_+))$.
\end{theorem}

The rest of the paper is organized as follows. 
In Section \ref{ss2}, we explain the difficulties and outline our approach to show
the long time  well-posedness for the 3D Prandtl equations. 
In Section \ref{s3}, we  investigate the approximate solutions to \eqref{3d prandtl constant outflow } by a parabolic
regularization.
In Section \ref{s4},  we are devoted to improve the results of Section 3 by formal transformation.  In Sections \ref{s5}-\ref{s6},  we prove finally  Theorem \ref{e-u-s}. The  existence and uniqueness of the  solutions will be established in Sections \ref{s5} and \ref{s6}, respectively.

\section{Preliminaries}\label{ss2}
\subsection{Difficulties and outline of our approach}
In establishing well-posedness theories, the degeneracy in viscous dissipation coupled with the loss
of derivative in the nonlocal term presents the main challenge.  Therefore, the primary obstacles to extending the lifespan of solutions  are the terms $w\partial_{z}u$,  $w\partial_{z}v$ in the equation \eqref{3d prandtl}.  As
$$
w(t, x, y, z)=-\int^{z}_{0}\partial_{x}{u}(t, x, y, \tilde{z}) d\tilde{z}-\int^{z}_{0}\partial_{y}{v}(t, x, y, \tilde{z}) d\tilde{z},
$$
this term loses one tangential derivative, so the standard estimates
cannot apply. For  the  2D case, by establishing a cancellation mechanism between equation \eqref{3d prandtl} and its corresponding vorticity equation, these bad terms can be dealt with under the Oleinik's monotonicity assumption, see \cite{Alexandre-Wang-Xu-Yang-2015,Masmoudi-Wong-2015,Xu-Zhang-2017}. But in the 3D case,  when we apply $\partial_{z}$ to both sides of the equation $\eqref{3d prandtl} _{1}$ and $\eqref{3d prandtl}_{2}$,  some new terms (underlined terms) appear as we shall see in the following equation $\big((U, V)=(1,1)\big)$,
\begin{equation*}
	\begin{cases} \partial_t (\partial_{z}u)  + u \partial_x (\partial_{z}u)  +v\partial_y (\partial_{z}u) + w  \partial_z (\partial_{z}u) 
		+ \underline{(\partial_{z}v)\partial_{y}u} -  \underline{(\partial_{z}u)\partial_{y}v	}= \partial^2_z (\partial_{z}u)   , \\
		\partial_t (\partial_{z}v)  + u \partial_x (\partial_{z}v)  +v \partial_y (\partial_{z}v) + w  \partial_z   (\partial_{z}v)  + \underline{(\partial_{z}u)\partial_{x}v} -  \underline{(\partial_{z}v)\partial_{x}u }= \partial^2_z (\partial_{z}v) .
	\end{cases}
\end{equation*}
In addition, the appearance of secondary flow in the 3D boundary layer equations implies that the monotonicity assumption is insufficient to secure well-posedness for the Prandtl equations in Sobolev spaces.

To overcome this difficulty,   Liu, Yang and Wang  \cite{Liu-Wang-Yang-1-2017} constructed a solution of the three-dimensional Prandtl equations \eqref{3d prandtl constant outflow } with a special structure
$$(u(t,x,y,z), K( x,y)u(t,x,y,z),w(t,x,y,z)),$$
which implies that the original equation is reduced to a new equation involving only $u$ and $w$. Since under the Crocco's transform $w$  is hidden,  the local existence of  the 3D Prandtl equation is naturally obtained. 
But in this paper, we give an accurate estimate for the new nonlinear terms to overcome the loss of $xy$-derivative  under a special structure assumption by the energy method directly.

Moreover, the lack of high-order boundary conditions prevents us from using the integration by part in the $z$ variable. However, we derive a reconstruction argument of the boundary reduction in the three dimensional case for the higher-order boundary conditions which  can  help us fix this
technical difficulty. Thanks to the properties of shear flow described in Lemma 2.1 of \cite{Xu-Zhang-2017}, we can establish the long time existence of solutions. More precisely,
we  will construct the solution $(u, w)=(u^s+\tu, \tw)$ of the
Prandtl equation \eqref{3d prandtl} as a small perturbation of a monotonic shear flow $u^s$ firstly.
Then the  following  idea of \cite{Xu-Zhang-2017}, 
dividing the  equations involving higher-order terms $\partial_{xy}^{n}u$  by $\partial_{z}u$, and taking $\partial_{z}$ on the resulting equation, we can get the following formal transformations of system \eqref{3d prandtl constant outflow } after regularizing 
$$
\partial_{t}g^{n}+(u^{s}+u)\partial_{x}g^{n}
+K(u^s+u)\partial_{y}g^{n}
-\partial_{z}^{2}g^{n}
-\epsilon\partial_{x}^{2}g^{n}
-\epsilon\partial_{y}^{2}g^{n}= \text{other terms}+ \text{terms containing}~K, 
$$
with a new linearly-good unknown
$$g^{n}=\bigg(\frac{\partial_{xy}^{n} \tu}{u_{z}^{s}+\tu_{z}}\bigg)_{z}, $$
which helps us cancel out the bad terms  directly in the two-dimensional case $(K=0)$. But in the three-dimensional space, those bad terms such as $w\partial_{z}u$    after the cancellation will lead to  producing some new higher-order terms of $\partial_{x}u$ through the divergence free condition 
$\partial_{x}u+\partial_{y}(Ku)+\partial_{z}w=0$ in the above equation .
This introduces additional complexity in the 3D setting. Finally, the existence of the 3D Prandtl boundary
layer equations is obtained  via the  closeness of  a priori estimate for the
approximate solutions.


\subsection{Analysis of shear flow}
To understand the problem \eqref{3d prandtl constant outflow },   we consider the  initial data $u_{0}$ around a shear flow, i.e.,
$$ u_{0}(x,y,z)=u^{s}_{0}(z)+\tilde{u}_{0}( x, y, z).$$
Let $u^s(t,z)$  be smooth solution of the
heat equation: 
\begin{equation}\label{heat}\begin{cases}
		\partial_t {u^s}-\partial_z^2{u^s}=0,\\
		u^s|_{z=0}=0,\quad \lim\limits_{z\rightarrow+\infty} u^s=1,\\
		u^s|_{t=0}=u_{0}^s(z),
\end{cases}\end{equation}
with $u^s-1$ rapidly tending to $0$ when $z\rightarrow+\infty$. It is straightforward to check that the shear velocity profile $u^s(t,z)$ satisfies the problem \eqref{3d prandtl constant outflow }. Furthermore,  denote 
 \begin{align*}
	u(t, x, y, z) = u^s(t, z) + \tilde{u}(t, x, y,z ),\,\, w(t, x, y, z)=\tilde w(t, x, y, z),
\end{align*} 
then the equation \eqref{3d prandtl constant outflow } can be written as
\begin{equation}\label{shear-prandtl}
	\begin{cases} \partial_t\tilde{u} + (u^s + \tilde{u}) \partial_x\tilde{u} +K(u^s + \tilde{u}) \partial_y\tilde{u}+ \tilde{w} \partial_z(u^s +  \tilde{u})
		= \partial^2_z\tilde{u}, \\
		\partial_x\tilde{u} +\partial_y\big(K(u^{s}+\tu)\big) +\partial_z\tilde{w} =0, \\
		(\tilde{u}, \tilde{w})|_{z=0} =(0, 0) , \ \lim\limits_{z\rightarrow +\infty} \tilde{u} = 0, \\
		\tilde{u}|_{t=0} =\tilde{u}_0 (x,y, z) .
	\end{cases}
\end{equation}
The shear flow $u^{s} (t, z)$ has the following profile. 
\begin{proposition} \label{shear-profile}
	Assume that the initial date $u^s_0$ satisfy conditions \eqref{uv0} , then for any $T>0$, there exist constants $\tilde{c}_1, \tilde{c}_2, \tilde{c}_3>0$ such that the solution $u^s(t,z)$ of the initial boundary value problem \eqref{heat} satisfies
	\begin{align} 
		\begin{cases}
			\tilde{c}_1\langle z \rangle^{-k}\leq \partial_{z}u^s(t,z) \leq  \tilde{c}_2 \langle z \rangle^{-k}, ~~
			\forall\,(t, z)\in [0, T]\times \bar{\mathbb{R}}_+,\\
			|\partial_z^p u^s(t,z)| \le \tilde{c}_3 \langle z \rangle^{-k-p+1},\,\, \forall\,\,(t, z)\in [0, T]\times \bar{\mathbb{R}}_+,\,\, 1\le p\le m+4,
		\end{cases}
	\end{align}
	where $\tilde{c}_1, \tilde{c}_2, \tilde{c}_3$ depend on $T$.
\end{proposition}
Such a proposition can be found in \cite{Xu-Zhang-2017,Qin-Dong-2021}, so we omit some details here.
\begin{remark}
From this proposition, we know that the decay of the gradient of shear flow $u^s$ is
the same as that of the gradient of initial data  $u^s_0$.  The gradient of shear flow
is of polynomial decay when the gradient of initial data of shear flow is the polynomial
decay.
\end{remark}

In this position, we  introduce the precise version of the compatibility condition for the
nonlinear system \eqref{shear-prandtl}.

\begin{proposition}\label{uv t=0}
	Let $m \geq 6$  be  an even integer, and assume that $\tilde{u}$ is a smooth solution of the system \eqref{shear-prandtl}, then the initial data $\tilde{u}_0$ have to satisfy the following compatibility conditions up to  $(m+2)^{th}$ order:
	\begin{equation}
		\begin{cases}
	&\tilde{u}_0|_{z=0}=0, \quad\,\partial^2_z \tilde{u}_0 |_{z=0}=0, \\
	&\partial^4_z \tilde{u}_0 |_{z=0}=2 \big(\partial_{z}(u^{s}_{0}+\tilde{u}_{0})\partial_{z}\partial_{x}\tilde{u}_{0}\big)\big|_{z=0}
	+2\big(\partial_{z} (K(u^{s}_{0}+\tilde{u}_{0}))\partial_{z}\partial_{y}\tilde{u}_{0}\big)\big|_{z=0}\\
	&\qquad    \quad \quad ~~~
	- \big(\partial_{z}(u^{s}_{0}+\tilde{u}_{0})\partial_{z}\partial_{x}\tilde{u}_{0}\big)\big|_{z=0}
	- \big(\partial_{z}(u^{s}_{0}+\tilde{u}_{0})\partial_{z}\partial_{y}(K(u^s_{0}+\tilde{u}_{0}))\big)\big|_{z=0}
~,
\end{cases}
		\label{uv4 t=0}
	\end{equation}
		and  
		\begin{align}
			\begin{split}
				\partial^{2(p+1)}_z \tilde{u}_{0}\big|_{z=0}
				&=\sum^p_{q=2}\sum_{( \beta, \gamma)\in \Lambda_{q}}C_{K, p,  \beta, \gamma}\prod\limits_{i=1}^{q_{1}}  \partial^{\beta}\partial_{z}\Big( u^s_{0} + \tilde{u}_{0} \Big) \Big|_{z=0} \\
				&\qquad \qquad \qquad \qquad \qquad \times 
				\prod\limits_{j=1}^{q_{2}}  \partial^{\gamma}\partial_{z}\Big( K\big(u^s_{0} + \tilde{u}_{0} \big) \Big) \Big|_{z=0}\, ,
			\end{split}
			\label{uv2p t=0}
		\end{align}
		for $2 \leq p\leq  \frac{m}{2}$, where
		\begin{align}
			\begin{split}	
				\Lambda_{q}=&\bigg\{
				\beta=(\beta_{x}, \beta_{y}, \beta_{z}) =(\beta_x^{1}, \cdots, \beta_x^{q_{1}}; \beta_y^{1}, \cdots, \beta_y^{q_{1}}; \beta_z^{1}, \cdots, \beta_z^{q_{1}})\in \mathbb{N}^{q_{1}}\times\,\mathbb{N}^{q_{1}}\times\,\mathbb{N}^{q_{1}};\\
				&\quad\gamma=(\gamma_{x}, \gamma_{y}, \gamma_{z}) =(\gamma_x^{1}, \cdots, \gamma_x^{q_{2}}; \gamma_y^{1}, \cdots, \gamma_y^{q_{2}}; \gamma_z^{1}, \cdots, \gamma_z^{q_{2}})\in \mathbb{N}^{q_{2}}\times\,\mathbb{N}^{q_{2}}\times\,\mathbb{N}^{q_{2}};\\
				&\quad \beta^i+\gamma^j \leq 2p-1, ~ 1\leq i \leq q_{1}, ~ 1\leq j \leq q_{2}, ~ q_{1}+q_{2}=q; \\
				&	 \qquad \sum^{q_{1}}_{i=1}\big\{3(\beta^i_{x}+\beta^i_{y}) + \beta_z^{i}\big\}
				+ \sum^{q_{2}}_{j=1}\big\{3(\gamma^j_{x}+\gamma^j_{y}) + \gamma_z^{j}\big\}= 2p +1;\\
				&\qquad \quad
				0<\sum^{q_{1}}_{i=1}(\beta^i_{x}+\beta^i_{y})  
				+ \sum^{q_{2}}_{j=1}(\gamma^j_{x}+\gamma^j_{y}) = p -1; \quad
				\sum^{q_{1}}_{i=1} \beta_z^{i} 
				+ \sum^{q_{2}}_{j=1} \gamma_z^{j}= 2p-2 \bigg\}.
			\end{split}
			\label{Lambda}
		\end{align}
	\end{proposition}

By referring to the  method   (see Appendix B) of Proposition \ref{boundary reduction}, one can easily deduce the  proposition.  In addition,  the above Proposition implies also the following result.

\begin{corollary}\label{c-uv t=0}
	Let $m \geq 6$  be  an even integer, and assume that $\tilde{u}$ is a smooth solution of the system \eqref{shear-prandtl}, then the initial data $\tilde{u}_0$ have to satisfy the following compatibility conditions up to  $(m+2)^{th}$ order:
	\begin{equation}
		\begin{cases}
			&\tilde{u}_0|_{z=0}=0, \quad\,\partial^2_z \tilde{u}_0 |_{z=0}=0, \\
			&\partial^4_z \tilde{u}_0 |_{z=0}=2 \big(\partial_{z}(u^{s}_{0}+\tilde{u}_{0})\partial_{z}\partial_{x}\tilde{u}_{0}\big)\big|_{z=0}
			+2\big(K\partial_{z} (u^{s}_{0}+\tilde{u}_{0})\partial_{z}\partial_{y}\tilde{u}_{0}\big)\big|_{z=0}\\
			&\qquad    \quad \quad ~~~
			- \big(\partial_{z}(u^{s}_{0}+\tilde{u}_{0})\partial_{z}\partial_{x}\tilde{u}_{0}\big)\big|_{z=0}
			- \big(K\partial_{z}(u^{s}_{0}+\tilde{u}_{0})\partial_{z}\partial_{y}\tilde{u}_{0}\big)\big|_{z=0}
			 \\
			 &\qquad    \quad \quad ~~~
			 -\partial_{y}K\big(\partial_{z}(u^s_{0}+\tilde{u}_{0})\big)^2 |_{z=0}~,
		\end{cases}
		\label{c-uv4 t=0}
	\end{equation}
and  
\begin{align}
	\begin{split}
		\partial^{2(p+1)}_z \tilde{u}_{0}\big|_{z=0}
		&=\sum^p_{q=2}\sum^{q-1}_{h=0}\partial_{xy}^h K\sum_{( \beta, \gamma)\in \Lambda_{q}}C_{K, p,  \beta, \gamma}\prod\limits_{i=1}^{q} \partial_{x}^{\alpha^{i}}\partial_{y}^{\beta^{i}} \partial^{\gamma^i+1}_{z}\Big( u^s_{0} + \tilde{u}_{0} \Big) \Big|_{z=0} \, ,
	\end{split}
	\label{c-uv2p t=0}
\end{align}
for $2 \leq p\leq  \frac{m}{2}$, where
\begin{align}
	\begin{split}	
		\Lambda_{q}=&\bigg\{
		\beta=(\alpha, \beta, \gamma) =(\alpha^{1}, \cdots, \alpha^{q}; \beta^{1}, \cdots, \beta^{q}; \gamma^{1}, \cdots, \gamma^{q})\in \mathbb{N}^{q}\times\,\mathbb{N}^{q}\times\,\mathbb{N}^{q};\\
		&\quad \alpha^i+\beta^i+\gamma^j \leq 2p-1, ~ 1\leq i \leq q; 	 \qquad \sum^{q}_{i=1}\big\{3(\alpha^i+\beta^i) + \gamma^{i}\big\}
		= 2p +1;\\
		&\qquad \quad
		0<\sum^{q}_{i=1}(\alpha^i+\beta^i)  
		\leq p -1; \quad
		\sum^{q}_{i=1} \gamma^{i} 
		\leq 2p-2 \bigg\}.
	\end{split}
	\label{c-Lambda}
\end{align}
\end{corollary}

 \section{The approximate solutions}\label{s3}

In  this section, in order to prove the  existence of solution to initial-boundary value problem \eqref{shear-prandtl},  we consider now the regularized equations for any  $0 <\epsilon \leq 1$, 

\begin{equation}\label{regular-shear-prandtl}
	\begin{cases} \partial_t\tilde{u}^\epsilon + (u^s + \tilde{u}^\epsilon) \partial_x\tilde{u}^\epsilon +K(u^s + \tilde{u}^\epsilon) \partial_y\tilde{u}^\epsilon+ \tilde{w}^\epsilon \partial_z(u^s +  \tilde{u}^\epsilon)
		= \partial^2_z\tilde{u}^\epsilon+\epsilon \partial^2_x\tilde{u}^\epsilon+\epsilon \partial^2_y\tilde{u}^\epsilon, \\
	\partial_x\tilde{u}^\epsilon +\partial_y\big(K(u^{s}+\tu^\epsilon)\big) +\partial_z\tilde{w}^\epsilon =0, \\
		(\tilde{u}^\epsilon,   \tilde{w}^\epsilon)|_{z=0} =(0, 0) , \ \lim\limits_{z\rightarrow +\infty} \tilde{u}^\epsilon =0, \\
		\tilde{u}^\epsilon|_{t=0} =\tilde{u}^\epsilon_{0}=\tilde{u}_{0}+\epsilon\mu^{\epsilon} ,
	\end{cases}
\end{equation}
where $\epsilon\mu^{\epsilon}$ is a corrector and $\tilde{u}_{0}+\epsilon\mu^{\epsilon}$ satisfies the compatibility condition up to $(m+2)^{th}$ order for regularized system \eqref{regular-shear-prandtl}.

Now we give accurate edition of  the  compatibility condition for the nonlinear  regularized system \eqref{regular-shear-prandtl},   and the reduction properties of
boundary data ,  which is used
to control the highest-order derivatives for the  key integral.
\begin{proposition}\label{boundary reduction}
	Let $m \geq 6$ be an even integer, $k > 1$, $0 < \ell < \frac{1}{2}$, $k+\ell > \frac{3}{2}$, $\epsilon \in (0, 1]$, and 
	assume that $\tilde{u}^\epsilon_{0}$ satisfies the compatibility conditions \eqref{uv4 t=0} and \eqref{uv2p t=0} for the system \eqref{shear-prandtl}.
	If $\tue \in
	L^\infty ([0, T]; H^{m+3}_{k +\ell}(\mathbb{R}^3_+))\cap Lip([0, T];
	H^{m+1}_{k +\ell}(\mathbb{R}^3_+))$ and $(\tilde{u}^\epsilon, \tilde{w}^\epsilon)$ solves \eqref{regular-shear-prandtl}, then, we have at boundary $z=0$, 
	\begin{equation}
		\begin{cases}
	&\tilde{u}^\epsilon|_{z=0}=0, \qquad \partial^2_z \tilde{u}^\epsilon |_{z=0}=0, \\
	&\partial^4_z \tilde{u}^\epsilon|_{z=0}=2\partial_{z}(u^s + \tilde{u}^\epsilon)\partial_{x}\partial_{z}\tue |_{z=0}+2\partial_{z}\big(K(u^s + \tilde{u}^\epsilon)\big)\partial_{y}\partial_{z}\tue |_{z=0}\\
	&\quad \quad  \quad \quad  ~~~-\partial_{x}\partial_{z}\tilde{{u}}^{\epsilon} \partial_{z}(u^s+\tue) |_{z=0}-\partial_{y}\partial_{z}(K(u^s+\tilde{{u}}^{\epsilon})) \partial_{z}(u^s+\tue) |_{z=0}
\, ,
\end{cases}
	\end{equation}
	and  for $2 \leq p\leq  \frac{m}{2}$,
	\begin{align}
		\begin{split}
			\partial^{2(p+1)}_z (\tilde{u}^\epsilon, \tilde{v}^\epsilon)|_{z=0}
			&=\sum^p_{q=2}\sum^{q-1}_{l=0}\epsilon^{l}
			\sum_{( \beta, \gamma)\in \Lambda_{q, l}}C_{K, p, l, \beta, \gamma}\prod\limits_{i=1}^{q_{1}}  \partial^{\beta}\partial_{z}\Big( u^s + \tilde{u}^\epsilon  \Big)\Big|_{z=0} \\
			& \qquad \qquad \qquad \times 
			\prod\limits_{j=1}^{q_{2}}  \partial^{\gamma}\partial_{z}\Big( K\big(u^s + \tilde{u}^\epsilon \big) \Big)\Big|_{z=0} \, ,
		\end{split}
		\label{z2p+1 uv}
	\end{align}
	where
	\begin{align}
		\begin{split}	
			\Lambda_{q, l}=&\bigg\{
			\beta=(\beta_{x}, \beta_{y}, \beta_{z}) =(\beta_x^{1}, \cdots, \beta_x^{q_{1}}; \beta_y^{1}, \cdots, \beta_y^{q_{1}}; \beta_z^{1}, \cdots, \beta_z^{q_{1}})\in \mathbb{N}^{q_{1}}\times\,\mathbb{N}^{q_{1}}\times\,\mathbb{N}^{q_{1}};\\
			&\quad\gamma=(\gamma_{x}, \gamma_{y}, \gamma_{z}) =(\gamma_x^{1}, \cdots, \gamma_x^{q_{2}}; \gamma_y^{1}, \cdots, \gamma_y^{q_{2}}; \gamma_z^{1}, \cdots, \gamma_z^{q_{2}})\in \mathbb{N}^{q_{2}}\times\,\mathbb{N}^{q_{2}}\times\,\mathbb{N}^{q_{2}};\\
			&\quad \beta^i+\gamma^j \leq 2p-1, ~ 1\leq i \leq q_{1}, ~ 1\leq j \leq q_{2}, ~ q_{1}+q_{2}=q; \\
			&	 \qquad \sum^{q_{1}}_{i=1}\big\{3(\beta^i_{x}+\beta^i_{y}) + \beta_z^{i}\big\} 
			+ \sum^{q_{2}}_{j=1}\big\{3(\gamma^j_{x}+\gamma^j_{y}) + \gamma_z^{j}\big\}= 2p+ 4l +1;\\
			&\qquad \quad
			0<\sum^{q_{1}}_{i=1}(\beta^i_{x}+\beta^i_{y})  
			+ \sum^{q_{2}}_{j=1}(\gamma^j_{x}+\gamma^j_{y}) = p+ 2l -1; \quad
			\sum^{q_{1}}_{i=1} \beta_z^{i} 
			+ \sum^{q_{2}}_{j=1} \gamma_z^{j}= 2p-2l-2\bigg\}.
		\end{split}
		\label{LLambda}
	\end{align}
\end{proposition}
The proof of the above Proposition implies also the following result.

\begin{corollary}\label{c-boundary reduction}
	Let $m \geq 6$ be an even integer, $k > 1$, $0 < \ell < \frac{1}{2}$, $k+\ell > \frac{3}{2}$, $\epsilon \in (0, 1]$, and 
	assume that $\tilde{u}^\epsilon_{0}$ satisfies the compatibility conditions \eqref{uv4 t=0} and \eqref{uv2p t=0} for the system \eqref{shear-prandtl}.
	If $\tue \in
	L^\infty ([0, T]; H^{m+3}_{k +\ell}(\mathbb{R}^3_+))\cap Lip([0, T];
	H^{m+1}_{k +\ell}(\mathbb{R}^3_+))$ and $(\tilde{u}^\epsilon,   \tilde{w}^\epsilon)$ solves \eqref{regular-shear-prandtl}, then, we have at boundary $z=0$, 
	\begin{equation}
		\begin{cases}
			&\tilde{u}^\epsilon|_{z=0}=0, \qquad \partial^2_z \tilde{u}^\epsilon |_{z=0}=0, \\
			&\partial^4_z \tilde{u}^\epsilon|_{z=0}=2\partial_{z}(u^s + \tilde{u}^\epsilon)\partial_{x}\partial_{z}\tue |_{z=0}+2K\partial_{z}(u^s + \tilde{u}^\epsilon)\partial_{y}\partial_{z}\tue |_{z=0}\\
			&\quad \quad  \quad \quad  ~~~-\partial_{x}\partial_{z}\tilde{{u}}^{\epsilon} \partial_{z}(u^s+\tue) |_{z=0}-K\partial_{y}\partial_{z}\tilde{{u}}^{\epsilon} \partial_{z}(u^s+\tue) |_{z=0}\\
			&\quad \quad  \quad \quad  ~~~- \partial_{y}K\big(\partial_{z}(u^s+\tue)\big)^2 |_{z=0}~,
		\end{cases}
	\end{equation}
	and  for $2 \leq p\leq  \frac{m}{2}$,
\begin{align}
    \begin{split}
		\partial^{2(p+1)}_z \tilde{u}^\epsilon|_{z=0}
&=\sum^p_{q=2}\sum^{q-1}_{l=0}\epsilon^{l}
\sum^{2q-1}_{h=0}\partial_{xy}^h K
\sum_{( \beta, \gamma)\in \Lambda_{q, l}}C_{K, p, l, \beta, \gamma}\prod\limits_{i=1}^{q} \partial_{x}^{\alpha^{i}}\partial_{y}^{\beta^{i}} \partial^{\gamma^{i}+1}_{z}\Big( u^s + \tilde{u}^\epsilon \Big) \Big|_{z=0}  \, ,
    \end{split}
    \label{c-z2p+1 uv}
	\end{align}
	where
	\begin{align}
	\begin{split}	
		\Lambda_{q, l}=&\bigg\{
\beta=(\alpha, \beta, \gamma) =(\alpha^{1}, \cdots, \alpha^{q}; \beta^{1}, \cdots, \beta^{q}; \gamma^{1}, \cdots, \gamma^{q})\in \mathbb{N}^{q}\times\,\mathbb{N}^{q}\times\,\mathbb{N}^{q};\\
&\quad \alpha^i+\beta^i+\gamma^j \leq 2p-1, ~ 1\leq i \leq q; 	 \qquad \sum^{q}_{i=1}\big\{3(\alpha^i+\beta^i) + \gamma^{i}\big\}
= 2p +4l+1;\\
&\qquad \quad
0<\sum^{q}_{i=1}(\alpha^i+\beta^i)  
\leq p+2l -1; \quad
\sum^{q}_{i=1} \gamma^{i} 
\leq 2p-2l-2 \bigg\}.
	\end{split}
\label{cc-Lambda}
\end{align}
\end{corollary}
\begin{remark}
	The condition
	$$0<\sum^{q}_{i=1}(\alpha^i+\beta^i)   $$
	imply that, for each term of \eqref{c-z2p+1 uv}, there are at last
	one factor like
	$$  \partial_{x}^{\alpha^{i}}\partial_{y}^{\beta^{i}} \partial^{\gamma^i+1}_{z} \tilde{u}^\epsilon \Big|_{z=0}.$$
\end{remark}

With the above proposition,  we can draw a corollary, which helps us to
understand why we add a corrector $\epsilon\mu^{\epsilon}$.
\begin{corollary}\label{corrector reduction}
Under the hypotheses of Proposition	\ref{boundary reduction}, assume also
that $\partial_{z}\tilde{u}_{0} \in 	H^{m+2}_{k +\ell'}(\mathbb{R}^3_+)$, then there exists $\epsilon_{0}$, and   $\mu_\epsilon \in H^{m+3}_{k +\ell'-1}(\mathbb{R}^{3}_+)$, for some $\frac{1}{2} < \ell' < \ell+\frac{1}{2}$ and any $0< \epsilon \leq  \epsilon_{0}$, such that  
$\tilde{u}_{0}+\epsilon\mu^{\epsilon}$ satisfies the compatibility condition up to $(m+2)^{th}$ order for regularized system \eqref{regular-shear-prandtl}. 
 Moreover, for any $m\leq \tilde m\leq m+2$, we have 
\begin{align*}
\big\|\partial_z\tilde{u}^\epsilon_{0}\big\|_{H^{\tilde m}_{k+\ell'}(\mathbb{R}^{3}_{+})}\leq \frac {3}{2} \big\|
\partial_z\tilde{u}_{0}\big\|_{H^{\tilde m}_{k+\ell'}(\mathbb{R}^3_+)},
\end{align*}
and
\begin{align*}
\lim_{\epsilon \rightarrow 0}\big\|\partial_z\tilde{u}^\epsilon_{0}-\partial_z\tilde{u}_{0}\big\|_{H^{\tilde m}_{k+\ell'}(\mathbb{R}^{3}_+)}=0.
\end{align*}
\end{corollary}

We are now going to prove the existence of approximate solutions to the original regularized  system \eqref{regular-shear-prandtl} by utilizing the vorticity $\varphi^\epsilon=\partial_{z}\tue$, which is reformulated 
  as the following form, for  any $0 <\epsilon \leq 1$, 
\begin{equation}\label{vv-regular-shear-prandtl}
	\begin{cases} \partial_t\tilde{\varphi}^\epsilon + (u^s + \tilde{u}^\epsilon) \partial_x\tilde{\varphi}^\epsilon +K(u^s + \tilde{u}^\epsilon) \partial_y\tilde{\varphi}^\epsilon+ \tilde{w}^\epsilon \partial_z(u^s_{z} +  \tilde{\varphi}^\epsilon)
	+\partial_{y}K (u^s+\tue)\partial_{z}(u^s+\tue)
	\\
	\qquad = \partial^2_z\tilde{\varphi}^\epsilon +\epsilon \partial^2_x\tilde{\varphi}^\epsilon+\epsilon \partial^2_y\tilde{\varphi}^\epsilon , \\
		\partial_{z}\tilde{\varphi}^\epsilon|_{z=0} =0 , \\
		\tvae|_{t=0} =\tilde{\varphi}_{0}+\epsilon\partial_{z}\mu^{\epsilon},
	\end{cases}
\end{equation}
where
\begin{align*}
&\tue(t, x, y, z)=-\int^{+\infty}_{z}\tilde{\varphi}^{\epsilon}(t, x, y, \tilde{z}) d\tilde{z}, \\
& \twe(t, x, y, z)=-\int^{z}_{0}\partial_{x}\tilde{u}^{\epsilon}(t, x, y, \tilde{z}) d\tilde{z}-\int^{z}_{0}\partial_{y}(K(u^{s}+\tilde{u}^{\epsilon}))(t, x, y, \tilde{z}) d\tilde{z}.
\end{align*}

With the above preparations, as in \cite{Xu-Zhang-2017,Masmoudi-Wong-2015,Fan-Ruan-Yang-2021}, we derive there exits a life existence time $T^\epsilon$  such that if $\|\tilde{\varphi}_{0}\|_{H^{m+2}_{k+\ell}(\mathbb{R}^3_+)}$ owns a 
 bound, then  system \eqref{vv-regular-shear-prandtl} admits a unique solution. Specifically, we have the following proposition for the existence of approximate solutions.
\begin{theorem}\label{epsilon-existence}
Assume the condition (H) holds. 	Let $m \geq 6$ be an even integer, $k > 1$, $0 \leq \ell < \frac{1}{2}$, $k+\ell > \frac{3}{2}$.
	Assume that $\partial_{z}\tilde{u}_{0}^\epsilon$ belongs to $H^{m+2}_{k+\ell}(\mathbb{R}^3_+)$
	and satisfies the compatibility conditions up to order $m+2$ for \eqref{shear-prandtl}.
Also, we  assume  that the shear flow $u^s(t,z)$, for $0 \leq p \leq m+2$ and  $(t, z)\in [0, T_1]\times \mathbb{R}_+$,  satisfies
	\begin{align*}
\big|\partial_{z}^{p+1}u^s(t,z)\big| \leq C \langle z \rangle^{-k-p}.
	\end{align*}
And $K(x, y)$ is supposed to satisfy that
$$\big\|K\big\|_{W^{m+1, \infty}(\mathbb{R}^2)} < \infty.$$
Moreover, for any $0<\epsilon\leq \epsilon_{0}$ and $\tilde{\zeta} >0$, there exits $T^\epsilon > 0 $ depending only on $\epsilon$ and $\tilde{\zeta}$, such that if
\begin{align*}
\|\tilde{\varphi}_{0}\|_{H^{m+2}_{k+\ell}(\mathbb{R}^3_+)}\leq \tilde{\zeta},
\end{align*}
	then the initial boundary value problem \eqref{vv-regular-shear-prandtl} admits a unique solution
\begin{align*}
	\tvae \in L^\infty([0, T^\epsilon]; H^{m+2}_{k+\ell}(\mathbb{R}^3_+)) \, ,
\end{align*}
which satisfies
\begin{align*}
\|\tvae\|_{L^{\infty}([0, T^\epsilon]; H^{m}_{k+\ell}(\mathbb{R}^3_+))}
\leq \frac{4}{3}\|\tvae_{0}\|_{H^{m}_{k+\ell}(\mathbb{R}^3_+)} \leq
2\|\tilde{\varphi}_{0}\|_{H^{m}_{k+\ell}(\mathbb{R}^3_+)}.
\end{align*}
\end{theorem}

The proof of Theorem \ref{epsilon-existence} will be given in the following two subsections. More   specifically,  we will derive the a priori   estimate on $\partial^{\alpha
}\varphi$ for   $\alpha =(\alpha_1, \alpha_{2}, \alpha_3)$ satisfying $|\alpha| \leq s $ with  $\alpha_1+\alpha_{2} \leq m-1$ in the first subsection and  $|\alpha_{1}+\alpha_{2}| = s $ in the second  subsection.

\subsection{Weighted $L^2$ estimates on $\partial^{\alpha} \tvae$ with $|\alpha| \leq m, \alpha_1+\alpha_{2} \leq m-1$}

\begin{lemma}\label{E1}
Assume the condition (H) holds. Let $m \geq 6$ be an even integer, $k > 1$, $0 < \ell < \frac{1}{2}$, $k+\ell > \frac{3}{2}$.
Assume $\tvae$ is a  solution to the initial boundary value problem \eqref{vv-regular-shear-prandtl}
in $[0, T^\epsilon]$  and satisfies $\tvae \in L^\infty([0, T^\epsilon]; H^{m+2}_{k+\ell}(\mathbb{R}^3_+))$. And $K(x, y)$ is supposed to satisfy that
$$\big\|K\big\|_{W^{m+1, \infty}(\mathbb{R}^2)} < \infty.$$ Then, it holds that
\begin{align}\begin{split}
&\frac{d}{dt}\|\tvae\|^{2}_{H^{m, m-1}_{k+\ell}(\mathbb{R}_{+}^{3})}+\|\partial_{z}\tvae\|^{2}_{H^{m, m-1}_{k+\ell}(\mathbb{R}_{+}^{3})}+\epsilon\left( \|\partial_{x}\tvae\|^{2}_{H^{m, m-1}_{k+\ell}(\mathbb{R}_{+}^{3})}+\|\partial_{y}\tvae\|^{2}_{H^{m, m-1}_{k+\ell}(\mathbb{R}_{+}^{3})}\right) \\
& \leq C \left( \|\tvae\|^{2}_{H^{m}_{k+\ell}(\mathbb{R}_{+}^{3})}+\|\tvae\|^{m}_{H^{m}_{k+\ell}(\mathbb{R}_{+}^{3})}\right) . 
	\end{split}
	\label{3.1}
\end{align}
\end{lemma}
\begin{proof}
Applying the operator $\partial^\alpha =\partial^{\alpha_1}_{x}\partial^{\alpha_2}_{y}\partial^{\alpha_3}_{z}$
 for $\alpha =(\alpha_1, \alpha_{2}, \alpha_3)$ satisfying $|\alpha| \leq s,  \alpha_1+\alpha_{2} \leq m-1$ in the vorticity
equation $\eqref{vv-regular-shear-prandtl}_{1}$, we have
\begin{align}
\begin{split}
&\partial_t\partial^\alpha\tilde{\varphi}^\epsilon +  (u^s + \tilde{u}^\epsilon) \partial_x\partial^\alpha\tilde{\varphi}^\epsilon +K(u^s + \tilde{u}^\epsilon) \partial_y\partial^\alpha\tilde{\varphi}^\epsilon+ \tilde{w}^\epsilon \partial_z\partial^\alpha  \tilde{\varphi}^\epsilon -\partial^2_z\partial^\alpha\tilde{\varphi}^\epsilon -\epsilon \partial^2_x\partial^\alpha\tilde{\varphi}^\epsilon-\epsilon \partial^2_y\partial^\alpha\tilde{\varphi}^\epsilon 
\\
&=-\sum\limits_{ \beta \leq \alpha, 1\leq |\beta|} C^\beta_\alpha \partial^{\beta}(u^s + \tilde{u}^\epsilon) \partial^{\alpha - \beta}\partial_x \tilde{\varphi}^\epsilon
-\sum\limits_{ \beta \leq \alpha, 1\leq |\beta|} C^\beta_\alpha \partial^{\beta}\big(K(u^s + \tilde{u}^\epsilon) \big)\partial^{\alpha - \beta}\partial_y \tilde{\varphi}^\epsilon
-\sum\limits_{ \beta \leq \alpha, 1\leq |\beta|} C^\beta_\alpha \partial^{\beta}\tilde{w}^\epsilon \partial^{\alpha - \beta}\partial_z \tilde{\varphi}^\epsilon
\\
&\quad-\partial^{\alpha}(\tilde{w}^\epsilon u^s_{zz}) -\partial^{\alpha}\big(\partial_{y}K (u^s+\tue)\partial_{z}(u^s+\tue)\big) . \label{3.2}
\end{split}
\end{align}
Multiplying \eqref{3.2} by $ \langle z \rangle^{2(k+\ell+{\alpha_3})} \partial^{\alpha} \tilde{\varphi}^\epsilon $, and then integrating over $\mathbb{R}^{3}_{+}$, we
have
\begin{align}
\begin{split}
\frac{1}{2}\frac{d}{dt}\|\langle z \rangle^{k+\ell+{\alpha_3}} \partial^{\alpha}  \tvae\|_{L^2(\mathbb{R}^{3}_{+})}
&=-\int_{{\mathbb{R}^3_+}} \langle z \rangle^{2(k+\ell+{\alpha_3})} 
\big((u^s + \tilde{u}^\epsilon) \partial_x\partial^\alpha\tilde{\varphi}^\epsilon +K(u^s + \tilde{u}^\epsilon) \partial_y\partial^\alpha\tilde{\varphi}^\epsilon+ \tilde{w}^\epsilon \partial_z\partial^\alpha  \tilde{\varphi}^\epsilon\big)\partial^{\alpha} \tilde{\varphi}^\epsilon
\\
&\quad+\int_{{\mathbb{R}^3_+}} \langle z \rangle^{2(k+\ell+{\alpha_3})} 
\partial_{z}^{2}\partial^{\alpha} \tilde{\varphi}^\epsilon\partial^{\alpha} \tilde{\varphi}^\epsilon
+\epsilon\int_{{\mathbb{R}^3_+}} \langle z \rangle^{2(k+\ell+{\alpha_3})} 
\big(\partial^2_x\partial^\alpha\tilde{\varphi}^\epsilon
+ \partial^2_y\partial^\alpha\tilde{\varphi}^\epsilon \big)\partial^{\alpha} \tilde{\varphi}^\epsilon
\\
&
\quad-\sum\limits_{ \beta \leq \alpha, 1\leq |\beta|} C^\beta_\alpha \int_{{\mathbb{R}^3_+}} \langle z \rangle^{2(k+\ell+{\alpha_3})} 
\partial^{\beta}(u^s + \tilde{u}^\epsilon) \partial^{\alpha - \beta}\partial_x \tilde{\varphi}^\epsilon
\partial^{\alpha} \tilde{\varphi}^\epsilon
\\
&
\quad-\sum\limits_{ \beta \leq \alpha, 1\leq |\beta|} C^\beta_\alpha \int_{{\mathbb{R}^3_+}} \langle z \rangle^{2(k+\ell+{\alpha_3})} 
\partial^{\beta}\big(K(u^s + \tilde{u}^\epsilon)\big) \partial^{\alpha - \beta}\partial_y \tilde{\varphi}^\epsilon
\partial^{\alpha} \tilde{\varphi}^\epsilon
\\
&
\quad-\sum\limits_{ \beta \leq \alpha, 1\leq |\beta|} C^\beta_\alpha \int_{{\mathbb{R}^3_+}} \langle z \rangle^{2(k+\ell+{\alpha_3})} 
\partial^{\beta}\tilde{w}^\epsilon \partial^{\alpha - \beta}\partial_z \tilde{\varphi}^\epsilon
\partial^{\alpha} \tilde{\varphi}^\epsilon\\
&\quad- \int_{{\mathbb{R}^3_+}} \langle z \rangle^{2(k+\ell+{\alpha_3})} 
-\partial^{\alpha}(\tilde{w}^\epsilon u^s_{zz}) 
\partial^{\alpha} \tilde{\varphi}^\epsilon \\
&\quad- \int_{{\mathbb{R}^3_+}} \langle z \rangle^{2(k+\ell+{\alpha_3})} 
\partial^{\alpha}\big(\partial_{y}K (u^s+\tue)\partial_{z}(u^s+\tue)\big)
\partial^{\alpha} \tilde{\varphi}^\epsilon \\
&= \sum_{i=1}^{8} I_{i}. \label{3.3}
\end{split}
\end{align}
Now, we estimate the right-hand side of \eqref{3.3} term by term as follows.

\textbf{Dealing with $I_1$  term :}
Integrating by parts in the $x$-variable, $y$-variable and $z$-variable, respectively,  
we have
\begin{align*}
I_{1}&=\frac{1}{2}\int_{{\mathbb{R}^3_+}}  \langle{z}\rangle^{2(k+\ell+{\alpha_3})} 
\big(\partial_{x}(u^s + \tilde{u}^\epsilon)  +\partial_{y}\big(K(u^s + \tilde{u}^\epsilon)\big) + \partial_{z}\tilde{w}^\epsilon\big)(\partial^\alpha  \tilde{\varphi}^\epsilon)^2 \\
&\quad + (k+\ell+{\alpha_3})\ir \langle z\rangle^{2(k+\ell+{\alpha_3})} (\langle z\rangle^{-1}\tilde{w}^\epsilon) (\partial^\alpha  \tilde{\varphi}^\epsilon)^2\\
&\leq C\|\langle z\rangle^{-1}w^\epsilon\|_{L^\infty(\mathbb{R}^{3}_{+})}\|\tvae\|_{H^{m}_{k+\ell}(\mathbb{R}^{3}_{+})}^{2}\\
&\leq C \|\tvae\|_{H^{m}_{{\frac{1 }{2}+\delta}}(\mathbb{R}^{3}_{+})}\|\tvae\|_{H^{m}_{k+\ell}(\mathbb{R}^{3}_{+})}^{2},
\end{align*}
where, in the last step, we have used the following fact  by applying the Hardy inequality \eqref{hardy2},  \eqref{infty2}, and the divergence-free condition $\eqref{regular-shear-prandtl}_{3}$
\begin{align*}
	\|\langle z\rangle^{-1}w^\epsilon\|_{L^{\infty}(\mathbb{R}^{3}_{+})}
	&\leq C\left( 
	\|\langle z\rangle^{-\frac{1 }{2}+\delta}\partial_{z}w^\epsilon\|_{L^{2}(\mathbb{R}_{+}^{3})}
	+\|\langle z\rangle^{-\frac{1 }{2}+\delta}\partial_{x}\partial_{z}w^\epsilon\|_{L^{2}(\mathbb{R}_{+}^{3})}
	+\|\langle z\rangle^{-\frac{1 }{2}+\delta}\partial_{y}\partial_{z}w^\epsilon\|_{L^{2}(\mathbb{R}_{+}^{3})} \right.\\
	&\quad\left.+\|\langle z\rangle^{-\frac{1 }{2}+\delta}\partial_{x}\partial_{y}\partial_{z}w^\epsilon\|_{L^{2}(\mathbb{R}_{+}^{3})}
	\right)\\
	&\leq C \|\partial_{z}\uve\|_{H^{|\alpha_{1}+\alpha_{2}|(\leq 2)}_{{\frac{1 }{2}+\delta}}(\mathbb{R}_{+}^{3})}
	\\
	&\leq C_{K}  \|\tvae\|_{H^{|\alpha_{1}+\alpha_{2}|(\leq 3)}_{{\frac{1 }{2}+\delta}}(\mathbb{R}_{+}^{3})} \\
&\leq C \|\tvae\|_{H^{m}_{{\frac{1 }{2}+\delta}}(\mathbb{R}^{3}_{+})}.
\end{align*}

\textbf{Dealing with $I_2$  term :}
For  $I_{2}$, using integration by parts again in the
$z$-variable, we have
\begin{align*}
I_{2}&=-\|\langle z \rangle^{k+\ell+{\alpha_3}} \partial_{z}\partial^{\alpha}  \tvae\|_{{L^2}(\mathbb{R}^{3}_{+})}^{2}
-(k+\ell+{\alpha_3})\ir \langle z\rangle^{2(k+\ell+{\alpha_3})-1} \partial_{z} \partial^\alpha  \tilde{\varphi}^\epsilon \partial^\alpha  \tilde{\varphi}^\epsilon\\
&\quad-\int_{\mathbb{R}^{2}} \partial_{z} \partial^\alpha  \tilde{\varphi}^\epsilon \partial^\alpha  \tilde{\varphi}^\epsilon\Big|_{z=0}\\
&=-\|\langle z \rangle^{k+\ell+{\alpha_3}} \partial_{z}\partial^{\alpha}  \tvae\|_{{L^2}(\mathbb{R}^{3}_{+})}^{2}+I_{2}^{1}+I_{2}^{2},
\end{align*}
where $I_{2}^{1}$ is controlled by using the Cauchy inequality
\begin{align*}
I_{2}^{1} \leq \frac{1}{4}\|\langle z \rangle^{k+\ell+{\alpha_3}} \partial_{z}\partial^{\alpha}  \tvae\|_{{L^2}(\mathbb{R}^{3}_{+})}^{2}+C\|\langle z \rangle^{k+\ell+{\alpha_3}} \partial^{\alpha}  \tvae\|_{{L^2}(\mathbb{R}^{3}_{+})}^{2}.
\end{align*}
To control the boundary integral   
$$I_{2}^{2}=-\int_{\mathbb{R}^{2}} \partial_{z} \partial^\alpha  \tilde{\varphi}^\epsilon \partial^\alpha  \tilde{\varphi}^\epsilon\Big|_{z=0} \ ,$$
the following three cases 
 should be considered.

\textbf{ Case 1 :} $|\alpha| \leq m-1$. Employing  trace estimate \eqref{trace}, we obtain
\begin{align*}
I_{2}^{2} 
&\leq \|\partial^{\alpha}\partial_{z}\tvae|_{z=0}\|_{L^2(\mathbb{R}^{2})}\|\partial^{\alpha}\tvae|_{z=0}\|_{L^2(\mathbb{R}^{2})}\\
&\leq C\|\partial^{\alpha}\partial^2_z\tvae\|_{L^2_{k+\ell}(\mathbb{R}^3_+)}
\|\partial^{\alpha}\partial_z\tvae\|_{L^2_{k+\ell}(\mathbb{R}^3_+)}\\
&\leq C\|\partial_z\tvae\|_{H^{m,m-1}_{k+\ell}(\mathbb{R}^3_+)}\|\tvae\|_{H^{m, m-1}_{k+\ell}(\mathbb{R}^3_+)}\\
&\leq \frac{1}{4}\|\partial_z\tvae\|^2_{H^{m, m-1}_{k+\ell}(\mathbb{R}^3_+)}+C\|\tvae\|^2_{H^{m}_{k+\ell}
	(\mathbb{R}^3_+)}.
\end{align*}

\textbf{ Case 2 :} $|\alpha|=m$ and $\alpha_{3}$ is even. 

{\rm(i)} When $\alpha_{1}=\alpha_{2} =0$, which implies $\alpha_{3}=m$ and $m$ is even, 
 we arrive at
\begin{align*}
	I_{2}^{2} 
&=-\int_{\mathbb{R}^{2}} \partial_{z}^{m+1}  \tilde{\varphi}^\epsilon 
	\partial_{z}^{m}  \tilde{\varphi}^\epsilon\Big|_{z=0}
	\\
&\leq \|\partial_{z}^{m+2}\tue|_{z=0}\|_{L^2(\mathbb{R}^{2})}\|\partial^{m}_{z}\tvae|_{z=0}\|_{L^2(\mathbb{R}^{2})} \\
&
\leq \|\partial_{z}^{m+2}\tue|_{z=0}\|_{L^2(\mathbb{R}^{2})}\|\partial^{m+1}_{z}\tvae\|_{L^2_{k+\ell}(\mathbb{R}^3_+)}\\
& \leq \frac{1}{4}\|\partial_z\tvae\|^2_{H^{m, m-1}_{k+\ell}(\mathbb{R}^3_+)}
+C\|\partial_{z}^{m+2}\tue|_{z=0}\|_{L^2(\mathbb{R}^{2})}^{2}.
\end{align*}
The  index $m+2$ is too high so that we cannot
control $\partial_{z}^{m+2}\tue|_{z=0}$ by virtue of trace estimate.  But by using the boundary reduction of Corollary \ref{c-boundary reduction}, for $p \leq \frac{m}{2}$ and $2 \leq q \leq p$, one has
\begin{align*}
\|\partial_{z}^{m+2}\tue|_{z=0}\|_{L^2(\mathbb{R}^{2})}
&\leq  C_{K, p, l, \alpha, \beta, \gamma, \epsilon}
\left\| \prod\limits_{i=1}^{q}   \partial^{\alpha}_{x}\partial^{\beta}_{y} \partial^{\gamma}_{z}( u^s + \tilde{u}^\epsilon ) 
\right\| _{L^{2}(\mathbb{R}_{+}^{3})} \\
& \leq C_{K, p, l, \alpha, \beta, \gamma, \epsilon} 
\left\|\tilde{\varphi}^\epsilon\right\|^{q}_{H^{m}_{k+\ell}
	(\mathbb{R}^3_+)} \\
& \leq C_{K, p, l, \alpha, \beta, \gamma, \epsilon} 
\|\tilde{\varphi}^\epsilon\|^{\frac{m}{2}}_{H^{m}_{k+\ell}
	(\mathbb{R}^3_+)}.
\end{align*}

{\rm(ii)} When $\alpha_{1}+\alpha_{2} \neq 0$,  the maximum of the index $\alpha_{1}+\alpha_{2}$ is $m -2$ since $m$ is even and $\alpha_{1}+\alpha_{2} \leq m-1$.  Then $I_{2}^{2}$   can be estimated as follows by similar method of argument in (i) of Case 2,
\begin{align*}
	I_{2}^{2}
	 \leq \frac{1}{4}\|\partial_z\tvae\|^2_{H^{m, m-1}_{k+\ell}(\mathbb{R}^3_+)}+ C_{K, p, l, \alpha, \beta, \gamma, \epsilon} 
	\|\tilde{\varphi}^\epsilon\|^{\alpha_{3}}_{H^{m}_{k+\ell}
		(\mathbb{R}^3_+)}.
\end{align*}

\textbf{ Case 3 :} $|\alpha|=m$ and $\alpha_{3}$ is odd. For the special case:
$\alpha_1+\alpha_{2}=m-1, \alpha_3=1$,    using Proposition \ref{boundary reduction} or the boundary condition $\eqref{vv-regular-shear-prandtl}_{2}$,   it is easy to check
\begin{align*}
	I_{2}^{2} =-\int_{\mathbb{R}^{2}}\partial^{\alpha_{1} }_{x}\partial^{\alpha_{2} }_{y} \partial^{2}_{z}\tilde{u}^\epsilon\partial_z\tvae   \Big|_{z=0}=0,
\end{align*}
then for the  other general cases: $1 \leq \alpha_{1}+\alpha_{2} \leq m-2$,
by employing integration by parts in the variable $x$ or  variable $y$, we have
\begin{align*}
I_{2}^{2}
&=\int_{\mathbb{R}^{2}} \partial_{xy}^{-1}\partial_z\partial^{\alpha}\tvae   
\partial_{xy}^{1}\partial^{\alpha}\tvae \big|_{z=0} \\
&\leq \|\partial_{xy}^{-1}\partial_z\partial^{\alpha}\tvae |_{z=0}\|_{L^2(\mathbb{R}^{2})}\|\partial_{xy}^{1}\partial^{\alpha}\tvae|_{z=0}\|_{L^2(\mathbb{R}^{2})}\\
&\leq \frac{1}{4}\|\partial_z\tvae\|^2_{H^{m, m-1}_{k+\ell}(\mathbb{R}^3_+)}+C
\|\partial_{xy}^{1}\partial^{\alpha_{1}}_{x}\partial^{\alpha_{2}}_{y}\partial^{\alpha_{3}+1}_{z}\tue|_{z=0}\|_{L^2(\mathbb{R}^{2})}\\
&\leq \frac{1}{4}\|\partial_z\tvae\|^2_{H^{m, m-1}_{k+\ell}(\mathbb{R}^3_+)}+ C_{K, p, l, \alpha, \beta, \gamma, \epsilon} 
	\|\tilde{\varphi}^\epsilon\|^{\alpha_{3}-1}_{H^{m}_{k+\ell}
		(\mathbb{R}^3_+)}.
\end{align*}

For   $I_{3}$,  it is trivial  to obtain directly  by integration by parts
\begin{align*}
	I_{3}=-\epsilon\left( \|\langle z \rangle^{k+\ell+{\alpha_3}} \partial_{x}\partial^{\alpha}  \tvae\|_{{L^2}(\mathbb{R}^{3}_{+})}^{2}+\|\langle z \rangle^{k+\ell+{\alpha_3}} \partial_{y}\partial^{\alpha}  \tvae\|_{{L^2}(\mathbb{R}^{3}_{+})}^{2}\right) .
\end{align*}

\textbf{Dealing with $I_4$ and $I_5$ terms :}
The term $I_{4}$ will be  estimated in two cases $\beta_{3}=0$ and $\beta_{3} \geq 1$ by using Lemmas \ref{hardy}-\ref{infty} and \ref{lemma2.4}.

When $\beta_{3}=0$, which implies $\beta_{1}+\beta_{2} \leq m-1$, we obtain 
\begin{align*}
I_{4}
&=-\sum\limits_{ \beta \leq \alpha, 1\leq |\beta|} C^\beta_\alpha\int_{{\mathbb{R}^3_+}} \langle z \rangle^{2(k+\ell+{\alpha_3})} 
\partial^{\beta_{1}}_{x}\partial^{\beta_{2}}_{y}\partial_{xy}^{-1}\partial_{xy}^{1}(u^s + \tilde{u}^\epsilon) \partial^{\alpha_{1}+\alpha_{2}+\alpha_{3} - \beta_{1}-\beta_{2}} \partial_{x}\tilde{\varphi}^\epsilon
\partial^{\alpha} \tilde{\varphi}^\epsilon
\\
&=-\sum\limits_{ \beta \leq \alpha, 1\leq |\beta|} C^\beta_\alpha\int_{{\mathbb{R}^3_+}} \langle z \rangle^{2(k+\ell+{\alpha_3})} 
\partial^{\beta_{1}}_{x}\partial^{\beta_{2}}_{y}\partial_{xy}^{-1}\partial_{z}^{-1}\partial_{xy}^{1}\tvae \partial^{\alpha_{1}+\alpha_{2}+\alpha_{3} - \beta_{1}-\beta_{2}} \partial_{x}\tilde{\varphi}^\epsilon
\partial^{\alpha} \tilde{\varphi}^\epsilon
\\
&\leq C\|\partial_{xy}^{1}\tvae\|_{H^{m-1}_{\frac {1}{2}+\delta}(\mathbb{R}^3_+)}\|\partial_{x}\tvae\|_{H^{m-1}_{k+\ell}(\mathbb{R}^3_+)}\|\partial^{\alpha} \tilde{\varphi}^\epsilon\|_{L^2_{k+\ell+\alpha_{3}}(\mathbb{R}^3_+)}\\
& \leq C \|\tvae\|_{H^{m}_{\frac {3}{2}+\delta}
	(\mathbb{R}^3_+)} \|\tvae\|^2_{H^{m}_{k+\ell}
	(\mathbb{R}^3_+)}.
\end{align*}
When $\beta_{3}\geq 1$, we have  with $e_{1}=(1, 0, 0)$  and $e_{3}=(0, 0, 1)$,
\begin{align*}
	I_{4}
	&=-\sum\limits_{ \beta \leq \alpha, 1\leq |\beta|} C^\beta_\alpha\int_{{\mathbb{R}^3_+}} \langle z \rangle^{2(k+\ell+{\alpha_3})} 
	\partial^{\beta-e_{3}}(u^s_{z} + \tilde{\varphi}^\epsilon) \partial^{\alpha- \beta+e_{1}} \tilde{\varphi}^\epsilon
	\partial^{\alpha} \tilde{\varphi}^\epsilon
	\\
	&
	\leq C\left( 1+\|\tvae\|_{H^{m}_{k+\ell}
		(\mathbb{R}^3_+)}\right) \|\tvae\|^2_{H^{m}_{k+\ell}
		(\mathbb{R}^3_+)}.
\end{align*}
Similar to the estimates on $I_{4}$, we can obtain
\begin{align*}
	I_{5}\leq
	C_{K}\left( 1+\|\tvae\|_{H^{m}_{k+\ell}
		(\mathbb{R}^3_+)}\right) \|\tvae\|^2_{H^{m}_{k+\ell}
		(\mathbb{R}^3_+)}.
\end{align*}

\textbf{Dealing with $I_6$ term :} By exploiting Lemmas \ref{hardy}-\ref{infty} and \ref{lemma2.4} likewise, 
we  can estimate  $I_{6}$, which is divided into four cases, as follows.

When $\beta_{3}= 0$, for all $|\alpha| \leq m-1$, we deduce 
\begin{align*}
	I_{6}
	&=-\sum\limits_{ \beta \leq \alpha, 1\leq |\beta|} C^\beta_\alpha\int_{{\mathbb{R}^3_+}} \langle z \rangle^{2(k+\ell+{\alpha_3})} 
	\partial^{\beta_{1}}_{x}\partial^{\beta_{2}}_{y} \partial_{z}^{-1} \uve \partial^{\alpha_{1}+\alpha_{2}+\alpha_{3} - \beta_{1}-\beta_{2}} \partial_{z}\tilde{\varphi}^\epsilon
	\partial^{\alpha} \tilde{\varphi}^\epsilon\\
	&\leq C\|\uve\|_{H^{m-1}_{\frac {1}{2}+\delta}(\mathbb{R}^3_+)}\|\partial_{z}\tvae\|_{H^{m-1}_{k+\ell}(\mathbb{R}^3_+)}\|\partial^{\alpha} \tilde{\varphi}^\epsilon\|_{L^2_{k+\ell+\alpha_{3}}(\mathbb{R}^3_+)}\\
	&\leq C\|(\partial_{x}\tvae+K\partial_{y}\tvae+\partial_{y}K(u^s_z+\tvae))\|_{H^{m-1}_{\frac {3}{2}+\delta}(\mathbb{R}^3_+)}\|\partial_{z}\tvae\|_{H^{m-1}_{k+\ell}(\mathbb{R}^3_+)}\|\partial^{\alpha} \tilde{\varphi}^\epsilon\|_{L^2_{k+\ell+\alpha_{3}}(\mathbb{R}^3_+)}\\
	& \leq C_{K}\left(1+\|\tvae\|_{H^{m}_{\frac {3}{2}+\delta}(\mathbb{R}^3_+)}\right)\|\tvae\|^2_{H^{m}_{k+\ell}
		(\mathbb{R}^3_+)}.
\end{align*}
When $\beta_{3}= 0$, $|\alpha| = m$, which implies $\alpha_{3} \geq 1$, we have with  $e_{3}=({0, 0, 1})$,
\begin{align*}
	I_{6}
	&=-\sum\limits_{ \beta \leq \alpha, 1\leq |\beta|} C^\beta_\alpha\int_{{\mathbb{R}^3_+}} \langle z \rangle^{2(k+\ell+{\alpha_3})} 
	\partial^{\beta_{1}}_{x}\partial^{\beta_{2}}_{y}\partial_{xy}^{-1} \partial_{z}^{-1}\partial_{xy}^{1}\uve \partial^{\alpha_{1}+\alpha_{2}+\alpha_{3} - \beta_{1}-\beta_{2}-e_{3}} \partial_{z}^{2}\tilde{\varphi}^\epsilon
	\partial^{\alpha} \tilde{\varphi}^\epsilon\\
	& \leq C  \|\partial_{xy}^{1}\uve\|_{H^{m-2}_{\frac {1}{2}+\delta}(\mathbb{R}^3_+)} \|\partial_{z}^{2}\tilde{\varphi}^\epsilon\|_{H^{m-2}_{k+\ell}}
	\|\partial^{\alpha} \tilde{\varphi}^\epsilon\|_{L^2_{k+\ell+\alpha_{3}}(\mathbb{R}^3_+)}
	\\
	& \leq C  \|\partial_{xy}^{1}(\partial_{x}\tvae+K\partial_{y}\tvae+\partial_{y}K(u^s_{z}+\tvae))\|_{H^{m-2}_{\frac {3}{2}+\delta}(\mathbb{R}^3_+)} \|\tilde{\varphi}^\epsilon\|_{H^{m}_{k+\ell}}
	\|\partial^{\alpha} \tilde{\varphi}^\epsilon\|_{L^2_{k+\ell+\alpha_{3}}(\mathbb{R}^3_+)}
	\\
	& \leq C_{K}\left(1+\|\tvae\|_{H^{m}_{\frac {3}{2}+\delta}(\mathbb{R}^3_+)}\right)\|\tvae\|^2_{H^{m}_{k+\ell}
		(\mathbb{R}^3_+)}.
\end{align*}
When $\beta_{3}= 1$, one has with 
$e_{3}=({0, 0, 1})$
\begin{align*}
	I_{6}
	&=-\sum\limits_{ \beta \leq \alpha, 1\leq |\beta|} C^\beta_\alpha\int_{{\mathbb{R}^3_+}} \langle z \rangle^{2(k+\ell+{\alpha_3})} 
	\partial^{\beta_{1}}_{x}\partial^{\beta_{2}}_{y} \uve \partial^{\alpha_{1}+\alpha_{2}+\alpha_{3} - \beta_{1}-\beta_{2}-e_{3}} \partial_{z}\tilde{\varphi}^\epsilon
	\partial^{\alpha} \tilde{\varphi}^\epsilon
	\\
	&=-\sum\limits_{ \beta \leq \alpha, 1\leq |\beta|} C^\beta_\alpha\int_{{\mathbb{R}^3_+}} \langle z \rangle^{2(k+\ell+{\alpha_3})} 
	\partial^{\beta_{1}}_{x}\partial^{\beta_{2}}_{y} \partial_{z}^{-1}(\partial_{x}\tvae+K\partial_{y}\tvae+\partial_{y}K(u^s_{z}+\tvae)) \partial^{\alpha_{1}+\alpha_{2}+\alpha_{3} - \beta_{1}-\beta_{2}-e_{3}} \partial_{z} \tilde{\varphi}^\epsilon
	\partial^{\alpha} \tilde{\varphi}^\epsilon
	\\
&\leq C\|(\partial_{x}\tvae+K\partial_{y}\tvae+\partial_{y}K(u^s_{z}+\tvae))\|_{H^{m-1}_{\frac {1}{2}+\delta}(\mathbb{R}^3_+)}\|\partial_{z}\tvae\|_{H^{m-1}_{k+\ell}(\mathbb{R}^3_+)}\|\partial^{\alpha} \tilde{\varphi}^\epsilon\|_{L^2_{k+\ell+\alpha_{3}}(\mathbb{R}^3_+)}\\
	& \leq C_{K}(1+\|\tvae\|_{H^{m}_{k+\ell}
		(\mathbb{R}^3_+)})\|\tvae\|^2_{H^{m}_{k+\ell}
		(\mathbb{R}^3_+)}.
\end{align*}
When $\beta_{3} \geq 2$, one has with 
$e_{3}=({0, 0, 1})$
\begin{align*}
	I_{6}
	&=-\sum\limits_{ \beta \leq \alpha, 1\leq |\beta|} C^\beta_\alpha\int_{{\mathbb{R}^3_+}} \langle z \rangle^{2(k+\ell+{\alpha_3})} 
	\partial^{\beta_{1}}_{x}\partial^{\beta_{2}}_{y} \partial_{z}^{\beta_{3}-2}	(\partial_{x}\tvae+K\partial_{y}\tvae+\partial_{y}K(u^s_{z}+\tvae))  \partial^{\alpha_{1}+\alpha_{2}+\alpha_{3} - \beta_{1}-\beta_{2}-\beta_{3}+e_{3}}  \tilde{\varphi}^\epsilon
	\partial^{\alpha} \tilde{\varphi}^\epsilon
	\\
		&\leq C\|(\partial_{x}\tvae+K\partial_{y}\tvae+\partial_{y}K(u^s_{z}+\tvae)) \|_{H^{m}_{k+\ell}(\mathbb{R}^3_+)}\|\tvae\|_{H^{m}_{k+\ell}(\mathbb{R}^3_+)}\|\partial^{\alpha} \tilde{\varphi}^\epsilon\|_{L^2_{k+\ell+\alpha_{3}}(\mathbb{R}^3_+)}\\
	& \leq C_{K}(1+\|\tvae\|_{H^{m}_{k+\ell}
		(\mathbb{R}^3_+)})\|\tvae\|^2_{H^{m}_{k+\ell}
		(\mathbb{R}^3_+)}.
\end{align*}

\textbf{Dealing with $I_7$ term :}
We move to estimate $I_{7}$ involving two cases: $\alpha_{3}=0$ and $\alpha_{3} \geq 1$. 
If $\alpha_{3}=0$, which implies $\alpha_{1}+\alpha_{2} \leq m-1$, we get
\begin{align*}
	I_{7}& \leq \|\partial^{\alpha_{1}}_{x} \partial^{\alpha_{2}}_{y}\twe ~ {u^s_{zz}}\|_{L^2_{k+\ell}(\mathbb{R}^3_+)}\|\partial^{\alpha_{1}}_{x} \partial^{\alpha_{2}}_{y}\tilde{\varphi}^\epsilon\|_{L^2_{k+\ell}(\mathbb{R}^3_+)} \\
	& \leq \|\partial^{\alpha_{1}}_{x} \partial^{\alpha_{2}}_{y}\partial_{z}^{-1}\uve{u^s_{zz}}\|_{L^2_{k+\ell}(\mathbb{R}^3_+)}\|\partial^{\alpha_{1}}_{x} \partial^{\alpha_{2}}_{y}\tilde{\varphi}^\epsilon\|_{L^2_{k+\ell}(\mathbb{R}^3_+)} \\
	& \leq \|(1+K+\partial_{y}K)\tue\|_{H^{m}_{\frac {1}{2}+\delta}(\mathbb{R}^3_+)}
	\|{u^s_{zz}}\|_{L^2_{k+\ell}(\mathbb{R}_+)}\|\partial^{\alpha_{1}}_{x} \partial^{\alpha_{2}}_{y}\tilde{\varphi}^\epsilon\|_{L^2_{k+\ell}(\mathbb{R}^3_+)} \\
	& \leq C_{K}\|\tvae\|_{H^{m}_{\frac {3}{2}+\delta}(\mathbb{R}^3_+)}
	\|{u^s_{zz}}\|_{L^2_{k+\ell}(\mathbb{R}_+)}\|\partial^{\alpha_{1}}_{x} \partial^{\alpha_{2}}_{y}\tilde{\varphi}^\epsilon\|_{L^2_{k+\ell}(\mathbb{R}^3_+)} \\
	& \leq C_{K}\|\tilde{\varphi}^{\epsilon}\|_{H^{m}_{\frac {3}{2}+\delta}(\mathbb{R}^3_+)}\|\tvae\|_{H^{m}_{k+\ell}
		(\mathbb{R}_{+}^{3})}.
\end{align*}
If $\alpha_{3} \geq 1$,  it is straightforward to obtain  by similar computations
\begin{align*}
	I_{7} 
	=-\sum\limits_{ \beta_{3} \leq \alpha_{3},  |\beta_{3}|\leq m-1} C^\beta_\alpha\int_{{\mathbb{R}^3_+}} \langle z \rangle^{2(k+\ell+{\alpha_3})} 
	\partial^{\alpha-\beta_{3} - e_{3}} \uve 
	\partial^{\beta_{3}}_{z}{u^s_{zz}}
	\partial^{\alpha} \tilde{\varphi}^\epsilon
	\leq C_{K}\|\tvae\|_{H^{m}_{k+\ell}
		(\mathbb{R}_{+}^{3})}^{2}.
\end{align*}

\textbf{Dealing with $I_8$  term :}
Rewrite $I_{8}$ as
\begin{align*}
	I_{8}=  -\int_{{\mathbb{R}^3_+}} \langle z \rangle^{2(k+\ell+{\alpha_3})} 
	\partial^{\alpha_{1}+\alpha_{2}}\left\lbrace \partial_{y}K\partial^{\alpha_{3}} \big((u^s+\tue)(u_z^s+\tvae)\big)\right\rbrace 
	\partial^{\alpha} \tilde{\varphi}^\epsilon.
\end{align*}
Obviously, it holds by virtue of  the same argumentation process as $I_{4}$
\begin{align*}
I_{8} \leq C_{K} 
	 \left( 1+\|\tvae\|_{H^{m}_{k+\ell}
	(\mathbb{R}^3_+)}\right)^2 \|\tvae\|_{H^{m}_{k+\ell}
	(\mathbb{R}^3_+)}.
\end{align*}
Collecting all estimates on $I_{1}-I_{8}$,  we have proved the 
inequality \eqref{3.1}. This completes the proof of Lemma \eqref{E1}.
\end{proof}

\subsection{Weighted $L^2$ estimates on $\partial^{m}_{xy} \tilde{\varphi}^{\epsilon}$ }

\begin{lemma}\label{E2}
	Under the hypotheses of Lemma \ref{E1}, it holds that
	\begin{align}\begin{split}
		&\frac{d}{dt}\left\|\langle z \rangle^{k+\ell } \partial^{m}_{xy}  \tvae\right\|_{L^2(\mathbb{R}^3_+)}^{2}
		+\left\|\langle z \rangle^{k+\ell } \partial_{z}\partial^{m}_{xy}  \tvae\right\|_{L^2(\mathbb{R}^3_+)}^{2}
		+\epsilon\left( \left\|\langle z \rangle^{k+\ell } \partial_{x}\partial^{m}_{xy}  \tvae\right\|_{L^2(\mathbb{R}^3_+)}^{2}+\left\|\langle z \rangle^{k+\ell } \partial_{y}\partial^{m}_{xy}  \tvae\right\|_{L^2(\mathbb{R}^3_+)}^{2}\right) \\
		&\leq C\left( \|\tvae\|_{H^{m}_{k+\ell}
			(\mathbb{R}^3_+)}^2+\|\tvae\|_{H^{m}_{k+\ell}
			(\mathbb{R}^3_+)}^4\right) +\frac{C}{\epsilon}\left( \|\tvae\|_{H^{m}_{k+\ell}
			(\mathbb{R}^3_+)}^2+\|\tvae\|_{H^{m}_{k+\ell}
			(\mathbb{R}^3_+)}^4\right) .
		\end{split}
		\label{3.4}
	\end{align}
\end{lemma}
\begin{proof}
By \eqref{3.2}, we have
\begin{align}
	\begin{split}
		&\partial_t\partial^{m}_{xy}\tilde{\varphi}^\epsilon +  (u^s + \tilde{u}^\epsilon) \partial_x\partial^{m}_{xy}\tilde{\varphi}^\epsilon +K(u^s + \tilde{u}^\epsilon) \partial_y\partial^{m}_{xy}\tilde{\varphi}^\epsilon+ \tilde{w}^\epsilon \partial_z\partial^{m}_{xy}  \tilde{\varphi}^\epsilon -\partial^2_z\partial^{m}_{xy}\tilde{\varphi}^\epsilon -\epsilon \partial^2_x\partial^{m}_{xy}\tilde{\varphi}^\epsilon-\epsilon \partial^2_y\partial^{m}_{xy}\tilde{\varphi}^\epsilon 
		\\
		&=-\sum\limits_{ 1 \leq j\leq m} C^j_m \partial^{j}_{xy}(u^s + \tilde{u}^\epsilon) \partial^{m - j}_{xy}\partial_x \tilde{\varphi}^\epsilon
		-\sum\limits_{ 1 \leq j \leq m} C^j_m \partial^{j}_{xy}\big(K(u^s + \tilde{u}^\epsilon) \big)\partial^{m- j}_{xy}\partial_y \tilde{\varphi}^\epsilon
		-\sum\limits_{ 1 \leq j \leq m} C^j_m \partial^{j}_{xy}\tilde{w}^\epsilon \partial^{m - j}_{xy}\partial_z \tilde{\varphi}^\epsilon
		\\
		&\quad-\partial^{m}_{xy}\big(\tilde{w}^\epsilon ~u^s_{zz} \big)-\partial^{m}_{xy}\big(\partial_{y}K (u^s+\tue)\partial_{z}(u^s+\tue)\big).
	\end{split}
	\label{3.5}
\end{align}
Then the same trick  as in Lemma \ref{E1} yields
\begin{align}
		&\frac{1}{2}\frac{d}{dt}\left\|\langle z \rangle^{k+\ell } \partial^{m}_{xy}  \tvae\right\|_{L^2(\mathbb{R}^3_+)}^{2}
		+\frac{3}{4}\left\|\langle z \rangle^{k+\ell } \partial_{z}\partial^{m}_{xy}  \tvae\right\|_{L^2(\mathbb{R}^3_+)}^{2}\nonumber\\
		&\quad
		+\epsilon\left( \left\|\langle z \rangle^{k+\ell } \partial_{x}\partial^{m}_{xy}  \tvae\right\|_{L^2(\mathbb{R}^3_+)}^{2}+\left\|\langle z \rangle^{k+\ell } \partial_{y}\partial^{m}_{xy}  \tvae\right\|_{L^2(\mathbb{R}^3_+)}^{2}\right) \nonumber\\
		& \leq C \left( \|\tvae\|^{2}_{H^{m}_{k+\ell}(\mathbb{R}_{+}^{3})}+\|\tvae\|^{3}_{H^{m}_{k+\ell}(\mathbb{R}_{+}^{3})}\right) 
		-\sum\limits_{ 1 \leq j \leq m} C^j_m
		\int_{{\mathbb{R}^3_+}} \langle z \rangle^{2(k+\ell )}
		 \partial^{j}_{xy}\tilde{w}^\epsilon \partial^{m - j}\partial_z \tilde{\varphi}^\epsilon  \partial^{m}_{xy} \tilde{\varphi}^\epsilon \nonumber\\
	&\quad-\int_{{\mathbb{R}^3_+}} \langle z \rangle^{2(k+\ell )}\partial^{m}_{xy}\big(\tilde{w}^\epsilon (u^s_{zz}, v^{s}_{zz})\big) \partial^{m}_{xy} \tilde{\varphi}^\epsilon
	-\int_{{\mathbb{R}^3_+}} \langle z \rangle^{2(k+\ell )}\partial^{m}_{xy}\big(\partial_{y}K (u^s+\tue)\partial_{z}(u^s+\tue)\big) \partial^{m}_{xy} \tilde{\varphi}^\epsilon
	,
 	\label{3.6}
 	\end{align}
where we have used the fact $\partial_{z}\partial^{m}_{xy}  \tvae|_{z=0}=0$.

Now, let us estimate the three integral terms on the right-hand side of \eqref{3.6}.

\rm{(1)} We divide the first integral on the right-hand side of the above equation
$$-\sum\limits_{ 1 \leq j \leq m} C^j_m
\int_{{\mathbb{R}^3_+}} \langle z \rangle^{2(k+\ell )}
\partial^{j}_{xy}\tilde{w}^\epsilon \partial^{m - j}\partial_z \tilde{\varphi}^\epsilon  \partial^{m}_{xy} \tilde{\varphi}^\epsilon$$
into two parts:
\begin{align*}
	&-\sum\limits_{ 1 \leq j \leq m} C^j_m
	\int_{{\mathbb{R}^3_+}} \langle z \rangle^{2(k+\ell )}
	\partial^{j}_{xy}\tilde{w}^\epsilon \partial^{m - j}\partial_z \tilde{\varphi}^\epsilon  \partial^{m}_{xy} \tilde{\varphi}^\epsilon
	\\
	&=-\sum\limits_{ 1 \leq j \leq m-1} C^j_m
	\int_{{\mathbb{R}^3_+}} \langle z \rangle^{2(k+\ell )}
	\partial^{j}_{xy}\tilde{w}^\epsilon \partial^{m - j}_{xy}\partial_z \tilde{\varphi}^\epsilon  \partial^{m}_{xy} \tilde{\varphi}^\epsilon
	-
	\int_{{\mathbb{R}^3_+}} \langle z \rangle^{2(k+\ell )}
	\partial^{m}_{xy}\tilde{w}^\epsilon \partial_z \tilde{\varphi}^\epsilon  \partial^{m}_{xy} \tilde{\varphi}^\epsilon .
\end{align*}
For the first part, it is easy to get
$$-\sum\limits_{ 1 \leq j \leq m-1} C^j_m
\int_{{\mathbb{R}^3_+}} \langle z \rangle^{2(k+\ell )}
\partial^{j}_{xy}\tilde{w}^\epsilon \partial^{m - j}_{xy}\partial_z \tilde{\varphi}^\epsilon  \partial^{m}_{xy} \tilde{\varphi}^\epsilon
\leq 
C_{K}\big(1+\|\tvae\|_{H^{m}_{k+\ell}(\mathbb{R}_{+}^{3})}\big) \|\tvae\|^{2}_{H^{m}_{k+\ell}(\mathbb{R}_{+}^{3})}.$$
One the other hand, the second part, which  contains the bad term $\partial^{m}_{xy}\tilde{w}^\epsilon \partial_z \tilde{\varphi}^\epsilon$, can be estimated as follows
\begin{align*}
&-
\int_{{\mathbb{R}^3_+}} \langle z \rangle^{2(k+\ell )}
\partial^{m}_{xy}\tilde{w}^\epsilon \partial_z \tilde{\varphi}^\epsilon  \partial^{m}_{xy} \tilde{\varphi}^\epsilon  \\
&
\leq \left\|  \langle z \rangle^{k+\ell }\partial^{m}_{xy}\partial_{z}^{-1}\uve \partial_z \tilde{\varphi}^\epsilon\right\|_{L^{2}(\mathbb{R}^{3}_{+})}
\left\| \langle z \rangle^{k+\ell }\partial^{m}_{xy} \tilde{\varphi}^\epsilon\right\|_{L^{2}(\mathbb{R}^{3}_{+})} \\
&\leq C \left\|\partial_{xy}\uve\right\|_{H^{m-1}_{\frac 12+\delta}(\mathbb{R}^3_+)}\left\|\partial_z \tilde{\varphi}^\epsilon\right\|_{H^{m-1}_{k+\ell}(\mathbb{R}^3_+)}\left\| \langle z \rangle^{k+\ell }\partial^{m}_{xy} \tilde{\varphi}^\epsilon\right\|_{L^{2}(\mathbb{R}^{3}_{+})}
\\
&
\leq C_{K} \big(1+\|\partial_{xy}^{m+1}\tvae\|_{L^{2}_{\frac{3}{2}+\delta}(\mathbb{R}^3_+)}\big)\|\tvae\|^{2}_{H^{m}_{k+\ell}(\mathbb{R}_{+}^{3})}.
\end{align*}
where, in the last two steps, we have used the \eqref{s2} and  Hardy inequality \eqref{hardy1}.

\rm{(2)} Using the assumption for the shear flow $u^s$, we have
\begin{align*}
&-\int_{{\mathbb{R}^3_+}} \langle z \rangle^{2(k+\ell )}\partial^{m}_{xy}\big(\tilde{w}^\epsilon ~ {u^s_{zz}}\big) \partial^{m}_{xy} \tilde{\varphi}^\epsilon \\
&\leq 
\left\|  \langle z \rangle^{k+\ell }\partial^{m}_{xy}\tilde{w}^\epsilon {u^s_{zz}}\right\|_{L^{2}(\mathbb{R}^{3}_{+})}
\left\| \langle z \rangle^{k+\ell }\partial^{m}_{xy} \tilde{\varphi}^\epsilon\right\|_{L^{2}(\mathbb{R}^{3}_{+})} \\
&\leq \left\|\partial^{m}_{xy}\tilde{w}^\epsilon\right\|_{L^{2}(\mathbb{R}_{xy}^{2}; L^\infty(\mathbb{R}^{+}))}
\big\|{u^s_{zz}}\big\|_{L^2_{k+\ell}(\mathbb{R}_+)}\left\| \langle z \rangle^{k+\ell }\partial^{m}_{xy} \tilde{\varphi}^\epsilon\right\|_{L^{2}(\mathbb{R}^{3}_{+})}\\
&\leq C\|\partial_{xy}^{m}(\partial_{x}\tvae+K\partial_{y}\tvae+\partial_{y}K(u^s_{z}+\tvae))\|_{L^{2}_{\frac {3}{2}+\delta}(\mathbb{R}^3_+)}
\left\| \langle z \rangle^{k+\ell }\partial^{m}_{xy} \tilde{\varphi}^\epsilon\right\|_{L^{2}(\mathbb{R}^{3}_{+})}\\
& \leq C_{K} \big(1+\|\partial_{xy}^{m+1}\tvae\|_{L^{2}_{\frac{3}{2}+\delta}(\mathbb{R}^3_+)}\big)\|\tvae\|_{H^{m}_{k+\ell}(\mathbb{R}_{+}^{3})}.
\end{align*}

The order of $K$ in the above inequality cannot exceed $m+1$, 
so this explains why $K$ needs to satisfy condition 
$$\big\|K\big\|_{W^{m+1, \infty}(\mathbb{R}^2)} < \infty.$$

\rm{(3)} The term $-\int_{{\mathbb{R}^3_+}} \langle z \rangle^{2(k+\ell )}\partial^{m}_{xy}\big(\partial_{y}K (u^s+\tue)\partial_{z}(u^s+\tue)\big) \partial^{m}_{xy} \tilde{\varphi}^\epsilon$ is estimated as follows,
\begin{align*}
	&-\int_{{\mathbb{R}^3_+}} \langle z \rangle^{2(k+\ell )}\partial^{m}_{xy}\big(\partial_{y}K (u^s+\tue)\partial_{z}(u^s+\tue)\big) \partial^{m}_{xy} \tilde{\varphi}^\epsilon\\
	&=\frac{1}{2}\int_{{\mathbb{R}^3_+}} \langle z \rangle^{2(k+\ell )}\partial^{m}_{xy}\big(\partial_{y}K (u^s+\tue)^2\big)\partial_{z} \partial^{m}_{xy} \tilde{\varphi}^\epsilon
	+(k+\ell)\int_{{\mathbb{R}^3_+}} \langle z \rangle^{2(k+\ell )-1}\partial^{m}_{xy}\big(\partial_{y}K (u^s+\tue)^2\big) \partial^{m}_{xy} \tilde{\varphi}^\epsilon
	\\
	&\leq 
	C_{K}\big(1+\|\tvae\|_{H^{m}_{k+\ell}(\mathbb{R}_{+}^{3})}\big)^2 \big(\|\partial_{z}\partial^{m}_{xy} \tilde{\varphi}^\epsilon\|_{L^2_{k+\ell}(\mathbb{R}^3_+)}+\|\partial^{m}_{xy} \tilde{\varphi}^\epsilon\|_{L^2_{k+\ell}(\mathbb{R}^3_+)}\big).
\end{align*}

Substituting estimates in (1)-(3) into \eqref{3.6}, we obtain the estimate \eqref{3.4} immediately after using the Cauchy inequality. This
completes the proof of Lemma \ref{E2}.
\end{proof}

\begin{proof}[{\bf Closeness of  a priori estimate and proof of Theorem \ref{epsilon-existence}}]
	Combining \eqref{E1} and \eqref{E2},  for $m \geq 6$, $k >1$, $\frac {3}{2} - k < \ell < \frac {1} {2} $ and $0 <\epsilon \leq 1$, we have
\begin{align}\begin{split}
		\frac{d}{dt}\|\tvae\|^{2}_{H^{m}_{k+\ell}(\mathbb{R}_{+}^{3})} \leq \frac{C}{\epsilon} \left( \|\tvae\|^{2}_{H^{m}_{k+\ell}(\mathbb{R}_{+}^{3})}+\|\tvae\|^{m}_{H^{m}_{k+\ell}(\mathbb{R}_{+}^{3})}\right) ,
	\end{split}
	\label{3.7}
\end{align}
with $C>0$ being independent of $\epsilon$.

	We shall denote
$$\Phi=\|\tvae\|^{2}_{H^{m}_{k+\ell}(\mathbb{R}_{+}^{3})},$$
then it follows  from \eqref{3.7} that
\begin{align*}
		\frac{d}{dt}\Phi \leq \frac{C}{\epsilon} \left( \Phi+\Phi^{\frac{m}{2}}\right) , 
\end{align*}
which implies
\begin{align*}
\left( -\frac{1}{\Phi^{\frac{m-2}{2}}}\right) _{t} \leq \frac{C}{\epsilon}
\frac{m-2}{2} \left( -\left( -\frac{1}{\Phi^{\frac{m-2}{2}}}\right)+1\right) .
\end{align*}
Integrating above inequality over $[0, t]$, we conclude that  by  Gronwall's inequality,
\begin{align*}
 -\frac{1}{\Phi^{\frac{m-2}{2}}}\leq 
 e^{-\frac{C}{\epsilon}
 	\frac{m-2}{2}t}
 \left(  -\frac{1}{\Phi(0)^{\frac{m-2}{2}}}+\frac{C}{\epsilon}
  	\frac{m-2}{2}t\right) .
\end{align*}
Through some simple calculations for $0< t \leq  T^{\epsilon}$, 
\begin{align*}
\Phi^\frac{m-2}{2} \leq \frac{\Phi(0)^{\frac{m-2}{2}}}{ e^{-\frac{C}{\epsilon}
		\frac{m-2}{2}t}
		-\frac{C}{\epsilon}
		\frac{m-2}{2}t\Phi(0)^{\frac{m-2}{2}}},
\end{align*}
%
%
%
%
%
%
%
%
	where we have chosen $T^\epsilon >0$ so small that
	\begin{align} \label{Tm}
		\left(e^{\frac{C}{\epsilon}
			\frac{m-2}{2}T^\epsilon}-\frac{C}{\epsilon}
		\frac{m-2}{2}T^\epsilon
		\tilde{\zeta}^{\, m-2} \right)^{-1}=\left(\frac 43\right)^{m-2}.
	\end{align}
Thus, we  deduce that for any $\|\tvae(0)\|_{H^m_{k+\ell}}\leq \tilde{\zeta}$, and $0 <\epsilon \leq \epsilon_0$,
	\begin{align*}
		\| \tvae(t)\|_{H^m_{k+\ell}(\mathbb{R}^3_+)} \leq \frac {4}{3}\| \tvae(0)\|_{H^m_{k+\ell}(\mathbb{R}^3_+)} \leq 2\| \tilde{\varphi}_{0}\|_{H^m_{k+\ell}(\mathbb{R}^3_+)}, 
	\end{align*}
for any $0< t \leq  T^{\epsilon}$. This completes the proof of Theorem \ref{epsilon-existence}.
\end{proof}

\section{Formal transformations}\label{s4}
In this section,  we are devoted to improving the results of Lemma \ref{E2},
since estimation \eqref{3.4} depends on $\epsilon$. 
To simplify the notations, from now on, we drop the notation tilde and sub-index $\epsilon$, that is, with no confusion, we take
$$ (u, w)=(\tue, \twe), \qquad \varphi=\tvae .$$
For $0 \leq n \leq m$, we have the following formal transformations of system \eqref{regular-shear-prandtl},
\begin{equation}\label{formal transformations}
	\begin{cases} 
		\partial_{t}g^{n}+(u^{s}+u)\partial_{x}g^{n}
	+K(u^s+u)\partial_{y}g^{n}
	-\partial_{z}^{2}g^{n}
	-\epsilon\partial_{x}^{2}g^{n}
	-\epsilon\partial_{y}^{2}g^{n}\\
-2\epsilon \partial_{x}\partial_{z}^{-1}g^{n}\partial_{z}\eta_{xz} -2\epsilon \partial_{y}\partial_{z}^{-1}g^{n}\partial_{z}\eta_{yz} 
	=\sum^{7}_{i=1}M_{i}(g^{n}), \\
	\partial_{z} g^{n}\big|_{z=0}=0, \\
	g^{n}\big|_{t=0}=g^{n}_{0},
	\end{cases}
\end{equation}
with
\begin{align*}
	M_{1}(g^{n})&=-\big\lbrace (u^s + {u})(g^{n} \eta_{xz} + \partial_z^{-1} g^{n}  \partial_z \eta_{xz}) +K(u^s + {u})(  g^{n} \eta_{yz} + \partial_z^{-1} g^{n}  \partial_z \eta_{yz})\big\rbrace ,
	\\
	M_{2}(g^{n})&= 2 \partial_z g^{n} \eta_{zz} +2 g^{n} \partial_z \eta_{zz} -  4g^{n}\eta_{zz}^2- 8 \partial_z^{-1} g^{n} \eta_{zz} \partial_z \eta_{zz},
	\\
	M_{3}(g^{n})&=\epsilon\big(
	2 \partial_x g^{n} \eta_{xz}  -  2g^{n}\eta_{xz}^2- 4 \partial_z^{-1} g^{n} \eta_{xz} \partial_z \eta_{xz}
	\big) ,
	\\
	M_{4}(g^{n})&=\epsilon\big(
	2 \partial_y g^{n} \eta_{yz}  -  2g^{n}\eta_{yz}^2- 4 \partial_z^{-1} g^{n} \eta_{yz} \partial_z \eta_{yz}
	\big) ,
	\\
	M_{5}(g^{n})&=-K\partial_{y}\partial_{xy}^{n} u+\partial_{y}\partial_{xy}^{n}\big(K(u^s+u)\big),
	\\
	M_{6}(g^{n})&=\partial_z\left\lbrace \partial_z^{-1} g^{n} \bigg(  \frac{(u^s+u)\partial_{x}\varphi+K(u^s+u)\partial_{y}\varphi+w(u_{zz}^s+\partial_{z}\varphi)}{u^s_z + \tilde{u}_z}-\partial_{y}K(u^s+u)\bigg)\right\rbrace ,
	\\
	M_{7}(g^{n})&=\partial_{z}	\left\lbrace \bigg(-\sum_{i=1}^{n}C^i_n\partial_{xy}^i {u} \, \partial^{n  -i}_{xy} \partial_{x}{u}
	-\sum_{i=1}^{n}C^i_n\partial_{xy}^i {(K(u^s+u))} \, \partial^{n  -i}_{xy} \partial_{y}{u}
	-\sum_{i=1}^{n}C^i_n\partial_{xy}^i {\varphi} \,\partial^{n  -i}_{xy} {w}\bigg) \bigg{/} (u_{z}^{s}+u_{z})\right\rbrace ,
\end{align*}
where
\begin{align*}
&g^{n}=\bigg(\frac{\partial_{xy}^{n} u}{u_{z}^{s}+u_{z}}\bigg)_{z}, \qquad 
\eta_{xz}=
\frac{ u_{xz}}{u_{z}^{s}+u_{z}}, 
\qquad \eta_{yz}
=\frac{ u_{yz}}{u_{z}^{s}+u_{z}}, \qquad \quad
\eta_{zz}
=\frac{ u_{zz}^{s}+ u_{zz}}{u_{z}^{s}+u_{z}}.
\end{align*}
For the justification of \eqref{formal transformations}, see
Appendix C.

\begin{lemma}\label{y4.1}
Let  $m \geq  6$, $k > 1$, $0< \ell< \frac{1}{2}$, $\frac{1}{2} < \ell' < \ell + \frac{1}{2}$, and $k + \ell > \frac {3}{2}$. If $\varphi \in L^{\infty} ([0, T]; H^{m}_{k+\ell}(\mathbb{R}^3_+)$ solves \eqref{vv-regular-shear-prandtl} and  satisfies the following a priori condition
\begin{align}
	\|\varphi\|_{ L^{\infty} ([0, T]; H^{m}_{k+\ell}(\mathbb{R}^3_+))}\leq  \zeta,
	\label{4.1}
\end{align}
then, for all $ (t, x, y, z) \in [0, T] \times
\mathbb{R}^{3}_{+}$, we have
\begin{align}
	|\partial_{z} u(t, x, y, z)|=|\varphi (t, x, y, z)| \leq C_{m}\zeta \langle z \rangle ^{-k-\ell}.
\label{4.2}
\end{align}
Furthermore, we assume that   $\zeta$ is small enough such that
\begin{align}
	C_m \zeta \leq \frac{\tilde {c_{1}}}{4},
	\label{4.3}
\end{align}
then,   for $\ell\geq 0$ and $(t, x, y, z)\in [0, T]\times \mathbb{R}^{2} \times \mathbb{R}^+$, we have
\begin{align}
	&\frac{\tilde c_1}{4} \langle z \rangle^{-k} \leq \left|  u^s_z + u_z\right| \leq 4\tilde c_2 \langle z \rangle^{-k}.
\label{4.4}
\end{align}
\end{lemma}
\begin{proof}
It follows from \ref{infty2} that 
\begin{align}
	&\|\langle z\rangle ^{k+\ell} \varphi\|_{ L^\infty ([0, T]\times\mathbb{R}^3_+)}
\nonumber\\
&\leq C\left(\|\langle z \rangle ^{\frac {1}{2}+\delta}(\langle z \rangle ^{k+\ell}w)_z \|_{ L^\infty ([0, T]; L^2(\mathbb{R}^3_+))}+
\|\langle z \rangle ^{\frac {1}{2}+\delta}(\langle z \rangle ^{k+\ell}w)_{xz} \|_{ L^\infty ([0, T]; L^2(\mathbb{R}^3_+))} \right.
\nonumber\\
&\quad+\left.\|\langle z \rangle ^{\frac {1}{2}+\delta}(\langle z \rangle ^{k+\ell}w)_{yz} \|_{ L^\infty ([0, T]; L^2(\mathbb{R}^3_+))}+
\|\langle z \rangle ^{\frac {1}{2}+\delta}(\langle z \rangle ^{k+\ell}w)_{xyz} \|_{ L^\infty ([0, T]; L^2(\mathbb{R}^3_+))} \right)
\nonumber\\
&\leq C_m \|\varphi\|_{ L^\infty ([0, T]; H^{m}_{k+\ell}(\mathbb{R}^3_+))},
  \label{4.6}
\end{align}
which, together with \eqref{4.1}, implies
\begin{align}
	|\partial_{z} u(t, x, y, z)|=|\varphi (t, x, y, z)| \leq C_{m}\zeta \langle z \rangle ^{-k-\ell}.
	\label{4.7}
\end{align}
Then, this  yields that
\begin{align*}
	\frac{\tilde c_1}{4} \langle z \rangle^{-k} \leq \left|  u^s_z + u_z\right| \leq 4\tilde c_2 \langle z \rangle^{-k},
\qquad (t, x, y, z)\in [0, T]\times \mathbb{R}^{2} \times \mathbb{R}^+.
\end{align*}
\end{proof}

\begin{lemma}\label{y4.2}
Under the hypotheses of \eqref{4.1}-\eqref{4.2} and Lemma \ref{y4.1}, 	  we have for $\tilde{\varphi}_{0}\in H^{m+2}_{k + \ell'}(\mathbb{R}^3_+)$,  $ g^{m}_{0} \in H^2_{k+\ell}(\mathbb{R}^3_+)$,  and $0<\zeta\le1$, 
\begin{align*}
	\|  g^m(0)\|_{H^2_{\ell'}(\mathbb{R}^3_+)}\le C \|\tilde w_0\|_{H^{m+2}_{k + \ell''}(\mathbb{R}^3_+)}.
\end{align*}
\end{lemma}
\begin{proof}
In actuality, 
\begin{align*}
g^{m}(0)=\bigg(\frac{\partial_{xy}^{n} u_{0}}{u_{0,z}^{s}+\tu_{0, z}}\bigg)_{z}
=\frac{\partial_{z}\partial_{xy}^{n} u_{0}}{u_{0,z}^{s}+\tu_{0, z}}
-
\frac{\partial_{xy}^{n} u_{0}}{u_{0,z}^{s}+\tu_{0, z}}\eta_{zz},
\end{align*}
then \eqref{4.3} implies that
\begin{align*}
\langle z \rangle^{\ell'} \left|g^{m}(0)\right|
&\leq\left|   
\frac{1}{u_{0,z}^{s}+\tu_{0, z}}\right| \cdot \Big(\left|  \partial_{xy}^{n}\varphi_{0}\right| 
+\left|  \eta_{zz} \partial_{xy}^{n}u_{0}\right| \Big) \\
&\leq C\langle z \rangle^{k+\ell'}\left| \partial_{xy}^{n}\varphi_{0}\right| +
C\langle z \rangle^{k+\ell'-1}\left|\partial_{xy}^{n}u_{0}\right|,
\end{align*}
where we have used the fact by Proposition \ref{shear-profile}
\begin{align*}
 \eta_{zz}  \leq \langle z \rangle^{-1}.
\end{align*}
Thus, the proof of Lemma \ref{y4.2}
 is completed.
\end{proof}

\begin{lemma}\label{gn}
Assume the condition (H) holds. Let ${\varphi} \in L^\infty([0, T]; H^{m+2}_{k + \ell}(\mathbb{R}^3_+)$,  $m \geq  6$, $k > 1$, $0\leq \ell< \frac{1}{2}$, $\frac{1}{2} < \ell' < \ell + \frac{1}{2}$, and $k + \ell > \frac {3}{2}$, satisfy \eqref{4.1}-\eqref{4.2} with
$0 < \zeta \leq 1$. Assume that the shear flow $u^s$ verifies the conclusion of  Proposition \ref{shear-profile},  and $g^n$ satisfies the equation
\eqref{formal transformations}  for $1 \leq n \leq m$. And $K(x, y)$ is supposed to satisfy that
$$\big\|K\big\|_{W^{m+1, \infty}(\mathbb{R}^2)} < \infty.$$
Then we have the following estimates, for any $t \in [0, T]$
\begin{align*}
&\frac{d}{dt}\sum^{m}_{n=1}\left\| g^{n}\right\|_{L^{2}_{\ell'}(\mathbb{R}_{+}^{3})}^2
+\sum^{m}_{n=1}\left\| \partial_{z}g^{n}\right\|_{L^{2}_{\ell'}(\mathbb{R}_{+}^{3})}^2+
+\epsilon\sum^{m}_{n=1}\left(\left\| \partial_{x}g^{n}\right\|_{L^{2}_{\ell'}(\mathbb{R}_{+}^{3})}^2
+
\left\| \partial_{y}g^{n}\right\|_{L^{2}_{\ell'}(\mathbb{R}_{+}^{3})}^2\right) \\
&\leq
C\left(\sum^{m}_{n=1}\left\| g^{n}\right\|_{L^{2}_{\ell'}(\mathbb{R}_{+}^{3})}^2
+
\left\| \varphi \right\|_{H^{m}_{k+\ell'}(\mathbb{R}_{+}^{3})}^2\right), 
\end{align*}
where a constant $C>0$ is independent of $\epsilon$. 
\end{lemma}
\begin{proof}
Multiplying \eqref{formal transformations} by $\langle z \rangle^{2\ell'} g^{n} $, integrating the resulting equation by parts over $\mathbb{R}_{+}^{3}$ in the $x$-variable and $y$-variable, respectively, we have
\begin{align}
\begin{split}
&\frac{1}{2}\frac{d}{dt}\left\|g^{n}\right\|^{2}_{L^{2}_{\ell'}(\mathbb{R}_{+}^{3})}
+\epsilon\left\|\partial_{x}g^{n}\right\|^{2}_{L^{2}_{\ell'}(\mathbb{R}_{+}^{3})}
+\epsilon\left\|\partial_{y}g^{n}\right\|^{2}_{L^{2}_{\ell'}(\mathbb{R}_{+}^{3})} \\
&= -\int_{{\mathbb{R}^3_+}}\langle z \rangle^{2\ell'} g^{n}	\left((u^{s}+u)\partial_{x}g^{n}
+K(u^s+u)\partial_{y}g^{n}\right) +\int_{{\mathbb{R}^3_+}}\langle z \rangle^{2\ell'} g^{n} \partial_{z}^{2}g^{n} \\
&\quad+
2\epsilon\int_{{\mathbb{R}^3_+}}  \langle z \rangle^{2\ell'} g^{n}\big(\partial_{x}\partial_{z}^{-1}g^{n}\partial_{z}\eta_{xz} \big)
+
2\epsilon\int_{{\mathbb{R}^3_+}}  \langle z \rangle^{2\ell'} g^{n}\big(\partial_{y}\partial_{z}^{-1}g^{n}\partial_{z}\eta_{yz} \big)
\\
&\quad + \sum_{i=1}^{7} \int_{{\mathbb{R}^3_+}}  \langle z \rangle^{2\ell'} g^{n} M(g^{n})_{i}  .
\end{split}
\label{4.9}
\end{align}
Now, we estimate each term on the right-hand side of \eqref{4.9}.
 By Lemma \eqref{infty}, we have
\begin{align*}
&-\int_{{\mathbb{R}^3_+}}\langle z \rangle^{2\ell'} g^{n}	\left((u^{s}+u)\partial_{x}g^{n}
	+K(u^s+u)\partial_{y}g^{n}\right) \\
&=\frac{1}{2} \int_{{\mathbb{R}^3_+}}\langle z \rangle^{2\ell'} \left( g^{n}\right)^{2}	\left(\partial_{x}u+\partial_{y}(K(u^s+u))\right) \leq 
C_{K} \left\|g^{n}\right\|^{2}_{L^{2}_{\ell'}(\mathbb{R}_{+}^{3})}
\left( 1+\left\|\varphi\right\|_{H^{3}_{\frac{1}{2}+\delta}(\mathbb{R}_{+}^{3})}\right) 
.
\end{align*} 
By the integration  by parts in the   $z$-variable, where the boundary value is vanish, we obtain
\begin{align*}
	\int_{{\mathbb{R}^3_+}}\langle z \rangle^{2\ell'} g^{n} \partial_{z}^{2}g^{n}
	&= -\left\|\partial_{z}g^{n}\right\|^{2}_{L^{2}_{\ell'}(\mathbb{R}_{+}^{3})}
	+2\ell'\int_{{\mathbb{R}^3_+}} \left| \langle z \rangle^{2\ell'-1}g^{n} \partial_{z}g^{n}\right|  \\
	& \leq 
	\frac{1}{4}\left\|\partial_{z}g^{n}\right\|^{2}_{L^{2}_{\ell'}(\mathbb{R}_{+}^{3})}+C\left\|g^{n}\right\|^{2}_{L^{2}_{\ell'}(\mathbb{R}_{+}^{3})} .
\end{align*}
Applying the Cauchy inequality leads to
\begin{align*}
&2\epsilon\int_{{\mathbb{R}^3_+}}  \langle z \rangle^{2\ell'} g^{n}\big(\partial_{x}\partial_{z}^{-1}g^{n}\partial_{z}\eta_{xz} \big)\\
&=-2\epsilon\int_{{\mathbb{R}^3_+}}  \langle z \rangle^{2\ell'} \partial_{x}g^{n}\big(\partial_{z}^{-1}g^{n}\partial_{z}\eta_{xz} \big)-
2\epsilon\int_{{\mathbb{R}^3_+}}  \langle z \rangle^{2\ell'} g^{n}\big(\partial_{z}^{-1}g^{n}\partial_{x}\partial_{z}\eta_{xz} \big)\\
& \leq  \frac{\epsilon}{8} \left\|\partial_{x}g^{n}\right\|^{2}_{L^{2}_{\ell'}(\mathbb{R}_{+}^{3})}
+
{\epsilon} \left\|g^{n}\right\|^{2}_{L^{2}_{\ell'}(\mathbb{R}_{+}^{3})}
+
{\epsilon} \left\|\big(\partial_{z}^{-1}g^{n}\partial_{z}\eta_{xz} \big)\right\|^{2}_{L^{2}_{\ell'}(\mathbb{R}_{+}^{3})}
\\
&\quad+
{\epsilon} \left\|\big(\partial_{z}^{-1}g^{n}\partial_{x}\partial_{z}\eta_{xz} \big)\right\|^{2}_{L^{2}_{\ell'}(\mathbb{R}_{+}^{3})}.
\end{align*}
Now we need to control the last two terms of the inequality. In fact, noticing that
$$  \left| \partial_{x}\partial_{z}\eta_{xz}\right|  \leq C \langle z \rangle^{-\ell-1},$$
we  conclude that  for $\frac{1}{2} < \ell' < \ell + \frac{1}{2}$,
\begin{align*}
\left\|\partial_{z}^{-1}g^{n}\partial_{x}\partial_{z}\eta_{xz} \right\|^{2}_{L^{2}_{\ell'}(\mathbb{R}_{+}^{3})} \leq C 
	\int_{{\mathbb{R}^3_+}}\langle z \rangle^{2(\ell'-\ell-1)} \left(\int_{0}^{z}g^{n}(t, x, y, \tilde{z}) d \tilde{z}\right)^2 dxdydz \leq C \left\|g^{n}\right\|^{2}_{L^{2}_{\ell'}(\mathbb{R}_{+}^{3})}.
\end{align*}
It can be checked straightforwardly that the same upper bound holds for the other term 
$$\left\|\partial_{z}^{-1}g^{n}\partial_{z}\eta_{xz} \right\|^{2}_{L^{2}_{\ell'}(\mathbb{R}_{+}^{3})}.$$ 
Then, we obtain  
\begin{align*}
2\epsilon\int_{{\mathbb{R}^3_+}}  \langle z \rangle^{2\ell'} g^{n}\big(\partial_{x}\partial_{z}^{-1}g^{n}\partial_{z}\eta_{xz} \big) \leq \frac{3\epsilon}{4} \left\|\partial_{x}g^{n}\right\|^{2}_{L^{2}_{\ell'}(\mathbb{R}_{+}^{3})} + C \left\|g^{n}\right\|^{2}_{L^{2}_{\ell'}(\mathbb{R}_{+}^{3})}.
\end{align*}
Analogously, we also have
\begin{align*}
2\epsilon\int_{{\mathbb{R}^3_+}}  \langle z \rangle^{2\ell'} g^{n}\big(\partial_{y}\partial_{z}^{-1}g^{n}\partial_{z}\eta_{yz} \big)
\leq \frac{3\epsilon}{4} \left\|\partial_{y}g^{n}\right\|^{2}_{L^{2}_{\ell'}(\mathbb{R}_{+}^{3})} + C \left\|g^{n}\right\|^{2}_{L^{2}_{\ell'}(\mathbb{R}_{+}^{3})}.
\end{align*}
Substituting these estimates  into \eqref{4.9}, we arrive at
\begin{align}
	\begin{split}
		&\frac{d}{dt}\left\|g^{n}\right\|^{2}_{L^{2}_{\ell'}(\mathbb{R}_{+}^{3})}
		+\left\|\partial_{z}g^{n}\right\|^{2}_{L^{2}_{\ell'}(\mathbb{R}_{+}^{3})}
		+\epsilon\left\|\partial_{x}g^{n}\right\|^{2}_{L^{2}_{\ell'}(\mathbb{R}_{+}^{3})}
		+\epsilon\left\|\partial_{y}g^{n}\right\|^{2}_{L^{2}_{\ell'}(\mathbb{R}_{+}^{3})} \\
		&\leq C \left\|g^{n}\right\|^{2}_{L^{2}_{\ell'}(\mathbb{R}_{+}^{3})}+ \sum_{i=1}^{6} \left|\int_{{\mathbb{R}^3_+}}  \langle z \rangle^{2\ell'} g^{n} M_{i}(g^{n}) \right|.
	\end{split}
	\label{4.10}
\end{align}
Next,  we deal with the overall  integral   terms  $$\sum_{i=1}^{7} \int_{{\mathbb{R}^3_+}}  \langle z \rangle^{2\ell'} g^{n} M_{i}(g^{n}),$$
which, for the sake of convenience,  is represented by $\sum_{i=1}^{7} N_{i}$.

For  $N_{1}$, by  the decay rate of $\eta_{xz}$,  $\eta_{yz}$, $\partial_{z}\eta_{xz}$ and $\partial_z\eta_{yz}$:
\begin{align*}
\left| \eta_{xz}\right| \leq C \langle z \rangle^{-\ell}, \quad \left| \eta_{yz}\right| \leq C \langle z \rangle^{-\ell},
\end{align*}
and
\begin{align*}
 \left| \partial_z\eta_{xz}\right| \leq C \langle z \rangle^{-\ell-1}, \quad \left| \partial_z\eta_{yz}\right| \leq C \langle z \rangle^{-\ell-1},
\end{align*}
then we infer from \eqref{4.1}
\begin{align*}
	N_{1}&\leq\left|\int_{{\mathbb{R}^3_+}}  \langle z \rangle^{2\ell'} g^{n} (u^s + {u})\big( g^{n}\eta_{xz} + \partial_z^{-1} g^{n}  \partial_z \eta_{xz}
	\big)  \right| \\
	&\quad+\left|\int_{{\mathbb{R}^3_+}}  \langle z \rangle^{2\ell'} g^{n} K(u^s + {u})\big(  g^{n} \eta_{yz} + \partial_z^{-1} g^{n}  \partial_z \eta_{yz} 
	\big)
	\big)  \right|\\
	&\leq C_{K}\left|\int_{{\mathbb{R}^3_+}}  \langle z \rangle^{2\ell'-\ell} (g^{n})^2  +\langle z \rangle^{2\ell'-\ell-1} (\partial_z^{-1} g^{n}) g^{n} \right|\\
	&\leq C_{K}\left(  \left\|g^{n}\right\|^{2}_{L^{2}_{\ell'}(\mathbb{R}_{+}^{3})}
	+\left\|\varphi\right\|^{2}_{H^{n}_{k+\ell}(\mathbb{R}_{+}^{3})}\right).
\end{align*}
For  $N_{2}$, we use the similar method  to get decay rate of $\eta_{zz}$ and $\partial_{z}\eta_{zz}$:
\begin{align*}
	 \eta_{zz} \leq
	C\langle z \rangle^{-1},
\end{align*}
and
\begin{align*}
\partial_{z}\eta_{zz} \leq
C\langle z \rangle^{-2},
\end{align*}
thus,
\begin{align*}
N_{2}&\leq C\bigg|\int_{{\mathbb{R}^3_+}}  \langle z \rangle^{2\ell'} g^{n} \big(  \eta_{zz} \partial_{z}g^{n}
+
 \partial_{z}\eta_{zz}
+
 \eta^{2}_{u_{zz}}
+
  \eta_{zz} \partial_z \eta_{zz}  (\partial_z^{-1} g^{n})
\big)  \bigg| \\
&\leq 
 C\left\|g^{n}\right\|_{L^{2}_{\ell'}(\mathbb{R}_{+}^{3})}
 \left( \left\|\partial_{z}g^{n}\right\|_{L^{2}_{\ell'}(\mathbb{R}_{+}^{3})}+ \left\|g^{n}\right\|_{L^{2}_{\ell'}(\mathbb{R}_{+}^{3})}
+\left\|\varphi\right\|_{H^{n}_{k+\ell}(\mathbb{R}_{+}^{3})}\right) \\
&\leq \frac{1}{4}\left\|\partial_{z}g^{n}\right\|_{L^{2}_{\ell'}(\mathbb{R}_{+}^{3})}^{2}+ C\left(  \left\|g^{n}\right\|^{2}_{L^{2}_{\ell'}(\mathbb{R}_{+}^{3})}
+\left\|\varphi\right\|^{2}_{H^{n}_{k+\ell}(\mathbb{R}_{+}^{3})}\right).
\end{align*}
Using the same algorithm with $N_{2}$, we have 
\begin{align*}
N_{3}+N_{4}&\leq \frac{\epsilon}{4}\left( \left\|\partial_{x}g^{n}\right\|_{L^{2}_{\ell'}(\mathbb{R}_{+}^{3})}^{2}
+
\left\|\partial_{y}g^{n}\right\|_{L^{2}_{\ell'}(\mathbb{R}_{+}^{3})}^{2}\right) + C\left(  \left\|g^{n}\right\|^{2}_{L^{2}_{\ell'}(\mathbb{R}_{+}^{3})}
	+\left\|\varphi\right\|^{2}_{H^{n}_{k+\ell}(\mathbb{R}_{+}^{3})}\right), \\
N_{6}&\leq C_{K}\left(  \left\|g^{n}\right\|^{2}_{L^{2}_{\ell'}(\mathbb{R}_{+}^{3})}
	+\left\|\varphi\right\|^{2}_{H^{n}_{k+\ell}(\mathbb{R}_{+}^{3})}\right).
\end{align*}
For  $N_{5}$, by the integration  by parts in the   $x$-variable and
$y$-variable, we arrive at
\begin{align*}
	N_{5}
	&\leq
	\left| \int_{{\mathbb{R}^3_+}}  \langle z \rangle^{2\ell'} g^{n}\left(-K\partial_{y}\partial_{xy}^{n} u+\partial_{y}\partial_{xy}^{n}\big(K(u^s+u)\big)\right) \right|
	\\
	&\leq \left| \int_{{\mathbb{R}^3_+}}  \langle z \rangle^{2\ell'} g^{n}\left( \partial_{y}K(u^s_z+u_z)\varphi\frac{\partial_{xy}^{n} u}{u_{z}^{s}+u_{z}}\right)\right|  
	+\left| \int_{{\mathbb{R}^3_+}}  \langle z \rangle^{2\ell'}  \partial_{y}g^{n} \left\lbrace  K\partial_{xy}^{n} u
	-\partial_{xy}^{n}\big(K(u^s+u)\big)\right\rbrace \right| 
\\
&\leq C\Big(\left\|g^{n}\right\|_{L^{2}_{\ell'}(\mathbb{R}_{+}^{3})}+ \left\|\partial_{y}g^{n}\right\|_{L^{2}_{\ell'}(\mathbb{R}_{+}^{3})}\Big) 
\left\|\langle z \rangle^{-k}\partial_{z}^{-1}g^{n}\right\|_{L^{2}_{\ell'}(\mathbb{R}_{+}^{3})}
\\
&\leq \frac{\epsilon}{4} 
\left\|\partial_{y}g^{n}\right\|_{L^{2}_{\ell'}(\mathbb{R}_{+}^{3})}^{2}+
 C\left(  \left\|g^{n}\right\|^{2}_{L^{2}_{\ell'}(\mathbb{R}_{+}^{3})}
 +\left\|\varphi\right\|^{2}_{H^{n}_{k+\ell}(\mathbb{R}_{+}^{3})}\right).
\end{align*}

Finally, in order to estimate  to $N_{7}$, we need to decompose $M_{7}(g^{n})$ 
  as follows
\begin{align*}
M_{7}^{1}(g^{n})=-\partial_{z}\left( \frac{\sum_{i=1}^{n}C^i_n\partial_{xy}^i {u} \, \partial^{n  -i}_{xy} \partial_{x}{u}}{u_{z}^{s}+u_{z}}
\right),
\end{align*}
\begin{align*}
M_{7}^{2}(g^{n})=-\partial_{z}	\left( 
	\frac{\sum_{i=1}^{n}C^i_n\partial_{xy}^i \big(K(u^s+u)\big)\, \partial^{n  -i}_{xy} \partial_{y}{u}}{u_{z}^{s}+u_{z}} 
\right),
\end{align*}
and
\begin{align*}
M_{7}^{3}(g^{n})=-\partial_{z}\left( \frac{\sum_{i=1}^{n}C^i_n\partial_{xy}^i {\varphi} \,\partial^{n  -i}_{xy} {w}}{u_{z}^{s}+u_{z}} \right).
\end{align*}
A straightforward calculation yields
\begin{align*}
-M_{7}^{1}(g^{n})-M_{7}^{2}(g^{n})=\sum_{i=1}C^i_n \partial_{xy}^{i}\big(u+K(u^s+u)\big)g^{n+1-i}
 +\sum_{i=1}C^i_n \partial_{xy}^{i}\big(\varphi+K(u^s_{z}+\varphi)\big)\partial_{z}^{-1}g^{n+1-i},
	\end{align*}
which implies
\begin{align*}
N_{7}^{1}+N_{7}^{2}&=-\sum_{i=1}^{n}\int_{{\mathbb{R}^3_+}}  \langle z  \rangle^{2\ell'} g^{n}(M_{7}^{1}(g^{n})+M_{7}^{2}(g^{n})) \leq 
C_{K} \sum^{n}_{i=1} \big\| g^{i}\big\|_{L_{\ell'}^{2}}\left( 1+\big\| g^{i}\big\|_{L_{\ell'}^{2}}\right) .
\end{align*}

Notice that $N_{7}^{3}$ can be written as
\begin{align*}
	N_{7}^{3}&=-\sum_{i=1}^{n}\int_{{\mathbb{R}^3_+}}  \langle z  \rangle^{2\ell'} g^{n}M_{7}^{3}(g^{n})\\
	&=-\sum_{i=1}^{n}\int_{{\mathbb{R}^3_+}}  \langle z  \rangle^{2\ell'} g^{n}\partial_{z}\Big\lbrace g^{i}+(\partial_{z}^{-1}g^{i})\eta_{zz} \partial^{n  -i}_{xy} {w} \Big\rbrace \\
	&=-\sum_{i=1}^{n}\int_{{\mathbb{R}^3_+}}  \langle z  \rangle^{2\ell'} g^{n}(\partial_{z}g^{i})\partial^{n  -i}_{xy} {w} -\sum_{i=1}^{n}\int_{{\mathbb{R}^3_+}}  \langle z  \rangle^{2\ell'} g^{n}g^{i}\partial^{n  -i}_{xy}(\partial_{x}u+\partial_{y}(K(u^s+u))) \\
	&\quad-\sum_{i=1}^{n}\int_{{\mathbb{R}^3_+}}  \langle z  \rangle^{2\ell'} g^{n}\left(\partial_{z}^{-1}g^{i}(\partial_{z}\eta_{zz}) \right)\partial^{n  -i}_{xy} {w} 
	\\
	&\quad-\sum_{i=1}^{n}\int_{{\mathbb{R}^3_+}}  \langle z  \rangle^{2\ell'} g^{n}\left(\partial_{z}^{-1}g^{i}(\eta_{zz}) \right)\partial^{n  -i}_{xy} (\partial_{x}u+\partial_{y}(K(u^s+u)))
	\\
	&\quad
	-\sum_{i=1}^{n}\int_{{\mathbb{R}^3_+}}  \langle z  \rangle^{2\ell'} g^{n}\left(g^{i}\eta_{zz}\right)\partial^{n  -i}_{xy} {w}.
\end{align*}
Here we only estimate the first term on the right-hand side of the above equality, the other terms can be obtained by the same argument, 
\begin{align*}
-\sum_{i=1}^{n}\int_{{\mathbb{R}^3_+}}  \langle z  \rangle^{2\ell'} g^{n}(\partial_{z}g^{i})\partial^{n  -i}_{xy} {w}
&=
\sum_{i=1}^{n}\int_{{\mathbb{R}^3_+}}  \langle z  \rangle^{2\ell'} \partial_{z}g^{n}(g^{i}\partial^{n  -i}_{xy} {w})-\sum_{i=1}^{n}\int_{{\mathbb{R}^3_+}}  \langle z  \rangle^{2\ell'} g^{n}g^{i}\partial^{n  -i}_{xy}(\partial_{x}u+\partial_{y}(K(u^s+u)))\\
& \leq \frac{1}{10}\big\| \partial_{z}g^{n}\big\|_{L^2_{\ell'}}^{2}+
 C\big\| g^{i}\partial^{n  -i}_{xy} {w}\big\|_{L^2_{\ell'}}^{2}+
 C\big\| g^{n}\big\|_{L^2_{\ell'}}^{2}+
 C\big\| g^{i}\partial^{n  -i}_{xy} (\partial_{x}u+\partial_{y}v)\big\|_{L^2_{\ell'}}^{2}.
\end{align*}
Indeed, the terms 
$$\big\| g^{i}\partial^{n  -i}_{xy} {w}\big\|_{L^2_{\ell'}}^{2}$$
and
$$\big\| g^{i}\partial^{n  -i}_{xy} (\partial_{x}u+\partial_{y}(K(u^s+u)))\big\|_{L^2_{\ell'}}^{2}$$
 can be controlled by  the standard Sobolev-type estimates. More precisely, we have the following:\\
\textbf{Claim 4.1}:
\begin{align*}
\big\| g^{i}\partial^{n  -i}_{xy} {w}\big\|_{L^2_{\ell'}}^{2}
+
\big\| g^{i}\partial^{n  -i}_{xy} (\partial_{x}u+\partial_{y}(K(u^s+u)))\big\|_{L^2_{\ell'}}^{2}
 \leq C_{K} \left(\sum^{n}_{i=1} \big\| g^{i}\big\|_{L_{\ell'}^{2}}+
	\|\tvae\|^{2}_{H^{m}_{k+\ell}}\right).
\end{align*}
Assuming that Claim holds, which will be later verified,  then we get  immediately
\begin{align*}
N_{7}^{3}\leq \frac{1}{2}\big\| \partial_{z}g^{n}\big\|_{L^2_{\ell'}}^{2}+C_{K} \left(\sum^{n}_{i=1} \big\| g^{i}\big\|_{L_{\ell'}^{2}}+
\|\tvae\|^{2}_{H^{m}_{k+\ell}}\right).
\end{align*}
Combining $N_{1}-N_{7}$ with \eqref{4.10}, we see that
\begin{align}
	\begin{split}
		&\frac{d}{dt}\left\|g^{n}\right\|^{2}_{L^{2}_{\ell'}(\mathbb{R}_{+}^{3})}
		+\left\|\partial_{z}g^{n}\right\|^{2}_{L^{2}_{\ell'}(\mathbb{R}_{+}^{3})}
		+\epsilon\left\|\partial_{x}g^{n}\right\|^{2}_{L^{2}_{\ell'}(\mathbb{R}_{+}^{3})}
		+\epsilon\left\|\partial_{y}g^{n}\right\|^{2}_{L^{2}_{\ell'}(\mathbb{R}_{+}^{3})} \\
		&\leq C\left(\sum^{m}_{n=1}\left\| g^{n}\right\|_{L^{2}_{\ell'}(\mathbb{R}_{+}^{3})}^2
		+
		\left\| \varphi \right\|_{H^{m}_{k+\ell'}(\mathbb{R}_{+}^{3})}^2\right).
	\end{split}
\end{align}
The
proof of Lemma \ref{gn} is thus completed. 

\textbf{Proof of Claim 4.1}:We must use different techniques depending on the value range of $i$.\\
The term $\big\| g^{i}\partial^{n  -i}_{xy} {w}\big\|_{L^2_{\ell'}}^{2}$ with $i\leq n-3$. In light of $k+\ell-1 >\frac{1}{2}$ and $\|\tvae\|^{2}_{H^{m}_{k+\ell}}\leq 1$, we have
\begin{align*}
\big\| g^{i}\partial^{n  -i}_{xy} {w}\big\|_{L^2_{\ell'}}^{2}
&\leq \big\| g^{i}\big\|_{L_{xy}^{\infty}(L_{\ell'}^{2}(\mathbb{R}_{+}))}
\big\| \partial^{n  -i}_{xy} w\big\|_{L_{z}^{\infty}(L_{xy}^{2})}\\
& \leq \left( \big\| g^{i}\big\|_{L_{\ell'}^{2}}+\big\|\partial_{xy}^{1} g^{i}\big\|_{L_{\ell'}^{2}}+\big\| \partial_{x}\partial_{y}g^{i}\big\|_{L_{\ell'}^{2}}\right) 
\sup_{z \in \mathbb{R}_+} \int_{-\infty}^{+\infty}\int_{-\infty}^{+\infty} \left|\int_{0}^{z}\partial_{xy}^{n-i} \big(\partial_{x}u+\partial_{y}(K(u^s+u))\big) d\tilde{z}\right|^2 dz\\
& \leq \left( \big\| g^{i}\big\|_{L_{\ell'}^{2}}+\big\|\partial_{xy}^{1} g^{i}\big\|_{L_{\ell'}^{2}}+\big\| \partial_{x}\partial_{y}g^{i}\big\|_{L_{\ell'}^{2}}\right) 
\sup_{z \in \mathbb{R}_+} \int_{-\infty}^{+\infty}\int_{-\infty}^{+\infty} \left|\int_{0}^{+\infty}\big|\partial_{xy}^{n-i} \big(\partial_{x}u+\partial_{y}(K(u^s+u))\big)\big| d\tilde{z}\right|^2 dz
\\
&\leq C \left(\sum^{n-2}_{i=1} \big\| g^{i}\big\|_{L_{\ell'}^{2}}+\big\| g^{n-1}\big\|_{L_{\ell'}^{2}}+\big\| g^{n}\big\|_{L_{\ell'}^{2}}+
\|\tvae\|^{2}_{H^{m-1}_{k+\ell}}\right) \\
&\quad \times
\sup_{z \in \mathbb{R}_+} \int_{-\infty}^{+\infty}\int_{-\infty}^{+\infty} \left|\int_{0}^{+\infty}\big|\partial_{xy}^{n-i} \big(\partial_{x}u+\partial_{y}(K(u^s+u))\big)\big| d\tilde{z}\right|^2 dz\\
&\leq C \left(\sum^{n-2}_{i=1} \big\| g^{i}\big\|_{L_{\ell'}^{2}}+\big\| g^{n-1}\big\|_{L_{\ell'}^{2}}+\big\| g^{n}\big\|_{L_{\ell'}^{2}}+
\|\tvae\|^{2}_{H^{m-1}_{k+\ell}}\right) \\
&\quad \times
\sup_{z \in \mathbb{R}_+} \int_{-\infty}^{+\infty}\int_{-\infty}^{+\infty} \left|\int_{0}^{+\infty}\langle z  \rangle^{-k-\ell'+1}\langle z  \rangle^{k+\ell'-1}\big|\partial_{xy}^{n-i} \big(\partial_{x}u+\partial_{y}(K(u^s+u))\big)\big| d\tilde{z}\right|^2 dz
\\
&\leq C_{K} \left(\sum^{n-1}_{i=1} \big\| g^{i}\big\|_{L_{\ell'}^{2}}^2+
\|\tilde{\varphi}\|^{2}_{H^{m}_{k+\ell}}\right),
\end{align*}
where we have used the fact :
\begin{align*}
&\partial_x g^{i}
=
 \partial_{z}\bigg(\frac{\partial_{x}\partial_{xy}^{i} u}{u_{z}^{s}+u_{z}}\bigg) 
-g^{i} \eta_{xz} - \partial_z^{-1} g^{i}  \partial_z \eta_{xz}
, \\
&\partial_y g^{i} 
=
\partial_{z}\bigg(\frac{\partial_{y}\partial_{xy}^{i} u}{u_{z}^{s}+u_{z}}\bigg) 
-g^{i} \eta_{yz} - \partial_z^{-1} g^{i}  \partial_z \eta_{yz}
,  \\
&(\partial_{x}+\partial_{y})g^{i}=g^{i+1}+
\left( 
-g^{i} (\eta_{xz}+g^{i} \eta_{yz}) - \partial_z^{-1} g^{i}  (\partial_z \eta_{xz}+ \partial_z \eta_{yz})
\right),\\
&\frac{1}{2}\partial_{x}\partial_{y}g^{i}=g^{i+2}+g^{i+1}+
\left( 
-g^{i+e_{1}} \eta_{xz} - \partial_z^{-1} g^{i+e_{1}}  \partial_z \eta_{xz}
\right) \\
&\qquad\qquad~~+\left( 
-g^{i+e_{1}} \eta_{yz} - \partial_z^{-1} g^{i+e_{1}}  \partial_z \eta_{yz}
\right)\\
&\qquad\qquad~~+\partial_{x}\left( 
-g^{i} \eta_{xz} - \partial_z^{-1} g^{i}  \partial_z \eta_{xz}
\right)\\
&\qquad\qquad~~+\partial_{y}\left( 
-g^{i} \eta_{yz} - \partial_z^{-1} g^{i}  \partial_z \eta_{yz}
\right).
\end{align*}
For the term $\big\| g^{i}\partial^{n  -i}_{xy} {w}\big\|_{L^2_{\ell'}}^{2}$ with $i\geq n-3$,  by using Hardy inequality \eqref{hardy2},  \eqref{infty2},  and $\|\tvae\|^{2}_{H^{m}_{k+\ell}}\leq 1, m\geq 6$, we obtain
\begin{align*}
\big\| g^{i}\partial^{n  -i}_{xy} {w}\big\|_{L^2_{\ell'}}^{2}
&\leq
\big\| g^{i}\big\|_{L^2_{\ell'}(\mathbb{R}_{+}^{3})}^{2}
\big\| \partial^{n  -i}_{xy} {w}\big\|_{{L^\infty}(\mathbb{R}_{+}^{3})}^{2}\\
&\leq
C\big\| g^{i}\big\|_{L^2_{\ell'}(\mathbb{R}_{+}^{3})}^{2}
\big\| u+K(u^s+u)\big\|_{{H^{n+3-i}_{\frac{1}{2}+\delta}(\mathbb{R}_{+}^{3})}}^{2}\\
&\leq
C_{K}\big\| g^{i}\big\|_{L^2_{\ell'}(\mathbb{R}_{+}^{3})}^{2}
\left( 1+\big\| \varphi\big\|_{{H^{n+3-i}_{\frac{3}{2}+\delta}(\mathbb{R}_{+}^{3})}}\right) ^{2}\\
&\leq C_{K}\left( \big\| g^{n-1}\big\|_{L^2_{\ell'}(\mathbb{R}_{+}^{3})}^{2}
+\big\| g^{n}\big\|_{L^2_{\ell'}(\mathbb{R}_{+}^{3})}^{2}\right).
\end{align*}
The term $\big\| g^{i}\partial^{n  -i}_{xy} \big(\partial_{x}u+\partial_{y}(K(u^s+u))\big)\big\|_{L^2_{\ell'}}^{2}$, with $i\leq n-2$ and $i\geq n-2$,
 is simpler because there is no longer any the vertical component of the velocity field $w$. Hence, 
\begin{align*}
	\big\| g^{i}\partial^{n  -i}_{xy} \big(\partial_{x}u+\partial_{y}(K(u^s+u))\big)\big\|_{L^2_{\ell'}}^{2}
	\leq C \left(\sum^{n}_{i=1} \big\| g^{i}\big\|_{L_{\ell'}^{2}}^2+
	\|\tvae\|^{2}_{H^{m}_{k+\ell}}\right).
\end{align*}
\end{proof}

\section{Existence of the solution}\label{s5}
The aim of this subsection is to construct the following  energy estimate for the sequence of approximate solutions and  later prove the   existence part of main Theorem \ref{e-u-s}. 
\begin{theorem}\label{ex1}
Assume the condition (H) holds.	Let   $m \geq  6$ be an even integer, $k > 1$, $0< \ell< \frac{1}{2}$, $\frac{1}{2} < \ell' < \ell + \frac{1}{2}$, and $k + \ell > \frac {3}{2}$.
	Assume that the initial date $u^s_0$ satisfy Proposition \ref{shear-profile}. In addition, we also suppose 
	$\tilde{u}_{0} \in H^{m+3}_{k+\ell'-1}(\mathbb{R}^3_+)$ satisfies the compatibility conditions \eqref{uv t=0}-\eqref{uv4 t=0}.
And $K(x, y)$ is supposed to satisfy that
$$\big\|K\big\|_{W^{m+1, \infty}(\mathbb{R}^2)} < \infty.$$
	If $\tvae \in L^\infty([0, T]; H^{m+2}_{k+\ell}(\mathbb{R}^3_+))$  solves \eqref{vv-regular-shear-prandtl} and  satisfies the following a priori condition
	\begin{align}
		\|\varphi\|_{ L^{\infty} ([0, T]; H^{m}_{k+\ell}(\mathbb{R}^3_+))}\leq  \zeta,
	\end{align}
with
$$0<\zeta \leq 1, \quad C_m \zeta \leq \frac{\tilde {c_{1}}}{2},$$
then there exists  constants $C_{T}, \tilde{C}_{T}$ such that 
\begin{align}
	\|\tvae\|_{L^{\infty}(H^{m}_{k+\ell}(\mathbb{R}_{+}^{3}))}
	\leq C_{T}\left\| \tu_{0}\right\| _{H^{m+1}_{k+\ell'-1}(\mathbb{R}^{3}_{+})} ,
	\label{5.0}
\end{align}
where $C_{T}>0$ is increasing with respect to $0 < T \leq T_{1}$ and independent of $0 < \epsilon \leq 1$. 
\end{theorem}

Let's go back to the notations with tilde and the sub-index $\epsilon$ and $S^{m, \epsilon}$ is the function defined by $\tue$.
According to Lemmas \ref{E2} and \ref{gn}, we know
\begin{align}
\begin{split}
&\frac{d}{dt}\|\tvae\|^{2}_{H^{m, m-1}_{k+\ell}(\mathbb{R}_{+}^{3})}+\|\partial_{z}\tvae\|^{2}_{H^{m, m-1}_{k+\ell}(\mathbb{R}_{+}^{3})}+\epsilon\left( \|\partial_{x}\tvae\|^{2}_{H^{m, m-1}_{k+\ell}(\mathbb{R}_{+}^{3})}+\|\partial_{y}\tvae\|^{2}_{H^{m, m-1}_{k+\ell}(\mathbb{R}_{+}^{3})}\right) \\
& \leq C_{1}  \|\tvae\|^{2}_{H^{m}_{k+\ell}(\mathbb{R}_{+}^{3})},
\end{split}
\label{5.1}
\\
\begin{split}
&\frac{d}{dt}\sum^{m}_{n=1}\left\|g^{n,\epsilon}\right\|^{2}_{L^{2}_{\ell'}(\mathbb{R}_{+}^{3})}
+\sum^{m}_{n=1}\left\|\partial_{z}g^{n, \epsilon}\right\|^{2}_{L^{2}_{\ell'}(\mathbb{R}_{+}^{3})}
+\epsilon\sum^{m}_{n=1}\left( \left\|\partial_{x}g^{n, \epsilon}\right\|^{2}_{L^{2}_{\ell'}(\mathbb{R}_{+}^{3})}
+\left\|\partial_{y}g^{n, \epsilon}\right\|^{2}_{L^{2}_{\ell'}(\mathbb{R}_{+}^{3})} \right) \\
&\leq C_{2}\left( \sum^{m}_{n=1}\left\|g^{n,\epsilon}\right\|^{2}_{L^{2}_{\ell'}(\mathbb{R}_{+}^{3})}+\|\tvae\|^{2}_{H^{m}_{k+\ell}(\mathbb{R}_{+}^{3})}\right) .
\end{split}
\label{5.2}
\end{align}

\begin{lemma}\label{S-leq-uv}
For the initial data, we have 
\begin{align*}
S^{m, \epsilon}(g, \varphi)(0)=\sum^{m}_{n=1}\left\|g^{n,\epsilon}(0)\right\|^{2}_{L^{2}_{\ell'}(\mathbb{R}_{+}^{3})}+\|\tvae(0)\|^{2}_{H^{m, m-1}_{k+\ell}(\mathbb{R}_{+}^{3})} \leq C\left\| \tu_{0}\right\| _{H^{m+1}_{k+\ell'-1}(\mathbb{R}^{3}_{+})},
\end{align*} 
where a constant $C>0$ is independent of $\epsilon$.
\begin{proof}
Recalling the definition of $g^{n, \epsilon}$, we know for any $1 \leq n \leq m$,
\begin{align*}
g^{n,\epsilon}&=\bigg(\frac{\partial_{xy}^{n} \tue}{u_{z}^{s}+\tvae}\bigg)_{z} \\
&=\frac{\partial_{xy}^{n} \partial_{z}\tue}{u^{s}_{z}+\tvae}
-
\bigg(\frac{\partial_{xy}^{n} \tue}{u^{s}_{z}+\tvae} \eta_{zz}\bigg).
\end{align*}
Using $\tilde{u}^{\epsilon}(0)=\tilde{u}_{0}$, we arrive at
\begin{align*}
\left\| g^{n,\epsilon}(0)\right\| _{L^2_{\ell'}(\mathbb{R}^{3}_{+})}
&\leq
2\left\| \frac{\partial_{xy}^{n} \partial_{z}\tu_{0}}{u^{s}_{0,z}+\tvae_{0}}\right\| _{L^2_{\ell'}(\mathbb{R}^{3}_{+})}
+2\left\| \frac{\partial_{xy}^{n} \tu_{0}}{u^{s}_{0,z}+\tvae_{0}} \eta_{zz}(0)\right\| _{L^2_{\ell'}(\mathbb{R}^{3}_{+})}\\
&\leq C\left\| {\partial_{xy}^{n} \partial_{z}\tu_{0}}\right\| _{L^2_{k+\ell'}(\mathbb{R}^{3}_{+})}
+C\left\| \partial_{xy}^{n} \tu_{0}\right\| _{L^2_{k+\ell'-1}(\mathbb{R}^{3}_{+})}\\
&\leq C\left\| \tu_{0}\right\| _{H^{m+1}_{k+\ell'-1}(\mathbb{R}^{3}_{+})},
\end{align*}
which yields 
\begin{align*}
	S^{m, \epsilon}(g, \varphi)(0) \leq C\left\| \tu_{0}\right\| _{H^{m+1}_{k+\ell'-1}(\mathbb{R}^{3}_{+})}.
\end{align*}
\end{proof}
\end{lemma}

\begin{lemma}\label{varphi-leq-g}
For any $1 \leq n \leq m$, the following estimate holds
\begin{align*}
\left\| \partial^n_{xy}\tvae\right\| _{L^2_{k+\ell'}(\mathbb{R}^{3}_{+})}^{2}
\leq C\left\| g^{m,\epsilon}\right\| _{L^2_{\ell'}(\mathbb{R}^{3}_{+})}^{2},
\end{align*}
where a constant $C>0$ is independent of $\epsilon$.
\begin{proof}
By the definition of $g^{m}$, we see that
\begin{align*}
\partial_{xy}^{m}\tue(t, x, y, z)= (u_{z}^{s}+\tvae)\int_{0}^{z}g^{m, \epsilon} d\tilde{z}  .
\end{align*}
A direct computation gives
\begin{align*}
\partial_{xy}^{m}\tvae= (u_{zz}^{s}+\tvae_{z})\int_{0}^{z}g^{m, \epsilon} d\tilde{z} 
-
 (u_{z}^{s}+\tvae)g^{m, \epsilon}   ,
\end{align*}
which implies
\begin{align*}
	\left\| \partial^n_{xy}\varphi\right\| _{L^2_{k+\ell'}(\mathbb{R}^{3}_{+})}
	\leq C\left\| g^{m,\epsilon}\right\| _{L^2_{\ell'}(\mathbb{R}^{3}_{+})}.
\end{align*}
\end{proof}
\end{lemma}

\begin{proof}[{\bf  Proof of Theorem \ref{ex1}}]
Summing up \eqref{5.1} and \eqref{5.2}, and integrating the resulting equation over $[0, t]$, we obtain by employing Lemma \ref{S-leq-uv}
\begin{align*}
&\|\tvae\|^{2}_{H^{m, m-1}_{k+\ell}(\mathbb{R}_{+}^{3})}+\sum^{m}_{n=1}\left\|g^{n,\epsilon}\right\|^{2}_{L^{2}_{\ell'}(\mathbb{R}_{+}^{3})}\\
&\leq S^{m, \epsilon}(g, \varphi)(0)+\int_{0}^{t} \left\lbrace C_{1}  \|\tvae\|^{2}_{H^{m}_{k+\ell}(\mathbb{R}_{+}^{3})}
+C_{2}\left( \sum^{m}_{n=1}\left\|g^{n,\epsilon}\right\|^{2}_{L^{2}_{\ell'}(\mathbb{R}_{+}^{3})}+\|\tvae\|^{2}_{H^{m}_{k+\ell}(\mathbb{R}_{+}^{3})}\right)\right\rbrace \\
&\leq e^{C_{2}t} S^{m, \epsilon}(g, \varphi)(0)+
\int_{0}^{t}  C_{1} e^{C_{2}(t-\tau)} \|\tvae\|^{2}_{H^{m}_{k+\ell}(\mathbb{R}_{+}^{3})} d\tau \\
&\leq C_{3}e^{C_{2}t}\left\| \tu_{0}\right\| _{H^{m+1}_{k+\ell'-1}(\mathbb{R}^{3}_{+})} +
\int_{0}^{t}  C_{1} e^{C_{2}(t-\tau)} \|\tvae\|^{2}_{H^{m}_{k+\ell}(\mathbb{R}_{+}^{3})} d\tau,
\end{align*}
which, together with Lemma \ref{varphi-leq-g}, yields
\begin{align*}
	\|\tvae\|^{2}_{H^{m}_{k+\ell}(\mathbb{R}_{+}^{3})}
	\leq C_{3}\left\| \tu_{0}\right\| _{H^{m+1}_{k+\ell'-1}(\mathbb{R}^{3}_{+})} e^{(C_{2}+C_{1})t}.
\end{align*}
This proves Theorem \ref{ex1}.
\end{proof}

With help of Theorem \ref{ex1}, we will finish the proof of the   existence part in Theorem \ref{e-u-s}.  
\begin{proof}[{\bf  Proof of    Existence Part in Theorem \ref{e-u-s}}]
Our first purpose is to show that the solution $\tvae$ in $[0, T^\epsilon]$ can be
extended to $[0, T_{1}]$ by recurrence, where $T_{1}$ is the lifespan of shear flow. Then we verify  convergence and consistency of the solution by the standard regularizing initial data argument.

Fixing  $\epsilon \in (0, 1]$, in light of Theorem \ref{epsilon-existence}, 
the initial boundary value problem  \eqref{vv-regular-shear-prandtl}  admits a unique solution 
\begin{align*}
	\tvae \in L^\infty([0, T^\epsilon]; H^{m+2}_{k+\ell}(\mathbb{R}^3_+)) \, ,
\end{align*}
for any $\partial_{z}\tilde{u}_{0}^\epsilon \in H^{m+2}_{k+\ell}(\mathbb{R}^3_+)$ and $0<\epsilon\leq \epsilon_{0}$, which satisfies
\begin{align*}
	\|\tvae\|_{L^{\infty}([0, T^\epsilon]; H^{m}_{k+\ell}(\mathbb{R}^3_+))}
	\leq \frac{4}{3}\|\tvae_{0}\|_{H^{m}_{k+\ell}(\mathbb{R}^3_+)} \leq
	2\|\tilde{u}_{0}\|_{H^{m+1}_{k+\ell-1}(\mathbb{R}^3_+)}.
\end{align*}
Choosing $\zeta_0$ so that
\begin{align*}
\max \{2, C_{T_1}\} \zeta_0 \leq \frac{\zeta}{2},
\end{align*}
and taking $\tvae(T^\epsilon)$ as an initial datum for  the equation \eqref{vv-regular-shear-prandtl},  Theorem \ref{epsilon-existence} ensures  that there exits a time $(T^\epsilon)'>0$, 
which is given by \eqref{Tm} with $\bar{\zeta}=\frac{\zeta}{2}$, such that the initial boundary value problem  \eqref{vv-regular-shear-prandtl} admits a unique solution
\begin{align*}
	(\tvae)' \in L^\infty([T^\epsilon, T^\epsilon+(T^\epsilon)']; H^{m}_{k+\ell}(\mathbb{R}^3_+)) \, ,
\end{align*}
 which satisfies
\begin{align*}
	\|(\tvae)'\|_{L^{\infty}([T^\epsilon, T^\epsilon+(T^\epsilon)']; H^{m}_{k+\ell}(\mathbb{R}^3_+))}
	\leq \frac{4}{3}\|\tvae_{0}\|_{H^{m}_{k+\ell}(\mathbb{R}^3_+)} \leq \zeta.
\end{align*}
We concatenate this solution $(\tvae)'$ with  a original solution  $\tvae$ to derive a  new solution  $\tvae \in L^\infty([0, T^\epsilon+(T^\epsilon)']; H^{m}_{k+\ell}(\mathbb{R}^3_+)) $
 which satisfies
\begin{align*}
	\|(\tvae)'\|_{L^{\infty}([0, T^\epsilon+(T^\epsilon)']; H^{m}_{k+\ell}(\mathbb{R}^3_+))}
	\leq \zeta.
\end{align*}
If $T^\epsilon+(T^\epsilon)' < T_{1}$, we use \eqref{5.0} and Theorem  \ref{ex1} to $\varphi$ with $T^\epsilon+(T^\epsilon)'$, which implies
\begin{align}
		\|(\tvae)'\|_{L^{\infty}([0, T^\epsilon+(T^\epsilon)']; H^{m}_{k+\ell}(\mathbb{R}^3_+))}
	\leq C_{T}\left\| \tu_{0}\right\| _{H^{m+1}_{k+\ell'-1}(\mathbb{R}^{3}_{+})} \leq 
	\frac{\zeta}{2},
\end{align}

Now we may take $T^\epsilon+(T^\epsilon)'$  as an initial datum and  proceed as in the first paragraph of the proof. Repeating this process $r$ times until $T^\epsilon+r((T^\epsilon)') = T_{1}$. In other words, the solution $\tvae$ is
extended to $[0, T_{1}]$,   and then the lifespan of approximate solution is equal to  that of shear flow if the initial datum $\tu_{0}$ is small enough.  Furthermore, we obtain 
\begin{align*}
\|\tvae(t)\|_{ H^{m}_{k+\ell}(\mathbb{R}^3_+))} \leq \zeta, \qquad t\in[0, T_{1}], 
\end{align*}
for $m \geq 6$ and $ 0< \epsilon \leq \epsilon_0$. 
Applying the Sobolev inequality, we have, for $0<\delta<1$
\begin{align*}
\|\tvae\|_{Lip ([0, T_{1}]; C^{2,  \delta}(\mathbb{R}^3_+))}\leq M< + \infty.
\end{align*}

Then there exist   a sequence a sequence $\{\epsilon_k\}_{k\in \mathbb{N}} \subseteq (0,1)$ with $\lim \limits_{k \rightarrow + \infty}
\epsilon_k = 0^{+}$ such that as $\epsilon_k \rightarrow 0^{+}$, for $0<\delta'<\delta$,
\begin{align*}
\tilde{\varphi}^{\epsilon_k} \rightarrow \tilde{\varphi} ~~ \text{locally strong in }~~C^{0}([0, T_{1}]; C^{2, \delta'}(\mathbb{R}^3_+)),
\end{align*}
and
\begin{align*}
\partial_t \tilde{\varphi} \in L^\infty ([0, T_1]; H^{m-2}_{k+\ell}(\mathbb{R}^3_+)), \quad
\tilde{\varphi} \in L^\infty ([0, T_1]; H^{m}_{k+\ell}(\mathbb{R}^3_+)),
\end{align*}
with
\begin{align*}
	\|\tilde{\varphi}\|_{L^\infty ([0, T_1]; H^{m}_{k+\ell}(\mathbb{R}^3_+))} \le \zeta.
\end{align*}

In fact, we also have by the Hardy inequality \ref{hardy2} and the condition $k+l-1 > \frac{1}{2} $
\begin{align*}
\tu= \partial_{z}^{-1}\varphi \in L^\infty ([0, T_1]; H^{m}_{k+\ell-1}(\mathbb{R}^3_+)).
\end{align*}
Using the  uniform convergence of $\partial_{xy}^{k}u^{\epsilon_{k}}$, we  have the
pointwise convergence of $w^{\epsilon_k}$: as $\epsilon_k \rightarrow 0^{+}$,
\begin{equation}  
\tilde{w}^{\epsilon_{k}}=-\int^{z}_{0}\partial_{x}\tilde{u}^{\epsilon_{k}} d\tilde{z}-\int^{z}_{0}\partial_{y}\big(K(u^{s,\epsilon^k}+\tilde{u}^{\epsilon_{k}})\big) d\tilde{z}
\rightarrow -\int^{z}_{0}\partial_{x}\tilde{u} d\tilde{z}-\int^{z}_{0}\partial_{y}\big(K(u^s+\tilde{u})\big) d\tilde{z}=\tilde{w},
\end{equation}
which yields
\begin{align*}
\tilde{w}=-\int^{z}_{0}\partial_{x}\tilde{u} d\tilde{z}-\int^{z}_{0}\partial_{y}\big(K(u^s+\tilde{u})\big) d\tilde{z}  \in L^\infty ([0, T_1];  L^\infty(\mathbb{R}_{+, z}); (H^{m}(\mathbb{R}^2)H^{-1}(\mathbb{R}_{x}) \cup H^{m}(\mathbb{R}^2)H^{-1}(\mathbb{R}_{y})).
\end{align*}
Thus, 
\begin{align*}
	\tilde{w} \in L^\infty ([0, T_1];  L^\infty(\mathbb{R}_{+, z}); (H^{m-1}(\mathbb{R}^2_{xy})) ).
\end{align*}
Now we have proven that  $\tilde{\varphi}$ is a classical solution to the following regularized vorticity system
\begin{equation}
	\begin{cases} \partial_t\tilde{\varphi}  + (u^s + \tilde{u} ) \partial_x\tilde{\varphi}  +K(u^s + \tilde{u} ) \partial_y\tilde{\varphi} + \tilde{w}  \partial_z\big((u^s_{z}, v^{s}_{z}) +  \tilde{\varphi} \big)
		+\partial_{y}K (u^s+\tu)\partial_{z}(u^s+\tu)	\\
		\qquad = \partial^2_z\tilde{\varphi}  +\epsilon \partial^2_x\tilde{\varphi} +\epsilon \partial^2_y\tilde{\varphi}  , \\
		\partial_{z}\tilde{\varphi} |_{z=0} =0 , \\
		\tvae|_{t=0} =\tilde{\varphi}_{0},
	\end{cases}
\end{equation}
and $(\tu,  \tw) $  is a classical solution to equation \eqref{shear-prandtl}. 
This immediately yields the  existence
of classical solution $(u, w)=(u^s+\tu,  \tw)$ to the Prandtl equation \eqref{3d prandtl constant outflow }. We have completed the proof of the    existence part in Theorem \ref{e-u-s}.  
\end{proof}

\section{Uniqueness and stability}\label{s6}
In this section, we are  devoted to the proof of the  stability part in Theorem \ref{e-u-s}, and thus the uniqueness of   solution obtained   will
follow immediately.  Let $\tu^{1}$ and $\tu^{2}$ be two solutions. Denote  $\bar{u}=\tu^{1}-\tu^{2}$ and $\bar{w}=\tw^{1}-\tw^{2}$. From $\eqref{shear-prandtl}$, we have 
\begin{equation}
	\begin{cases} 
		\partial_t \bar{u} + (u^s + \tilde{u}^{1}) \partial_x\bar{u} +K(u^s + \tilde{u}^{1}) \partial_y\bar{u}+ \tilde{w}^{1}\partial_{z}\bar{u} \\
		= \partial^2_z \bar{u}-\bar{u}\partial_{x}\tu^{2}
		-K\bar{u}\partial_{y}\tu^{2}-\bar{w}(u^{s}_{z}+\tu_{z}^{2}), \\			
		\partial_x\tilde{u} +\partial_y\big(K(u^s+\tilde{u})\big) +\partial_z\tilde{w} =0, \\
		(\bar{u},  \bar{w})|_{z=0} =0 , \\
		\bar{u}|_{t=0} =\tu^{1}_{0}-\tu^{2}_{0} .
	\end{cases}
	\label{6.1}
\end{equation}
Denote the vorticity $\bar\varphi=\tilde\varphi^{(1)}-\tilde\varphi^{(2)}=\partial_{z}\bar{u}$, we also have

\begin{equation}
	\begin{cases} 
		\partial_t \bar{\varphi}-\partial^2_z \varphi + (u^s + \tilde{u}^{1})
		 \partial_x\bar{\varphi} +K(u^s + \tilde{u}^{1}) \partial_y\bar{\varphi}+ \tilde{w}^{1}\partial_{z}\bar\varphi\\
		=
		-\bar{u}\partial_{x}\tilde{\varphi}^{(2)}
		-K\bar{u}\partial_{y}\tilde{\varphi}^{(2)}-\bar{w}\big(u^{s}_{zz}  +\tilde{\varphi}_{z}^{(2)}\big)
	\\
	\quad	+\partial_{y}K(u^s + \tilde{u}^{1})\bar\varphi
	+\partial_{y}K \bar\varphi(u^s_{z}+\tilde\varphi^{(2)})
		, \\
		\partial_{z}\varphi|_{z=0} =0 , \\
		\bar{\varphi}|_{t=0} =\tilde{\varphi}_{0}^{(1)}-\tilde{\varphi}_{0}^{(2)} .
	\end{cases}
	\label{6.2}
\end{equation}

\begin{lemma}\label{stability1}
Under the hypotheses of Theorem \ref{e-u-s}, let $\tu^{1}$ and $\tu^{2}$ be two solutions  with respect to the initial data $\tu^{1}_{0}$, $\tu^{2}_{0}$, then we have
\begin{align}
		\frac{d}{dt}\|\bar{\varphi}\|^{2}_{H^{m-2, m-3}_{k+\ell}(\mathbb{R}_{+}^{3})}+\|\partial_{z}\bar{\varphi}\|^{2}_{H^{m-2, m-3}_{k+\ell}(\mathbb{R}_{+}^{3})}
		 \leq C  \|\bar{\varphi}\|^{2}_{H^{m}_{k+\ell}(\mathbb{R}_{+}^{3})},
	\label{uu1} 
\end{align}
where  constant $C$ depends on the norm of $\tilde{\varphi}^{(1)},\tilde{\varphi}^{(2)}$ in $L^\infty([0, T]; H^m_{k+\ell}(\mathbb{R}^3_+))$.
\end{lemma}

\begin{proof}[{\bf  Outline of Proof of Lemma \ref{stability1}}]

 Applying $\partial^{\alpha}$ to the  equation \eqref{6.2}, for $|\alpha|=\alpha_{1}+\alpha_{2}+\alpha_{3} \leq m-2$, $\alpha_{1}+\alpha_{2}\leq m-3$, we obtain
\begin{align}
\begin{split}
&\partial^{\alpha}\left( \partial_t \bar{\varphi}-
\partial^2_z \bar{\varphi}\right)   +\partial^{\alpha}\left( (u^s + \tilde{u}^{1})
\partial_x\bar{\varphi} +K(u^s + \tilde{u}^{1}) \partial_y\bar{\varphi}+ \tilde{w}^{1}\partial_{z}\bar\varphi\right) \\
&=
\partial^{\alpha}\left\lbrace -\bar{u}\partial_{x}\tilde{\varphi}^{(2)}
-K\bar{u}\partial_{y}\tilde{\varphi}^{(2)}-\bar{w}u^{s}_{zz} \right\rbrace  \\
&\quad+ \partial^{\alpha}\left\lbrace \partial_{y}K(u^s + \tilde{u}^{1})\bar\varphi
+\partial_{y}K \bar\varphi(u^s_{z}+\tilde\varphi^{(2)})\right\rbrace .
\end{split}
\label{6.4}
\end{align}
Obviously, we can use the same algorithm with Lemma \ref{E1} for completing the proof of this lemma, but it is worth noting that 
 using integration by parts (or not) will lead to the following two different results:
\begin{align*}
&\int_{{\mathbb{R}^3_+}} \partial^{\alpha}\left( (u^s + \tilde{u}^{1})
\partial_x\bar{\varphi} +K(u^s + \tilde{u}^{1}) \partial_y\bar{\varphi}\right)  \langle z \rangle^{k+\ell'+\alpha_{3}}\partial^{\alpha}\bar{\varphi}\\
&=\frac{1}{2}\int_{{\mathbb{R}^3_+}} \langle z \rangle^{2(k+\ell+{\alpha_3})} 
\big((u^s + \tilde{u}^{1}) \partial_x(\partial^\alpha\bar{\varphi})^2 +K(u^s + \tilde{u}^1) \partial_y(\partial^\alpha\bar{\varphi})^2\big)
\\
&
\quad+\sum\limits_{ \beta \leq \alpha, 1\leq |\beta|} C^\beta_\alpha \int_{{\mathbb{R}^3_+}} \langle z \rangle^{k+\ell'+\alpha_{3}}
\Big(\partial^{\beta}(u^s + \tilde{u}^1) \partial^{\alpha - \beta}\partial_x \bar{\varphi}
+
\partial^{\beta}\big(K(u^s + \tilde{u}^{1}) \big)\partial^{\alpha - \beta}\partial_y \bar{\varphi}\Big)
\partial^{\alpha} \bar{\varphi}
\\
&=-\frac{1}{2}\int_{{\mathbb{R}^3_+}} \langle z \rangle^{2(k+\ell+{\alpha_3})} 
\big(\partial_x\tilde{u}^{1}(\partial^\alpha\bar{\varphi})^2 +\partial_{y} (K\tilde{u}^1) (\partial^\alpha\bar{\varphi})^2\big)
\\
&
\quad+\sum\limits_{ \beta \leq \alpha, 1\leq |\beta|} C^\beta_\alpha \int_{{\mathbb{R}^3_+}} \langle z \rangle^{k+\ell'+\alpha_{3}}
\Big(\partial^{\beta}(u^s + \tilde{u}^1) \partial^{\alpha - \beta}\partial_x \bar{\varphi}
+
\partial^{\beta}\big(K(u^s + \tilde{u}^{1}) \big)\partial^{\alpha - \beta}\partial_y \bar{\varphi}\Big)
\partial^{\alpha} \bar{\varphi}\\
& \leq C	\big\| \tilde{\varphi}^{(1)}\big\|_{H^{m-2}_{k+\ell'}(\mathbb{R}^3)}
\left\| \bar{\varphi}\right\|_{H^{m-2}_{k+\ell'}(\mathbb{R}^3)}^2 ,
\end{align*}
or
\begin{align*}
	&\int_{{\mathbb{R}^3_+}} \partial^{\alpha}\left( \bar{u}\partial_{x}\tilde{\varphi}^{(2)}
	+K\bar{u}\partial_{y}\tilde{\varphi}^{(2)}\right)  \langle z \rangle^{k+\ell'+\alpha_{3}}\partial^{\alpha}\bar{\varphi}\\
	&=\int_{{\mathbb{R}^3_+}} \langle z \rangle^{2(k+\ell+{\alpha_3})} 
	\big(\bar{u} \partial_{x}\partial^{\alpha}\tilde{\varphi}^{(2)}  +K\bar{u} \partial_{y}\partial^{\alpha}\tilde{\varphi}^{(2)}\big)\partial^\alpha\bar{\varphi}
	\\
	&
	\quad+\sum\limits_{ \beta \leq \alpha, 1\leq |\beta|} C^\beta_\alpha \int_{{\mathbb{R}^3_+}} \langle z \rangle^{k+\ell'+\alpha_{3}}
	\Big(\partial^{\beta}\bar{u} \partial^{\alpha - \beta}\partial_{x}\tilde{\varphi}^{(2)}
	+
	\partial^{\beta}(K\bar{u}) \partial^{\alpha - \beta}\partial_{y}\tilde{\varphi}^{(2)}\Big)
	\partial^{\alpha} \bar{\varphi}
	\\
	& \leq C	\big\| \tilde{\varphi}^{(2)}\big\|_{H^{m-1}_{k+\ell'}(\mathbb{R}^3)}
	\left\| \bar{\varphi}\right\|_{H^{m-2}_{k+\ell'}(\mathbb{R}^3)}^2 .
\end{align*}
It indicated that
why we only  get the estimate on $\|\bar{\varphi}\|_{H^{m-2}_{k+\ell}}^2$,  but require the norm of $\tilde{\varphi}^{(1)},\tilde{\varphi}^{(2)}$ in $L^\infty([0, T]; H^m_{k+\ell}(\mathbb{R}^3_+))$. Now we just need to prove Lemma \ref{stability1} step by step using  the standard energy methods, so we will omit some details here.

%
\end{proof}

In order to close the estimates, we need to estimate on the loss term 
$\partial^{m-2}_{xy}\bar{\varphi}$, which is controlled by  $\bar{g}^{n}$ via Lemma \ref{varphi-leq-g}.
 More precisely, we have the following theorem for the functions
$$
\bar{g}^{n}=\bigg(\frac{\partial_{xy}^{n} \bar{u}}{u_{z}^{s}+\tu_{z}^{1}}\bigg)_{z}.$$

\begin{lemma}\label{stability2}
	Under the hypotheses of Theorem \ref{e-u-s}, let $\tu^{1}$ and $\tu^{2}$ be two solutions  with respect to the initial date $\tu^{1}_{0}$, $\tu^{2}_{0}$, then we have
\begin{align}
	\frac{d}{dt}\sum^{m-2}_{n=1}\left\| \bar{g}^{n}\right\|_{L^{2}_{\ell'}(\mathbb{R}_{+}^{3})}^2
	+\sum^{m-2}_{n=1}\left\| \partial_{z}\bar{g}^{n}\right\|_{L^{2}_{\ell'}(\mathbb{R}_{+}^{3})}^2+
	\leq
	C\left(\sum^{m-2}_{n=1}\left\| \bar{g}^{n}\right\|_{L^{2}_{\ell'}(\mathbb{R}_{+}^{3})}^2
	+
	\left\| \bar\varphi \right\|_{H^{m-2}_{k+\ell'}(\mathbb{R}_{+}^{3})}^2\right), 
	\label{uu2}
\end{align}
	where the constant $C>0$ depends on the norm of $\tilde{\varphi}^{(1)},\tilde{\varphi}^{(2)}$ in $L^\infty([0, T]; H^m_{k+\ell}(\mathbb{R}^3_+))$.
\end{lemma}
The
 proof of Lemma \ref{stability2} can be recovered by the standard 
process as what we did in Lemma \ref{gn}.

Combining estimates \eqref{uu1} and \ref{uu2}, we arrive at 
\begin{align*}
\|\bar{\varphi}\|_{ L^\infty ([0, T]; H^{m-2}_{k+\ell}(\mathbb{R}^3_+))}
\leq  C \|\bar{u}_{0}\|_{H^{m+1}_{k + \ell'-1}(\mathbb{R}^3_+)},
\end{align*}
which, together with the argument of the  existence  in Section \ref{s5},  completes  the proof of Theorem \ref{e-u-s}.

	\begin{appendices}

\section{Some inequalities}

	First, we present the  following Hardy type inequality which is founded 
in \cite{Xu-Zhang-2017}.
\begin{lemma} \label{hardy} 
	Let $f:\mathbb{R}^{3}_{+}\rightarrow \mathbb{R}$, \\
	{\rm(i)}  if $\lambda>-\frac{1}{2}$ and $\lim\limits_{z\rightarrow+\infty}f(x,y ,z)=0$, then
	\begin{eqnarray}
		\|\langle z \rangle^{\lambda}f\|_{L^{2}( \mathbb{R}^{3}_{+})} \leq \frac{2}{2\lambda+1} \|\langle z \rangle^{\lambda+1}\partial_{z}f\|_{L^{2}(\mathbb{R}^{3}_{+})};
		\label{hardy1}
	\end{eqnarray}
	{\rm(ii)}  if $\lambda<-\frac{1}{2}$ and $ f(x,y,z )|_{z=0}=0$, then
	\begin{eqnarray}
		\|\langle z \rangle^{\lambda}f\|_{L^{2}(\mathbb{R}^{3}_{+})} \leq
		-\frac{2}{2\lambda+1} \|\langle z \rangle^{\lambda+1}\partial_{z}f\|_{L^{2}(\mathbb{R}^{3}_{+})}.
		\label{hardy2}
	\end{eqnarray}
\end{lemma}
	
\hspace*{\fill}

	Next, we shall state the following Sobolev-type inequality.
\begin{lemma}\label{infty}
 For any suitable function $F (x, y, z) : \mathbb{R}^{3}_{+}\rightarrow \mathbb{R}$, 
\begin{align}
	\|F\|_{L^{\infty}\left(\mathbb{R}_{x, y}^{2}\right)} \leq \sqrt{2}\left(\|F\|_{L^{2}\left(\mathbb{R}_{x, y}^{2}\right)}+\left\|\partial_{x} F\right\|_{L^{2}\left(\mathbb{R}_{x, y}^{2}\right)}+\left\|\partial_{y} F\right\|_{L^{2}\left(\mathbb{R}_{x, y}^{2}\right)}
	+\left\|\partial_{x} \partial_{y} F\right\|_{L^{2}\left(\mathbb{R}_{x, y}^{2}\right)}\right).
\label{infty1}
\end{align}
Moreover, if the function $F$ satisfies $F(x,y, z)|_{z=0} = 0$ or $ \lim\limits_{z \rightarrow +\infty}F(x, y, z)=0$, 
then for any small $\delta>0$,
\begin{align}
	\begin{split}
	\|F\|_{L^{\infty}(\mathbb{R}^{3}_{+})}
\leq C\left( 
\|\partial_{z}F\|_{L^{2}_{\frac{1 }{2}+\delta}(\mathbb{R}_{+}^{3})}
+\|\partial_{x}\partial_{z}F\|_{L^{2}_{\frac{1 }{2}+\delta}(\mathbb{R}_{+}^{3})}
+\|\partial_{y}\partial_{z}F\|_{L^{2}_{\frac{1 }{2}+\delta}(\mathbb{R}_{+}^{3})}
+\|\partial_{x}\partial_{y}\partial_{z}F\|_{L^{2}_{\frac{1 }{2}+\delta}(\mathbb{R}_{+}^{3})}
\right). 
	\end{split} 
	\label{infty2}
\end{align}
\end{lemma}
\begin{proof}
Using  the Sobolev embedding inequality, we may easily check that \eqref{infty1} holds. 	For \eqref{infty2},  let
us define
\begin{align*}
F(x, y, z)=\int^{z}_{0}\partial_{z} F(x, y, \tilde{z}) d \tilde{z}.
\end{align*}
Then, it follows from \eqref{infty1} and H\"{o}lder inequality that 
\begin{align*}
\|F\|_{L^{\infty}(\mathbb{R}^{3}_{+})}&\leq \left\|\int^{z}_{0} \langle z \rangle^{-(\frac{1}{2}+\delta)}\left(  \langle z \rangle^{\frac{1}{2}+\delta}\partial_{z} F(x, y, \tilde{z}) d \tilde{z}\right) \right\|_{L^{\infty}(\mathbb{R}^{3}_{+})}+\bigg\|F\Big|_{z=0}\bigg\|_{L^{\infty}(\mathbb{R}^{2}_{x,y})}
\leq C\|\partial_{z}F\|_{L^{\infty}(\mathbb{R}_{x,y}; L^{2}_{\frac{1 }{2}+\delta}(\mathbb{R}_{+}))}\\
&\leq C\left( 
\|\partial_{z}F\|_{L^{2}_{\frac{1 }{2}+\delta}(\mathbb{R}_{+}^{3})}
+\|\partial_{x}\partial_{z}F\|_{L^{2}_{\frac{1 }{2}+\delta}(\mathbb{R}_{+}^{3})}
+\|\partial_{y}\partial_{z}F\|_{L^{2}_{\frac{1 }{2}+\delta}(\mathbb{R}_{+}^{3})}
+\|\partial_{x}\partial_{y}\partial_{z}F\|_{L^{2}_{\frac{1 }{2}+\delta}(\mathbb{R}_{+}^{3})}
\right) .
\end{align*}
When $ \lim\limits_{z \rightarrow +\infty}F(x, y, z)=0$, we denote 
\begin{align*}
	F(x, y, z)=-\int^{\infty}_{y}\partial_{z} F(x, y, \tilde{z}) d \tilde{z},
\end{align*}
and  use the same method.
\end{proof}

The following lemma is a trace theorem (see Lemma A.2. in \cite{Xu-Zhang-2017}) which can help us to deal with boundary value.
\begin{lemma}\label{trace}
	Let $\lambda>\frac{1}{2}$, then there exists a constant $C>0$ such that for any function $f$ defined on $\mathbb{R}^{3}_{+}$, 
	if $\partial_z f\in L^2_{\lambda}(\mathbb{R}^{3}_+)$, it admits a trace on $\mathbb{R}^{2}_{x, y}\times\{0\}$, and satisfies
\begin{align*}
	\|\gamma_0(f)\|_{L^2 (\mathbb{R}^{2}_{x, y})}\leq C
\|\partial_z f\|_{L^2_\lambda(\mathbb{R}^{3}_{+})},
\end{align*}
where $\gamma_{0}(f)(x, y)=f(x, y, 0)$ is the trace operator.
\end{lemma}	
	
Finally, we introduce two interpolation inequalities which will be   used frequently in  Section \ref{s3}.
\begin{lemma}\label{lemma2.4}
	For the suitable functions $f, g$, we have
	\\
{\rm(i)} for $m\geq 6, k+\ell>\frac 32$, and any $\alpha, \beta\in \mathbb{N}^3$ with  $|\alpha|+|\beta|\le m$, we have
\begin{align}
			\|(\partial^\alpha f)(\partial^\beta g)\|_{L^2_{k+\ell+\alpha_3+\beta_3}(\mathbb{R}^3_+)}\le
			C\|f\|_{H^m_{k+\ell}(\mathbb{R}^3_+)}\|g\|_{H^m_{k+\ell}(\mathbb{R}^3_+)};
		\label{s1}
\end{align}
\\
{\rm(ii)} for $m\ge 6, k+\ell>\frac 32$, and any $\alpha\in \mathbb{N}^3, p\in\mathbb{N}$ with $|\alpha|+p\le m$, we have,
\begin{align}
		\|(\partial^\alpha f)(\partial^p_{xy} (\partial^{-1}_z g))\|_{L^2_{k+\ell+\alpha_3}(\mathbb{R}^3_+)}\le
	C\|f\|_{H^m_{k+\ell}(\mathbb{R}^3_+)}\|g\|_{H^m_{\frac 12+\delta}(\mathbb{R}^3_+)}, \label{s2}
\end{align}
	where $\partial^{-1}_z$ is the inverse of derivative $\partial_z$, meaning,
	$\partial^{-1}_z g=\int^z_0 g(x, y, \tilde z) \, d\tilde{z}$.
\end{lemma}

	\section{Compatibility conditions and reduction of boundary data }
The main purpose of this appendix  is to prove the compatibility conditions and reduction of boundary data. We will give more details in the reduction of derivatives on the boundary for the  three-dimensional case.

\begin{proof}[\textbf{Proof of Proposition 3.1 }]

Under the assumption of Proposition \ref{boundary reduction}, $\tilde{u}^{\epsilon}$ is a smooth solution.  For $p \leq \frac{m}{2}$, to make sure that the existence of the trace of $\partial_z^{2p+2} \tilde{u}^{\epsilon}$ on $z=0$, $\tilde{u}^{\epsilon}$ at least needs	to  satisfy  $\partial_{z}^{2p+2}\tilde{u}^{\epsilon} \in L^\infty([0, T]; H^{1}_{k+\ell+2p+1}(\mathbb{R}^3_+))$.
	
	Recalling the boundary condition in \eqref{regular-shear-prandtl}:
	\begin{align*}
		(\tilde{u}^{\epsilon}, \tilde{w}^{\epsilon})(t, x, y, 0)=(0, 0), \quad (t, x, y)\in [0, T]\times \mathbb{R}^{2},
	\end{align*}
then the following is obvious:
\begin{align*}
	\partial^n_{xy}(\tilde{u}^\epsilon,  \tilde{w}^\epsilon)(t, x, y, 0)=(0, 0),\quad (t, x, y)\in [0, T]\times \mathbb{R}^{2}, 
\end{align*}
for $0\leq n\leq m+2$, and 
	\begin{align*}
		\partial_t\partial^n_{xy}(\tilde{u}^\epsilon,  \tilde{w}^\epsilon)(t, x, y, 0)=(0,  0), \quad (t, x, y)\in [0, T]\times \mathbb{R}^{2}, 
	\end{align*}
	for $0\leq n\leq m$. Furthermore, we also have that
	\begin{align*}
		\partial_{z}^{2}\partial_{xy}^{n}\tilde{u}^{\epsilon}(t, x, y, 0)=0, \quad \partial_{t}\partial_{z}^{2}\partial_{xy}^{n}\tilde{u}^\epsilon(t, x, y, 0)=0, \quad (t, x, y)\in [0, T]\times \mathbb{R}^{2},
	\end{align*}	
 for $0\leq n \leq m$. 
 
 With the above conditions,  we now begin to prove the estimates of reduction on the boundary. 
 Then,  taking $\partial_{z}$ operator of equation $\eqref{regular-shear-prandtl}_{1}$,  we obtain
\begin{align*}
	\partial_t\partial_z\tue + \partial_z\big ((u^s + \tilde{u}^{\epsilon}) \partial_x\tue +K(u^s + \tilde{u}^{\epsilon})\partial_y\tilde {{u}}^\epsilon+\tilde{w}^{\epsilon}\partial_{z}({u}^s+\tue)\big)
	= \partial^3_{z}\tue+\epsilon \partial^2_{x}\partial_{z}\tue+\epsilon \partial^2_{y}\partial_{z}\tue,
\end{align*}
	which,   together with the divergence-free condition, yields directly after evaluating at $z = 0$
	\begin{align*}
		\partial_{t}\partial_{z} \tue|_{z=0}=\partial_z^{3}\tue|_{z=0}
		+\epsilon \partial_x^{2}\partial_{z}\tue|_{z=0}
		+\epsilon \partial_y^{2}\partial_{z}\tue|_{z=0}.
	\end{align*}
	Further, differentiating the equation $\eqref{regular-shear-prandtl}_{1}$ with respect to $z$ twice, and then we have
	\begin{align*}
		\partial_t\partial_z^{2}\tue + \partial_z^{2}\big ( (u^s + \tilde{u}^\epsilon) \partial_x\tue +K(u^s + \tilde{u}^\epsilon)\partial_y\tilde {u}^\epsilon+\tilde{w}^\epsilon\partial_{z}({u}^s+\tue)\big)
		= \partial^4_{z}\tue+\epsilon \partial^2_{x}\partial_{z}^{2}\tue+\epsilon \partial^2_{y}\partial_{z}^{2}\tue,
	\end{align*}
	which, together with the following facts obtained by Leibniz formula
	\begin{align*}
		&\partial_z^{2}\big( (u^s + \tilde{u}^{\epsilon}) \partial_x\tue +K(u^s + \tilde{u}^{\epsilon})\partial_y\tilde u^\epsilon+\tilde{w}\partial_{z}(u^s+\tue)\big) \big|_{z=0}\\
		&=2\partial_{z}(u^s + \tilde{u}^\epsilon)\partial_{x}\partial_{z}\tue |_{z=0}+2\partial_{z}\KK\partial_{y}\partial_{z}\tue |_{z=0}+\partial_{z}^{2}\tw \partial_{z}({u}^s+\tue) |_{z=0}
		\\
		&=2\partial_{z}(u^s + \tilde{u}^\epsilon)\partial_{x}\partial_{z}\tue |_{z=0}+2\partial_{z}\big(K(u^s + \tilde{u}^\epsilon)\big)\partial_{y}\partial_{z}\tue |_{z=0}-\partial_{z}\big(\partial_{x}\tue +\partial_{y}(K(u^s+\tue))\big) \partial_{z}(u^s+\tue) |_{z=0}
		\\
		&=2\partial_{z}(u^s + \tilde{u}^\epsilon)\partial_{x}\partial_{z}\tue |_{z=0}+2\partial_{z}\big(K(u^s + \tilde{u}^\epsilon)\big)\partial_{y}\partial_{z}\tue |_{z=0}-\partial_{x}\partial_{z}\tilde{{u}}^{\epsilon} \partial_{z}(u^s+\tue) |_{z=0}-\partial_{y}\partial_{z}\big(K(u^s+\tilde{{u}}^{\epsilon})\big) \partial_{z}(u^s+\tue) |_{z=0},
	\end{align*}
	leads to
	\begin{align}
	\begin{split}
			\partial^4_{z}\tilde{u}^{\epsilon}|_{z=0}&=
			2\partial_{z}(u^s + \tilde{u}^\epsilon)\partial_{x}\partial_{z}\tue |_{z=0}+2\partial_{z}\big(K(u^s + \tilde{u}^\epsilon)\big)\partial_{y}\partial_{z}\tue |_{z=0}-\partial_{x}\partial_{z}\tilde{{u}}^{\epsilon} \partial_{z}(u^s+\tue) |_{z=0}\\
			&\quad-\partial_{y}\partial_{z}\big(K(u^s+\tilde{{u}}^{\epsilon})\big) \partial_{z}(u^s+\tue) |_{z=0}.
	\end{split}
		\label{zu4}
	\end{align}
	Using  the heat equation \eqref{heat} and nonlinear equation \eqref{regular-shear-prandtl}, 
we may check that
\begin{align*}
&\partial_{t}\big(\partial_{z}(u^s + \tilde{u}^\epsilon)\partial_{x}\partial_{z}\tilde{u}^{\epsilon} \big)\big|_{z=0}\\
&=\big(\partial_{z}(u^s + \tilde{u}^\epsilon)(\partial_{x}\partial_{z}^{3}\tilde{u}^{\epsilon}
+\epsilon \partial_x^{3}\partial_{z}\tilde{u}^{\epsilon}
+\epsilon \partial_{x}\partial_y^{2}\partial_{z}\tilde{u}^{\epsilon})\big) \big|_{z=0}  +\big((\partial_{z}^{3} u^s+\partial^{3}_{z}\tue+\epsilon\partial_{x}^{2}\partial_{z}\tue+\epsilon\partial_{y}^{2}\partial_{z}\tue)\partial_{x}\partial_{z}\tilde{u}^\epsilon\big) \big|_{z=0} \, ,\\
&\partial_{t}\big(\partial_{z}(K(u^s + \tilde{u}^\epsilon))\partial_{y}\partial_{z}\tilde{u}^{\epsilon} \big)\big|_{z=0}\\
&=\big(\partial_{z}(K(u^s + \tilde{u}^\epsilon))(\partial_{y}\partial_{z}^{3}\tilde{u}^{\epsilon}
+\epsilon \partial_x^{2}\partial_{y}\partial_{z}\tilde{u}^{\epsilon}
+\epsilon \partial_y^{3}\partial_{z}\tilde{u}^{\epsilon})\big)\big |_{z=0}
\\
&\quad  +\big(\partial^{3}_{z}\big(K(u^s+\tilde{{u}}^{\epsilon})\big)+\epsilon K\partial_{x}^{2}\partial_{z}(u^s+\tilde{{u}}^{\epsilon})+\epsilon K\partial_{y}^{2}\partial_{z}(u^s+\tilde{{u}}^{\epsilon})\big)\partial_{y}\partial_{z}\tilde{u}^\epsilon \big|_{z=0} \, , \\
&\partial_{t}\big(\partial_{x}\partial_{z}\tilde{{u}}^{\epsilon} \partial_{z}(u^s+\tilde{u}^{\epsilon}) \big)\big|_{z=0}\\
&=\big(\partial_{x}\partial_{z}\tilde{{u}}^{\epsilon}(\partial_{z}^{3} u^s+\partial^{3}_{z}\tilde{u}^{\epsilon}+\epsilon\partial_{x}^{2}\partial_{z}\tilde{u}^{\epsilon}+\epsilon\partial_{y}^{2}\partial_{z}\tilde{u}^{\epsilon}) \big) \big|_{z=0}+\big((\partial_{x}\partial_{z}^{3}\tilde{{u}}^\epsilon
+\epsilon \partial_x^{3}\partial_{z}\tilde{{u}}^{\epsilon}
+\partial_{x}\epsilon \partial_y^{2}\partial_{z}\tilde{{u}}^{\epsilon}) \partial_{z}(u^s+\tilde{{u}}^{\epsilon}) \big)\big|_{z=0}\, , \\
&\partial_{t}\big(\partial_{y}\partial_{z}\big(K(u^s+\tilde{{u}}^{\epsilon})\big) \partial_{z}({u}^s+\tilde{{u}}^{\epsilon}) \big)\big|_{z=0}\\
&=\big(\partial_{y}\partial_{z}\big(K(u^s+\tilde{{u}}^{\epsilon})\big)(\partial_{z}^{3} {u}^s+\partial^{3}_{z}\tilde{{u}}^{\epsilon}+\epsilon\partial_{x}^{2}\partial_{z}\tilde{{u}}^{\epsilon}+\epsilon\partial_{y}^{2}\partial_{z}\tilde{{u}}^{\epsilon}) \big) \big|_{z=0}\\
&\quad
+\big((\partial_{y}\partial_{z}^{3}(K(u^s+\tilde{{u}}^\epsilon))
+\epsilon \partial_{y}\partial_{z}\big(K\partial_x^{2}(u^s+\tilde{{u}}^{\epsilon})\big)
+\epsilon \partial_y\partial_{z}\big(K\partial_y^{2}(u^s+\tilde{{u}}^{\epsilon})\big)) \partial_{z}({u}^s+\tilde{{u}}^{\epsilon}) \big)\big|_{z=0}
\, ,
\end{align*}
and hence,  we have 
\begin{align}
\begin{split}
&\partial_{t}\partial^4_{z}\tilde{u}^{\epsilon}|_{z=0}\\
&=2\big(\partial_{z}(u^s + \tilde{u}^\epsilon)(\partial_{x}\partial_{z}^{3}\tilde{u}^{\epsilon}
+\epsilon \partial_x^{3}\partial_{z}\tilde{u}^{\epsilon}
+\epsilon \partial_{x}\partial_y^{2}\partial_{z}\tilde{u}^{\epsilon})\big) \big|_{z=0}  
\\
&\quad+2\big((\partial_{z}^{3} u^s+\partial^{3}_{z}\tue+\epsilon\partial_{x}^{2}\partial_{z}\tue+\epsilon\partial_{y}^{2}\partial_{z}\tue)\partial_{x}\partial_{z}\tilde{u}^\epsilon\big) \big|_{z=0} 
\\
&\quad+2\big(\partial_{z}(K(u^s + \tilde{u}^\epsilon))(\partial_{y}\partial_{z}^{3}\tilde{u}^{\epsilon}
+\epsilon \partial_x^{2}\partial_{y}\partial_{z}\tilde{u}^{\epsilon}
+\epsilon \partial_y^{3}\partial_{z}\tilde{u}^{\epsilon})\big)\big |_{z=0} 
\\
&\quad+2\big(\partial^{3}_{z}\big(K(u^s+\tilde{{u}}^{\epsilon})\big)+\epsilon K\partial_{x}^{2}\partial_{z}(u^s+\tilde{{u}}^{\epsilon})+\epsilon K\partial_{y}^{2}\partial_{z}(u^s+\tilde{{u}}^{\epsilon})\big)\partial_{y}\partial_{z}\tilde{u}^\epsilon \big|_{z=0}
\\
&
\quad-
\big(\partial_{x}\partial_{z}\tilde{{u}}^{\epsilon}(\partial_{z}^{3} u^s+\partial^{3}_{z}\tilde{u}^{\epsilon}+\epsilon\partial_{x}^{2}\partial_{z}\tilde{u}^{\epsilon}+\epsilon\partial_{y}^{2}\partial_{z}\tilde{u}^{\epsilon}) \big) \big|_{z=0}
\\
&\quad
-\big((\partial_{x}\partial_{z}^{3}\tilde{{u}}^\epsilon
+\epsilon \partial_x^{3}\partial_{z}\tilde{{u}}^{\epsilon}
+\epsilon\partial_{x} \partial_y^{2}\partial_{z}\tilde{{u}}^{\epsilon}) \partial_{z}(u^s+\tilde{{u}}^{\epsilon}) \big)\big|_{z=0}\, , \\
&\quad-
\big(\partial_{y}\partial_{z}\big(K(u^s+\tilde{{u}}^{\epsilon})\big)(\partial_{z}^{3} {u}^s+\partial^{3}_{z}\tilde{{u}}^{\epsilon}+\epsilon\partial_{x}^{2}\partial_{z}\tilde{{u}}^{\epsilon}+\epsilon\partial_{y}^{2}\partial_{z}\tilde{{u}}^{\epsilon}) \big) \big|_{z=0}
\\
&\quad-
\big((\partial_{y}\partial_{z}^{3}(K(u^s+\tilde{{u}}^\epsilon))
+\epsilon \partial_{y}\partial_{z}\big(K\partial_x^{2}(u^s+\tilde{{u}}^{\epsilon})\big)
+\epsilon \partial_y\partial_{z}\big(K\partial_y^{2}(u^s+\tilde{{u}}^{\epsilon})\big)) \partial_{z}({u}^s+\tilde{{u}}^{\epsilon}) \big)\big|_{z=0}
\, ,
\end{split}
\label{tzu4}
\end{align}
	For $p = 2$, we have
	\begin{align*}
		\partial_t\partial_z^{4}\tilde{u}^{\epsilon} + \partial_z^{4}\big( (u^s + \tilde{u}^{\epsilon}) \partial_x\tilde{u}^{\epsilon} +K(u^s + \tilde{u}^{\epsilon})\partial_y\tilde {u}^\epsilon+\tilde{w}^{\epsilon}\partial_{z}(u^s+\tilde{u}^{\epsilon})\big)
		= \partial^6_{z}\tilde{u}^{\epsilon}+\epsilon \partial^2_{x}\partial_{z}^{4}\tilde{u}^{\epsilon}+\epsilon \partial^2_{y}\partial_{z}^{4}\tilde{u}^{\epsilon}.
	\end{align*}
Using Leibniz formula
	\begin{align*}
		\begin{split}
			&\partial_z^{4}\big( (u^s + \tilde{u}^{\epsilon}) \partial_x\tilde{u}^{\epsilon} +K(u^s + \tilde{u}^{\epsilon})\partial_y\tilde {u}^\epsilon+\tilde{w}^{\epsilon}\partial_{z}(u^s+\tilde{u}^{\epsilon})\big)\\
			&= \partial_{z}^{4}(u^s + \tilde{u}^{\epsilon}) \partial_x\tilde{u}^{\epsilon}+(u^s + \tilde{u}^{\epsilon}) \partial_{z}^{4}\partial_x\tilde{u}^{\epsilon} +\sum_{1\leq j\leq 3}C^4_j \bigg(\partial^j_z(u^s + \tilde{u}^{\epsilon}) \partial^{4-j}_z\partial_x\tilde{u}^{\epsilon} \bigg) \\
			&\quad + K\partial_{z}^{4}(u^s + \tilde{u}^{\epsilon}) \partial_y\tilde{u}^{\epsilon}+K(u^s + \tilde{u}^{\epsilon}) \partial_{z}^{4}\partial_y\tilde{u}^{\epsilon} +\sum_{1\leq j\leq 3}C^4_j \bigg(\partial^j_z\big(K(u^s + \tilde{u}^{\epsilon})\big) \partial^{4-j}_z\partial_y\tilde{u}^{\epsilon} \bigg)\\
			&\quad +\partial_{z}^{4}\tilde{w}\partial_{z}(u^s+\tilde{u}^{\epsilon})+\tilde{w}\partial_{z}^{5}(u^s+\tilde{u}^{\epsilon})+\sum_{1\leq j\leq 3}C^4_j \bigg(\partial^j_z \tw \partial^{4-j}_z \partial_{z}(u^s+\tilde{u}^{\epsilon}) \bigg),
		\end{split}
	\end{align*}
and   the divergence-free condition $\partial_x\tilde{u}^{\epsilon} +\partial_y\big(K(u^{s}+\tue)\big)+\partial_z\tilde{w} ^{\epsilon}=0$, we have
	\begin{align*}
		\partial^6_{z}\tilde{u}^{\epsilon}|_{z=0}&= \partial_t\partial_z^{4}\tilde{u}^{\epsilon}|_{z=0} - \partial_{z}^{3}\uve \partial_{z}(u^s+\tilde{u}^{\epsilon}) |_{z=0}+\sum_{1\leq j\leq 3}C^4_j \bigg(\partial^j_z(u^s + \tilde{u}^{\epsilon}) \partial^{4-j}_z\partial_x\tilde{u}^{\epsilon} \bigg)\bigg|_{z=0}\\
		&\quad   +\sum_{1\leq j\leq 3}C^4_j \bigg(\partial^j_z\big(K(u^s + \tilde{u}^{\epsilon})\big) \partial^{4-j}_z\partial_y\tilde{u}^{\epsilon} \bigg)\bigg|_{z=0}+\sum_{1\leq j\leq 3}C^4_j \bigg(\partial^{j-1}_z \big(\partial_x\tilde{u}^{\epsilon} +\partial_y\big(K(u^{s}+\tue)\big)\big) \partial^{4-j}_z \partial_{z}(u^s + \tilde{u}^{\epsilon}) \bigg) \bigg|_{z=0}\\
		&\quad -\epsilon \partial^2_{x}\partial_{z}^{4}\tilde{u}^{\epsilon}|_{z=0}-\epsilon \partial^2_{y}\partial_{z}^{4}\tilde{u}^{\epsilon}|_{z=0} \, .
	\end{align*}
As for the last two terms on the right-hand side of above equation,  we use \eqref{zu4} to conclude that
	\begin{align*}
		\begin{split}
		&-\epsilon\partial^2_{x}\partial_{z}^{4}\tilde{u}^{\epsilon}-\epsilon \partial^2_{y}\partial_{z}^{4}\tilde{u}^{\epsilon}\\
		&=-\epsilon(\partial^2_{x}+\partial^2_{y})
		\left(2\partial_{z}(u^s + \tilde{u}^\epsilon)\partial_{x}\partial_{z}\tue +2\partial_{z}\big(K(u^s + \tilde{u}^\epsilon)\big)\partial_{y}\partial_{z}\tue -\partial_{x}\partial_{z}\tilde{{u}}^{\epsilon} \partial_{z}(u^s+\tue) \right. 
	\\
	&\quad \left.	-\partial_{y}\partial_{z}\big(K(u^s+\tilde{{u}}^{\epsilon})\big) \partial_{z}(u^s+\tue)   \right)\\
		&=-2\epsilon\partial_{z}\use\partial_{x}^{3}\partial_{z}\tilde{u}^{\epsilon} \underline{-4\epsilon\partial_{x}\partial_{z}\use\partial_{x}^{2}\partial_{z}\tilde{u}^{\epsilon}} -2\epsilon\partial_{x}^{2}\partial_{z}\use\partial_{x}\partial_{z}\tilde{u}^{\epsilon}\\
		&\quad-2\epsilon \partial_{z}(K\use)\partial_{x}^{2}\partial_{y}\partial_{z}\tilde{u}^{\epsilon} \underline{-4\epsilon \partial_{x}\partial_{z}(K\use)\partial_{x}\partial_{y}\partial_{z}\tilde{u}^{\epsilon}} -2\epsilon K\partial_{x}^{2}\partial_{z}\use\partial_{y}\partial_{z}\tilde{u}^{\epsilon}\\
		&\quad\underline{
		 -2\epsilon \partial_{x}^{2}K\partial_{z}\use\partial_{y}\partial_{z}\tilde{u}^{\epsilon}
	   -2\epsilon \partial_{x}K\partial_{x}\partial_{z}\use\partial_{y}\partial_{z}\tilde{u}^{\epsilon}
		}
		\\
		&\quad 
		+\epsilon \partial_{x}\partial_{z}\tilde{{u}}^{\epsilon} \partial_{x}^{2}\partial_{z}(u^s+\tilde{u}^{\epsilon})
		 \underline{+2\epsilon \partial_{x}^{2}\partial_{z}\tilde{{u}}^{\epsilon} \partial_{x}\partial_{z}(u^s+\tilde{u}^{\epsilon})}
		 +\epsilon \partial_{x}^{3}\partial_{z}\tilde{{u}}^{\epsilon} \partial_{z}(u^s+\tilde{u}^{\epsilon}) \\
		&\quad 
		+\epsilon \partial_{y}\partial_{z}\big(K(u^s+\tilde{{u}}^{\epsilon})\big) \partial_{x}^{2}\partial_{z}(u^s+\tilde{u}^{\epsilon})
		 \underline{+2\epsilon \partial_{x}\partial_{y}\partial_{z}\big(K(u^s+\tilde{{u}}^{\epsilon})\big) \partial_{x}\partial_{z}(u^s+\tilde{u}^{\epsilon})}
		 +\epsilon \partial_{y}\partial_{z}\big(K\partial_{x}^{2}(u^s+\tilde{{u}}^{\epsilon})\big) \partial_{z}(u^s+\tilde{u}^{\epsilon}) \\
		 &\quad 
		 \underline{
		 +\epsilon \partial_{y}\partial_{z}\big(\partial_{x}^{2}K(u^s+\tilde{{u}}^{\epsilon})\big) \partial_{z}(u^s+\tilde{u}^{\epsilon})
		 +\epsilon \partial_{y}\partial_{z}\big(\partial_{x}K\partial_{x}(u^s+\tilde{{u}}^{\epsilon})\big) \partial_{z}(u^s+\tilde{u}^{\epsilon})
		 }
		  \\
		&\quad-2\epsilon\partial_{z}\use \partial_{x}\partial_{y}^{2}\partial_{z}\tilde{u}^{\epsilon} \underline{-4\epsilon\partial_{y}\partial_{z}\use\partial_{x}\partial_{y}\partial_{z}\tilde{u}^{\epsilon}} -2\epsilon\partial_{y}^{2}\partial_{z}\use\partial_{x}\partial_{z}\tilde{u}^{\epsilon}\\
		&\quad-2\epsilon \partial_{z}(K\use)\partial_{y}^{3}\partial_{z}\tilde{u}^{\epsilon} \underline{-4\epsilon \partial_{y}\partial_{z}(K\use)\partial_{y}^{2}\partial_{z}\tilde{u}^{\epsilon}} -2\epsilon K\partial_{y}^{2}\partial_{z}\use\partial_{y}\partial_{z}\tilde{u}^{\epsilon}\\
		&\quad
		\underline{
			-2\epsilon \partial_{y}^{2}K\partial_{z}\use\partial_{y}\partial_{z}\tilde{u}^{\epsilon}
		-2\epsilon \partial_{y}K\partial_{y}\partial_{z}\use\partial_{y}\partial_{z}\tilde{u}^{\epsilon}}
		\\
		&\quad
		+\epsilon \partial_{x}\partial_{z}\tilde{{u}}^{\epsilon}\partial_{y}^{2}\partial_{z}(u^s+\tilde{u}^{\epsilon})
		\underline{+2\epsilon \partial_{x}\partial_{y}\partial_{z}\tilde{{u}}^{\epsilon} \partial_{y}\partial_{z}(u^s+\tilde{u}^{\epsilon})}
		 +\epsilon \partial_{x}\partial_{y}^{2}\partial_{z}\tilde{{u}}^{\epsilon}\partial_{z}(u^s+\tilde{u}^{\epsilon})  \\
		&\quad 
		+\epsilon \partial_{y}\partial_{z}\big(K(u^s+\tilde{{u}}^{\epsilon})\big)\partial_{y}^{2}\partial_{z}(u^s+\tilde{u}^{\epsilon})
		\underline{+2\epsilon \partial_{y}^{2}\partial_{z}\big(K(u^s+\tilde{{u}}^{\epsilon})\big)\partial_{y}\partial_{z}(u^s+\tilde{u}^{\epsilon}) }
		+\epsilon \partial_{y}\partial_{z}\big(K\partial_{y}^2(u^s+\tilde{{u}}^{\epsilon})\big)\partial_{z}(u^s+\tilde{u}^{\epsilon})
		\\
&\quad 
\underline{
	+\epsilon \partial_{y}\partial_{z}\big(\partial_{y}^{2}K(u^s+\tilde{{u}}^{\epsilon})\big) \partial_{z}(u^s+\tilde{u}^{\epsilon})
	+\epsilon \partial_{y}\partial_{z}\big(\partial_{y}K\partial_{y}(u^s+\tilde{{u}}^{\epsilon})\big) \partial_{z}(u^s+\tilde{u}^{\epsilon})
}
.
		\end{split}
	\end{align*}
Noticing that the terms which contain $\epsilon$  in the expansions of \eqref{tzu4} cancel out with all the terms in the  above equation  at $z=0$ except the underlined term, we arrive at
	\begin{align*}
		&\partial^6_{z}\tilde{u}^{\epsilon}|_{z=0}\\
		&= 2\partial_{z}\use\partial_{z}^{3}\partial_{x}\tilde{u}^{\epsilon} |_{z=0}+2\partial_{z}^{3}\use\partial_{z}\partial_{x}\tilde{u}^{\epsilon} |_{z=0}+\sum_{1\leq j\leq 3}C^4_j \bigg(\partial^j_z(u^s + \tilde{u}^{\epsilon}) \partial^{4-j}_z\partial_x\tilde{u}^{\epsilon} \bigg)\bigg|_{z=0}\\
		&\quad +2\partial_{z}(K\use)\partial_{z}^{3}\partial_{y}\tilde{u}^{\epsilon} |_{z=0}+ 2\partial_{z}^{3}(K\use)\partial_{z}\partial_{y}\tilde{u}^{\epsilon}|_{z=0} +\sum_{1\leq j\leq 3}C^4_j \bigg(\partial^j_zK(u^s + \tilde{u}^\epsilon) \partial^{4-j}_z\partial_y\tilde{u}^{\epsilon} \bigg)\bigg|_{z=0}\\
		&\quad - 2\partial_{z}^{3}\big(\partial_x\tilde{u}^{\epsilon} +\partial_y\big(K(u^{s}+\tue)\big)\big) \partial_{z}(u^s+\tilde{u}^{\epsilon}) |_{z=0}-\partial_{z}^{3}(u^s+\tilde{u}^{\epsilon}) \partial_{z}\big(\partial_x\tilde{u}^{\epsilon} +\partial_y\big(K(u^{s}+\tue)\big)\big) |_{z=0}\\
		&\quad-\sum_{1\leq j\leq 3}C^4_j \bigg(\partial^{j-1}_z \big(\partial_x\tilde{u}^{\epsilon} +\partial_y\big(K(u^{s}+\tue)\big)\big) \partial^{5-j}_z (u^s + \tilde{u}^{\epsilon}) \bigg) \bigg|_{z=0} \\
&
\quad
\underline{-4\epsilon\partial_{x}\partial_{z}\use\partial_{x}^{2}\partial_{z}\tilde{u}^{\epsilon}} 
\underline{-4\epsilon \partial_{x}\partial_{z}(K\use)\partial_{x}\partial_{y}\partial_{z}\tilde{u}^{\epsilon}} 
\\
&\quad
\underline{
	-2\epsilon \partial_{x}^{2}K\partial_{z}\use\partial_{y}\partial_{z}\tilde{u}^{\epsilon}
	-2\epsilon \partial_{x}K\partial_{x}\partial_{z}\use\partial_{y}\partial_{z}\tilde{u}^{\epsilon}
}
\\
&\quad
\underline{+2\epsilon \partial_{x}^{2}\partial_{z}\tilde{{u}}^{\epsilon} \partial_{x}\partial_{z}(u^s+\tilde{u}^{\epsilon})}	
\underline{+2\epsilon \partial_{x}\partial_{y}\partial_{z}\big(K(u^s+\tilde{{u}}^{\epsilon})\big) \partial_{x}\partial_{z}(u^s+\tilde{u}^{\epsilon})}
\\
&\quad
\underline{
	+\epsilon \partial_{y}\partial_{z}\big(\partial_{x}^{2}K(u^s+\tilde{{u}}^{\epsilon})\big) \partial_{z}(u^s+\tilde{u}^{\epsilon})
	+\epsilon \partial_{y}\partial_{z}\big(\partial_{x}K\partial_{x}(u^s+\tilde{{u}}^{\epsilon})\big) \partial_{z}(u^s+\tilde{u}^{\epsilon})
}
\\
&\quad \underline{-4\epsilon\partial_{y}\partial_{z}\use\partial_{x}\partial_{y}\partial_{z}\tilde{u}^{\epsilon}} 
\underline{-4\epsilon \partial_{y}\partial_{z}(K\use)\partial_{y}^{2}\partial_{z}\tilde{u}^{\epsilon}} 
\\
&\quad
\underline{
	-2\epsilon \partial_{y}^{2}K\partial_{z}\use\partial_{y}\partial_{z}\tilde{u}^{\epsilon}
	-2\epsilon \partial_{y}K\partial_{y}\partial_{z}\use\partial_{y}\partial_{z}\tilde{u}^{\epsilon}}
\\
&\quad
\underline{+2\epsilon \partial_{x}\partial_{y}\partial_{z}\tilde{{u}}^{\epsilon} \partial_{y}\partial_{z}(u^s+\tilde{u}^{\epsilon})}
\underline{+2\epsilon \partial_{y}^{2}\partial_{z}\big(K(u^s+\tilde{{u}}^{\epsilon})\big)\partial_{y}\partial_{z}(u^s+\tilde{u}^{\epsilon}) }
\\		
&\quad 
\underline{
	+\epsilon \partial_{y}\partial_{z}\big(\partial_{y}^{2}K(u^s+\tilde{{u}}^{\epsilon})\big) \partial_{z}(u^s+\tilde{u}^{\epsilon})
	+\epsilon \partial_{y}\partial_{z}\big(\partial_{y}K\partial_{y}(u^s+\tilde{{u}}^{\epsilon})\big) \partial_{z}(u^s+\tilde{u}^{\epsilon})
}\, ,
	\end{align*}
which implies that
	\begin{align}
	\begin{split}
&\partial^6_{z}\tilde{u}^{\epsilon}|_{z=0}\\
&= \uuline{2\partial_{z}\use\partial_{z}^{3}\partial_{x}\tilde{u}^{\epsilon} |_{z=0}+2\partial_{z}^{3}\use\partial_{z}\partial_{x}\tilde{u}^{\epsilon} |_{z=0}}+\sum_{1\leq j\leq 3}C^4_j \bigg(\partial^j_z(u^s + \tilde{u}^{\epsilon}) \partial^{4-j}_z\partial_x\tilde{u}^{\epsilon} \bigg)\bigg|_{z=0}\\
&\quad \uuline{+2\partial_{z}(K\use)\partial_{z}^{3}\partial_{y}\tilde{u}^{\epsilon} |_{z=0}+ 2\partial_{z}^{3}(K\use)\partial_{z}\partial_{y}\tilde{u}^{\epsilon}|_{z=0}} +\sum_{1\leq j\leq 3}C^4_j \bigg(\partial^j_zK(u^s + \tilde{u}^\epsilon) \partial^{4-j}_z\partial_y\tilde{u}^{\epsilon} \bigg)\bigg|_{z=0}\\
&\quad  \uuline{-\partial_{z}^{3}\big(\partial_x\tilde{u}^{\epsilon} +\partial_y\big(K(u^{s}+\tue)\big)\big) \partial_{z}({u}^s+\tilde{u}^{\epsilon}) |_{z=0}-\partial_{z}^{3}({u}^s+\tilde{u}^{\epsilon}) \partial_{z}\big(\partial_x\tilde{u}^{\epsilon} +\partial_y\big(K(u^{s}+\tue)\big)\big) |_{z=0}}\\
&\quad-\sum_{1\leq j\leq 3}C^4_{j+1} \bigg(\partial^{j}_z \big(\partial_x\tilde{u}^{\epsilon} +\partial_y\big(K(u^{s}+\tue)\big)\big) \partial^{4-j}_z ({u}^s + \tilde{u}^{\epsilon}) \bigg) \bigg|_{z=0} \\
&
\quad
\underline{-4\epsilon\partial_{x}\partial_{z}\use\partial_{x}^{2}\partial_{z}\tilde{u}^{\epsilon}} 
\underline{-4\epsilon \partial_{x}\partial_{z}(K\use)\partial_{x}\partial_{y}\partial_{z}\tilde{u}^{\epsilon}} 
\\
&\quad
\underline{
	-2\epsilon \partial_{x}^{2}K\partial_{z}\use\partial_{y}\partial_{z}\tilde{u}^{\epsilon}
	-2\epsilon \partial_{x}K\partial_{x}\partial_{z}\use\partial_{y}\partial_{z}\tilde{u}^{\epsilon}
}
\\
&\quad
\underline{+2\epsilon \partial_{x}^{2}\partial_{z}\tilde{{u}}^{\epsilon} \partial_{x}\partial_{z}(u^s+\tilde{u}^{\epsilon})}	
\underline{+2\epsilon \partial_{x}\partial_{y}\partial_{z}\big(K(u^s+\tilde{{u}}^{\epsilon})\big) \partial_{x}\partial_{z}(u^s+\tilde{u}^{\epsilon})}
\\
&\quad
\underline{
	+\epsilon \partial_{y}\partial_{z}\big(\partial_{x}^{2}K(u^s+\tilde{{u}}^{\epsilon})\big) \partial_{z}(u^s+\tilde{u}^{\epsilon})
	+\epsilon \partial_{y}\partial_{z}\big(\partial_{x}K\partial_{x}(u^s+\tilde{{u}}^{\epsilon})\big) \partial_{z}(u^s+\tilde{u}^{\epsilon})
}
\\
&\quad \underline{-4\epsilon\partial_{y}\partial_{z}\use\partial_{x}\partial_{y}\partial_{z}\tilde{u}^{\epsilon}} 
\underline{-4\epsilon \partial_{y}\partial_{z}(K\use)\partial_{y}^{2}\partial_{z}\tilde{u}^{\epsilon}} 
\\
&\quad
\underline{
	-2\epsilon \partial_{y}^{2}K\partial_{z}\use\partial_{y}\partial_{z}\tilde{u}^{\epsilon}
	-2\epsilon \partial_{y}K\partial_{y}\partial_{z}\use\partial_{y}\partial_{z}\tilde{u}^{\epsilon}}
\\
&\quad
\underline{+2\epsilon \partial_{x}\partial_{y}\partial_{z}\tilde{{u}}^{\epsilon} \partial_{y}\partial_{z}(u^s+\tilde{u}^{\epsilon})}
\underline{+2\epsilon \partial_{y}^{2}\partial_{z}\big(K(u^s+\tilde{{u}}^{\epsilon})\big)\partial_{y}\partial_{z}(u^s+\tilde{u}^{\epsilon}) }
\\		
&\quad 
\underline{
	+\epsilon \partial_{y}\partial_{z}\big(\partial_{y}^{2}K(u^s+\tilde{{u}}^{\epsilon})\big) \partial_{z}(u^s+\tilde{u}^{\epsilon})
	+\epsilon \partial_{y}\partial_{z}\big(\partial_{y}K\partial_{y}(u^s+\tilde{{u}}^{\epsilon})\big) \partial_{z}(u^s+\tilde{u}^{\epsilon})
}
.
	\end{split}
\label{uve 6}
\end{align}
The double underlined terms   can be absorbed by the corresponding sum terms, and the underlined terms produced after cancellation are extra terms due to the addition of  the viscous terms $\epsilon(\partial^2_{x}\tilde{u}^{\epsilon}+\partial^2_{y}\tilde{u}^{\epsilon})$.  All terms on the right-hand side of the equality are  in the desired form,  then we justify the formula \eqref{z2p+1 uv} for $p=2$.	
	
%
Assume \eqref{z2p+1 uv} holds for $k$,   then taking $2k+2$ times derivatives on the equation
 $\eqref{regular-shear-prandtl}_{1}$, one has
	\begin{align}
	\begin{split}
&\partial_{z}^{2k+4} u^{\epsilon}|_{z=0}\\
&=(\partial_t-\epsilon \partial^2_{x}-\epsilon \partial^2_{y})\partial_{z}^{2k+2}\tilde{u}^{\epsilon} |_{z=0} \\
&\quad   
+\sum_{1\leq j\leq 2k+1}C^{2k+2}_j \bigg(\partial^j_z(u^s + \tilde{u}^{\epsilon}) \partial^{(2k+2)-j}_z\partial_x\tilde{u}^{\epsilon} \bigg)\bigg|_{z=0}
+\sum_{1\leq j\leq 2k+1}C^{2k+2}_j \bigg(\partial^j_z(K(u^s + \tilde{u}^{\epsilon})) \partial^{(2k+2)-j}_z\partial_y\tilde{u}^{\epsilon} \bigg)\bigg|_{z=0} \\
&\quad+\sum_{2\leq j\leq 2k+1}C^4_j \bigg(\partial^{j-1}_z \big(\partial_x\tilde{u}^{\epsilon} +\partial_y\big(K(u^{s}+\tue)\big)\big) \partial^{(2k+2)-j}_z \partial_{z}(u^s + \tilde{u}^{\epsilon}) \bigg) \bigg|_{z=0}.
	\end{split}
\label{z2k+4 u}
\end{align}	

By checking the index, it is enough to deal with the first  term on
the right-hand side of \eqref{z2k+4 u},
\begin{align*}
	&(\partial_t-\epsilon \partial^2_{x}-\epsilon \partial^2_{y})\partial_{z}^{2k+2}\tilde{{u}}^{\epsilon} |_{z=0}\\
	&=(\partial_t-\epsilon \partial^2_{x}-\epsilon \partial^2_{y})\bigg(\sum^p_{q=2}\sum^{q-1}_{l=0}\epsilon^{l}\sum_{( \beta, \gamma)\in \Lambda_{q, l}}C_{K, p, l, \beta, \gamma}\prod\limits_{i=1}^{q_{1}}  \partial^{\beta}\partial_{z}\big( u^s+ \tilde{u}^\epsilon \big) \times 
	\prod\limits_{j=1}^{q_{2}}  \partial^{\gamma}\partial_{z}\big( K( {u}^s+\tilde{u}^\epsilon) \big)\bigg) \bigg|_{z=0} \, .
\end{align*}

Three cases should be considered:

(1) The derivative operator $\partial_t-\epsilon \partial^2_{x}-\epsilon \partial^2_{y}$ on $\partial^{\beta}\partial_{z}\big( u^s+ \tilde{u}^\epsilon \big) $;

(2) The derivative operator $\partial_t-\epsilon \partial^2_{x}-\epsilon \partial^2_{y}$ on $\partial^{\gamma}\partial_{z}\big(K( u^s+ \tilde{u}^\epsilon) \big)$;

(3) The derivative operator $-\epsilon \partial^2_{x}-\epsilon \partial^2_{y}$ separate to $\partial^{\beta}\partial_{z}\big( u^s+ \tilde{u}^\epsilon \big) $ and $\partial^{\gamma}\partial_{z}\big( K(u^s+ \tilde{u}^\epsilon) \big)$.

\noindent\textbf{Case 1.}
\begin{align*}
&(\partial_t-\epsilon \partial^2_{x}-\epsilon \partial^2_{y})\big(\partial^{\beta^{i}}\partial_{z}(u^s+\tu)\big)
|_{z=0}
\\
&=  
-\sum_{\theta_{z} \geq 1,  \theta \leq \beta^{i}}\binom {\beta^{i}} {\theta} \bigg(\partial_z\partial^{\theta-e_{3}}(u^s + \tilde{u}^{\epsilon}) \partial_x\partial_{z}\partial^{\beta^{i}-\theta}\tilde{u}^{\epsilon} \bigg)\bigg|_{z=0}\\
&\quad-
\sum_{ \theta_{z} \geq 1, \theta \leq \beta^{i}}\binom {\beta^{i}} {\theta} \bigg(\partial_z\partial^{\theta-e_{3}}\big(K(u^s + \tilde{u}^{\epsilon})\big) \partial_y\partial_{z}\partial^{\beta^{i}-\theta}\tilde{u}^{\epsilon} \bigg)\bigg|_{z=0} \\
&\quad+\sum_{ \theta_{z} \geq 2,  \theta \leq \beta^{i}}\binom {\beta^{i}} {\theta} \bigg(\partial_{z}\partial^{\theta-2e_{3}}\big(\partial_x\tilde{u}^{\epsilon} +\partial_y\big(K(u^{s}+\tue)\big)\big)  \partial_{z}^{2}\partial^{\beta^{i}-\theta}(u^s + \tilde{u}^{\epsilon}) \bigg) \bigg|_{z=0}\\
&\quad
-\sum_{ \theta_{z} \geq 1,  \theta \leq \beta^{i}}\binom {\beta^{i}} {\theta} \bigg(\partial_{z}\partial^{\theta-e_{3}}\big(\partial_{y}K (u^s+\tue)\big) ~ \partial_{z}\partial^{\beta^{i}-\theta}(u^{s}+\tilde{u}^\epsilon) \bigg) \bigg|_{z=0}
- \partial^3_{z}\partial^{\beta^{i}}(u^s+\tilde{u}^{\epsilon})
\big|_{z=0}.
\end{align*}
We can check each term, which all satisfies Proposition \ref{boundary reduction}.

\noindent\textbf{Case 2.}
\begin{align*}
	&(\partial_t-\epsilon \partial^2_{x}-\epsilon \partial^2_{y})
	\big(\partial^{\gamma^{i}}\partial_{z}\big( K( u^s+\tilde{u}^\epsilon) \big)
	\big|_{z=0}
	\\
	&=  
	-\partial^{\gamma^{i}}\big( \partial_{x}^2K\partial_{z}( u^s+\tilde{u}^\epsilon)
	-\partial^{\gamma^{i}}\big( \partial_{x}K\partial_{x}\partial_{z}( u^s+\tilde{u}^\epsilon)
	-\partial^{\gamma^{i}}\big( \partial_{y}^2K\partial_{z}( u^s+\tilde{u}^\epsilon)
	-\partial^{\gamma^{i}}\big( \partial_{y}K\partial_{y}\partial_{z}( u^s+\tilde{u}^\epsilon)
	\\
	&\quad
	-\sum_{\kappa_{z} \geq 1,  \kappa \leq \gamma^{i}}\binom {\gamma^{i}} {\kappa} \bigg(
	\partial_z\partial^{\kappa-e_{3}}\big(K(u^s + \tilde{u}^{\epsilon})\big) \partial_x\partial_{z}\partial^{\gamma^{i}-\kappa}\tilde{{u}}^{\epsilon} 
	\bigg)\bigg|_{z=0} \\
	&\quad-
	\sum_{ \kappa_{z} \geq 1, \kappa \leq \gamma^{i}}\binom {\gamma^{i}} {\kappa} \bigg(
	\partial_z\partial^{\kappa-e_{3}}\big(K^2(u^s + \tilde{u}^{\epsilon})\big) \partial_y\partial_{z}\partial^{\gamma^{i}-\kappa}\tilde{u}^{\epsilon} 
		\bigg)\bigg|_{z=0} \\
	&\quad+\sum_{ \kappa_{z} \geq 2,  \kappa \leq \gamma^{i}}\binom {\gamma^{i}} {\kappa} \bigg(
	\partial_{z}\partial^{\kappa-2e_{3}}\big(K\partial_x\tilde{u}^{\epsilon} +K\partial_y\big(K(u^{s}+\tue)\big)\big)  \partial_{z}^{2}\partial^{\gamma^{i}-\kappa}({u}^s + \tilde{u}^{\epsilon}) 
	\bigg) \bigg|_{z=0}
	\\
	&\quad
	-\sum_{ \kappa_{z} \geq 1,  \kappa \leq \gamma^{i}}\binom {\gamma^{i}} {\kappa} \bigg(\partial_{z}\partial^{\kappa-e_{3}}\big(\partial_{y}K (u^s+\tue)\big) ~ \partial_{z}\partial^{\gamma^{i}-\theta}\big(K(u^{s}+\tilde{u}^\epsilon)\big) \bigg)\bigg|_{z=0} \\
	&\quad- \partial^{\gamma^{i}}\big(K\partial^3_{z}(u^s+\tilde{u}^{\epsilon})\big)
	\big|_{z=0}.
\end{align*}
We can also check each term, which all satisfies Proposition \ref{boundary reduction}.

\noindent\textbf{Case 3.} This situation is much easier than cases 1 and 2. We only need to check
the terms like 
$$-2\epsilon \big(\partial^{\beta^i}\partial_{x}\partial_{z}(u^s+\tilde{u}^{\epsilon}) \big)\big(\partial^{\gamma^i}\partial_{x}\partial_{z}(K(u^s+\tilde{u}^\epsilon ))\big)-2\epsilon \big(\partial^{\beta^i}\partial_{y}\partial_{z}(u^s+\tilde{u}^{\epsilon}) \big)\big(\partial^{\gamma^i}\partial_{y}\partial_{z}(K(u^s+\tilde{u}^\epsilon ))\big),$$
it is obvious that all these
terms satisfy Proposition \ref{boundary reduction}. This completes the proof of Proposition \ref{boundary reduction}.
\end{proof}

\begin{proof}[\textbf{Proof of Corollary 3.3}]
By  \ref{uv t=0},   Proposition \ref{boundary reduction} and the definition of $\epsilon\mu^{\epsilon}$  given in the equation \ref{regular-shear-prandtl}, it follows that
\begin{align*}
\partial^{n}_{xy}\mu^{\epsilon}(x, y, 0)=0, \ \ \partial^{2}_{z}\partial^{n}_{xy}\mu^{\epsilon}(x, y, 0)=0.
\end{align*}
Taking $t=0$ for $\eqref{zu4}$ and  evaluating at $z=0$, we obtain
	\begin{align*}
	\epsilon\partial^4_{z}
	\mu^{\epsilon}|_{z=0}
	&=\big(2\partial_{z}(u^s_{0} + \tilde{u}_{0}+\epsilon\mu^{\epsilon})\partial_{x}\partial_{z}(\tilde{u}_{0}+\epsilon\mu^{\epsilon}) \big)|_{z=0} +\big(2\partial_{z}\big(K(u^s_{0} + \tilde{u}_{0})+\epsilon\mu^{\epsilon})\big)\partial_{y}\partial_{z}(\tilde{u}_{0}+\epsilon\mu^{\epsilon})  \big) \big|_{z=0}\\
	&	\quad-\big(\partial_{x}\partial_{z}( \tilde{u}_{0}+\epsilon\mu^{\epsilon}) \partial_{z}(u^s_{0}+\tilde{u}_{0}+\epsilon\mu^{\epsilon})\big)|_{z=0} +\partial_{y}\partial_{z}( K(u^s_0+\tilde{u}_{0})+\epsilon\mu^{\epsilon}) \partial_{z}(u^s_{0}+\tilde{u}_{0}+\epsilon\mu^{\epsilon}) \big) \big|_{z=0}\\
	&\quad
	-\partial^{4}_{z}\tilde{u}_{0} |_{z=0}
	,
\end{align*}
and thus, using the representation of $\partial^{4}_{z}\tilde{u}_{0} |_{z=0}$ given in $\eqref{uv4 t=0}$,
	\begin{align*}
	\partial^4_{z}
	\mu^{\epsilon}|_{z=0}
	&=\big(
	2\partial_{z}(u^s_{0} + \tilde{u}_{0})\partial_{x}\partial_{z}\mu^{\epsilon}
   +2\partial_{z}\mu^{\epsilon}\partial_{x}\partial_{z}\tilde{u}_{0}
   +2\epsilon\partial_{z}\mu^{\epsilon}\partial_{x}\partial_{z}\mu^{\epsilon} 
	  \big) \big|_{z=0}
	 \\
	&\quad+\big(
	2\partial_{z}(K(u^s_{0} + \tilde{u}_{0}))\partial_{y}\partial_{z}\mu^{\epsilon}
	+2\partial_{z}\mu^{\epsilon}\partial_{y}\partial_{z}\tilde{u}_{0}
	+2\epsilon\partial_{z}\mu^{\epsilon}\partial_{y}\partial_{z}\mu^{\epsilon}   
	\big) \big|_{z=0}\\
	&	\quad-\big(
	\partial_{x}\partial_{z} \tilde{u}_{0} \partial_{z}\mu^{\epsilon}
	+\partial_{x}\partial_{z}\mu^{\epsilon} \partial_{z}(u^s_{0}+\tilde{u}_{0})
	+\epsilon\partial_{x}\partial_{z}\mu^{\epsilon} \partial_{z}\mu^{\epsilon}
	\big) \big|_{z=0}\\
	&	\quad-\big(
	\partial_{y}\partial_{z} \big(K(u^s_{0}+\tilde{u}_{0}) \big)\partial_{z}\mu^{\epsilon}
	+\partial_{y}\partial_{z}\mu^{\epsilon} \partial_{z}(u^s_{0}+\tilde{u}_{0}) 
	+\epsilon\partial_{y}\partial_{z}\mu^{\epsilon} \partial_{z}\mu^{\epsilon}
	\big) \big|_{z=0}
	.
\end{align*}
Analogously, taking the values at $t=0$ for \eqref{uve 6},  we have
{\small
	\begin{align*}
	&\partial^6_{z}\mu^{\epsilon}|_{z=0}\\
	&=2\partial_{z}\useo\partial_{z}^{3}\partial_{x}\mu^{\epsilon} |_{z=0}+2\partial_{z}\mu^{\epsilon}\partial_{z}^{3}\partial_{x}\tilde{u}_{0} |_{z=0}+2\epsilon\partial_{z}\mu^{\epsilon}\partial_{z}^{3}\partial_{x}\mu^{\epsilon} |_{z=0}\\
	&\quad 	+
	2\partial_{z}^{3}\useo\partial_{z}\partial_{x}\mu^{\epsilon} |_{z=0}
	+2\partial_{z}^{3}\mu^{\epsilon}\partial_{z}\partial_{x}\tilde{u}_{0} |_{z=0}
	+2\epsilon\partial_{z}^{3}\mu^{\epsilon}\partial_{z}\partial_{x}\mu^{\epsilon} |_{z=0}\\
	&\quad+ \sum_{1\leq j\leq 3}C^4_j \bigg(\partial^j_z\useo \partial^{4-j}_z\partial_x\mu^{\epsilon} 
	+\partial^j_z\mu^{\epsilon} \partial^{4-j}_z\partial_x\tilde{{u}}_{0}
	+\epsilon\partial^j_z\mu^{\epsilon} \partial^{4-j}_z\partial_x\mu^{\epsilon}
	\bigg)\bigg|_{z=0}\\
	&\quad+2\partial_{z}\big(K(u^s+\tilde{{u}}^{\epsilon})\big)\partial_{z}^{3}\partial_{y}\mu^{\epsilon} |_{z=0}+2\partial_{z}\mu^{\epsilon}\partial_{z}^{3}\partial_{y}\tilde{u}_{0} |_{z=0}+2\epsilon\partial_{z}\mu^{\epsilon}\partial_{z}^{3}\partial_{y}\mu^{\epsilon} |_{z=0}\\
	&\quad 	+
	2\partial_{z}^{3}\big(K(u^s+\tilde{{u}}^{\epsilon})\big)\partial_{z}\partial_{y}\mu^{\epsilon} |_{z=0}
	+2\partial_{z}^{3}\mu^{\epsilon}\partial_{z}\partial_{y}\tilde{u}_{0} |_{z=0}
	+2\epsilon\partial_{z}^{3}\mu^{\epsilon}\partial_{z}\partial_{y}\mu^{\epsilon} |_{z=0}\\
	&\quad+ \sum_{1\leq j\leq 3}C^4_j \bigg(\partial^j_z\big(K(u^s+\tilde{{u}}^{\epsilon})\big) \partial^{4-j}_z\partial_y\mu^{\epsilon} 
	+\partial^j_z\mu^{\epsilon} \partial^{4-j}_z\partial_y\tilde{u}_{0}
	+\partial^j_z\mu^{\epsilon} \partial^{4-j}_z\partial_y\mu^{\epsilon}
	\bigg)\bigg|_{z=0}\\
	&\quad-\partial_{z}^{3}\uveo \partial_{z}\mu^{\epsilon} |_{z=0}
	-\partial_{z}^{3}\uvmuo \partial_{z}(u^s_{0}+\tilde{u}_{0}) |_{z=0}
	-\epsilon\partial_{z}^{3}\uvmuo \partial_{z}\mu^{\epsilon} |_{z=0}
	\\
	&\quad -\partial_{z}^{3}(u^s_{0}+\tilde{u}_{0}) \partial_{z}\uvmuo |_{z=0}
	-\partial_{z}^{3}\mu^{\epsilon} \partial_{z}\uveo |_{z=0}
	-\epsilon\partial_{z}^{3}\mu^{\epsilon} \partial_{z}\uvmuo |_{z=0}
	\\
	&\quad  -\sum_{1\leq j\leq 3}C^4_{j+1} \bigg(
	-\partial^{j}_z \uveo \partial^{4-j}_z \mu^{\epsilon}
	-\partial^{j}_z \uvmuo \partial^{4-j}_z (u^s_{0} + \tilde{u}_{0})
	-\partial^{j}_z \uvmuo \partial^{4-j}_z \mu^{\epsilon}
	\bigg) \bigg|_{z=0} \\
	&\quad 
	-\underline{4\partial_{x}\partial_{z}\tu_{0}\partial_{x}^{2}\partial_{z}\tilde{u}_{0}|_{z=0}}
	-4\epsilon\partial_{x}\partial_{z}\tu_{0}\partial_{x}^{2}\partial_{z}\mu^{\epsilon}|_{z=0}
	-4\epsilon\partial_{x}\partial_{z}\mu^{\epsilon}\partial_{x}^{2}\partial_{z}\tilde{u}_{0}|_{z=0}
	-4\epsilon^{2}\partial_{x}\partial_{z}\mu^{\epsilon}\partial_{x}^{2}\partial_{z}\mu^{\epsilon}|_{z=0}
	\\
	&\quad
	-\underline{4\partial_{x}\partial_{z}\big(K(u^s_{0}+\tilde{{u}}^{\epsilon}_{0})\big)\partial_{x}\partial_{y}\partial_{z}\tilde{u}_{0}|_{z=0}}
	-4\epsilon\partial_{x}\partial_{z}\big(K(u^s_{0}+\tilde{{u}}^{\epsilon}_{0})\big)\partial_{x}\partial_{y}\partial_{z}\mu^{\epsilon}|_{z=0}
	-4\epsilon\partial_{x}\partial_{z}\mu^{\epsilon}\partial_{x}\partial_{y}\partial_{z}\tilde{u}_{0}|_{z=0}
	-4\epsilon^{2}\partial_{x}\partial_{z}\mu^{\epsilon}\partial_{x}\partial_{y}\partial_{z}\mu^{\epsilon}|_{z=0}\\
	&\quad
	-\underline{2\partial_{x}^2K\partial_{z}(u^s_0+\tu_0)\partial_{y}\partial_z\tu_0 |_{z=0}}
	-2\epsilon\partial_{x}^2K\partial_{z}(u^s_0+\tu_0)\partial_{y}\partial_z\mu_\epsilon |_{z=0}
	-2\epsilon\partial_{x}^2K\partial_{z}\mu_\epsilon\partial_{y}\partial_z\tu_0 |_{z=0}
	-2\epsilon^2\partial_{x}^2K\partial_{z}\mu_\epsilon\partial_{y}\partial_z\mu_\epsilon |_{z=0}
	\\
		&\quad
	-\underline{2\partial_{x}K\partial_{x}\partial_{z}(u^s_0+\tu_0)\partial_{y}\partial_z\tu_0 |_{z=0}}
	-2\epsilon\partial_{x}K\partial_{x}\partial_{z}(u^s_0+\tu_0)\partial_{y}\partial_z\mu_\epsilon |_{z=0}
	-2\epsilon\partial_{x}K\partial_{x}\partial_{z}\mu_\epsilon\partial_{y}\partial_z\tu_0 |_{z=0}
	-2\epsilon^2\partial_{x}K\partial_{x}\partial_{z}\mu_\epsilon\partial_{y}\partial_z\mu_\epsilon |_{z=0}
	\\
	&\quad+\underline{2 \partial_{x}\partial_{y}\partial_{z}\tilde{{u}}_{0} \partial_{y}\partial_{z}\tilde{u}_{0}|_{z=0}}
	+2\epsilon \partial_{x}\partial_{y}\partial_{z}\tilde{{u}}_{0} \partial_{y}\partial_{z}\mu^{\epsilon})|_{z=0}
	+2\epsilon \partial_{x}\partial_{y}\partial_{z}\mu^{\epsilon} \partial_{y}\partial_{z}\tilde{u}_{0}|_{z=0}
	+2\epsilon^{2} \partial_{x}\partial_{y}\partial_{z}\mu^{\epsilon} \partial_{y}\partial_{z}\mu^{\epsilon}|_{z=0}\\
	&\quad 
	+\underline{2\epsilon \partial_{x}\partial_{y}\partial_{z}\big(K(u^s_{0}+\tilde{{u}}^{\epsilon}_{0})\big) \partial_{x}\partial_{z}\tilde{u}_{0}|_{z=0}}
	+2\epsilon \partial_{x}\partial_{y}\partial_{z}\big(K(u^s_{0}+\tilde{{u}}^{\epsilon}_{0})\big) \partial_{x}\partial_{z}\mu^{\epsilon}|_{z=0}
	+2\epsilon \partial_{x}\partial_{y}\partial_{z}\mu^{\epsilon} \partial_{x}\partial_{z}\tilde{u}_{0}|_{z=0}
	+2\epsilon^{2} \partial_{x}\partial_{y}\partial_{z}\mu^{\epsilon} \partial_{x}\partial_{z}\mu^{\epsilon}|_{z=0}
	\\
	&\quad
	+\underline{
		 \partial_{y}\partial_{z}\big(\partial_{x}^{2}K(u^s_0+\tu_0)\big) \partial_{z}(u^s_0+\tilde{u}_0)|_{z=0}
	}
+
\epsilon \partial_{y}\partial_{z}\big(\partial_{x}^{2}K(u^s_{0}+\tilde{{u}}_{0})\big) \partial_{z}\mu_\epsilon|_{z=0}
+
\epsilon \partial_{y}\partial_{z}(\partial_{x}^2K\mu_\epsilon ) \partial_{z}(u^s_0+\tilde{u}_0)
+\epsilon^2 \partial_{y}\partial_{z}
(\partial_{x}^2K\mu_\epsilon )\partial_{z}\mu_\epsilon
\\
	&\quad
+\underline{
	\partial_{y}\partial_{z}\big(\partial_{x}K\partial_{x}(u^s_0+\tu_0)\big) \partial_{z}(u^s_0+\tilde{u}_0)|_{z=0}
}
+
\epsilon \partial_{y}\partial_{z}\big(\partial_{x}K\partial_{x}(u^s_{0}+\tilde{{u}}_{0})\big) \partial_{z}\mu_\epsilon|_{z=0}
+
\epsilon \partial_{y}\partial_{z}(\partial_{x}K\partial_{x}\mu_\epsilon ) \partial_{z}(u^s_0+\tilde{u}_0)
+\epsilon^2 \partial_{y}\partial_{z}
(\partial_{x}K\partial_{x}\mu_\epsilon )\partial_{z}\mu_\epsilon
\\
	&\quad
	-\underline{4\partial_{y}\partial_{z}\tu_{0}\partial_{x}\partial_{y}\partial_{z}\tilde{u}_{0}|_{z=0}}
	-4\epsilon\partial_{y}\partial_{z}\tu_{0}\partial_{x}\partial_{y}\partial_{z}\mu^{\epsilon}|_{z=0}
	-4\epsilon\partial_{y}\partial_{z}\mu^{\epsilon}\partial_{x}\partial_{y}\partial_{z}\tilde{u}_{0}|_{z=0}
	-4\epsilon^{2}\partial_{y}\partial_{z}\mu^{\epsilon}\partial_{x}\partial_{y}\partial_{z}\mu^{\epsilon}|_{z=0}
	\\
	&\quad
	-\underline{4\partial_{y}\partial_{z}\big(K(u^s_{0}+\tilde{{u}}^{\epsilon}_{0})\big)\partial_{y}^{2}\partial_{z}\tilde{u}_{0}|_{z=0}}
	-4\epsilon\partial_{y}\partial_{z}\big(K(u^s_{0}+\tilde{{u}}^{\epsilon}_{0})\big)\partial_{y}^{2}\partial_{z}\mu^{\epsilon}|_{z=0}
	-4\epsilon\partial_{y}\partial_{z}\mu^{\epsilon}\partial_{y}^{2}\partial_{z}\tilde{u}_{0}|_{z=0}
	-4\epsilon^{2}\partial_{y}\partial_{z}\mu^{\epsilon}\partial_{y}^{2}\partial_{z}\mu^{\epsilon}|_{z=0}
	\\
	&\quad
	-\underline{2\partial_{y}^2K\partial_{z}(u^s_0+\tu_0)\partial_{y}\partial_z\tu_0 |_{z=0}}
	-2\epsilon\partial_{y}^2K\partial_{z}(u^s_0+\tu_0)\partial_{y}\partial_z\mu_\epsilon |_{z=0}
	-2\epsilon\partial_{y}^2K\partial_{z}\mu_\epsilon\partial_{y}\partial_z\tu_0 |_{z=0}
	-2\epsilon^2\partial_{y}^2K\partial_{z}\mu_\epsilon\partial_{y}\partial_z\mu_\epsilon |_{z=0}
	\\
	&\quad
	-\underline{2\partial_{y}K\partial_{y}\partial_{z}(u^s_0+\tu_0)\partial_{y}\partial_z\tu_0 |_{z=0}}
	-2\epsilon\partial_{y}K\partial_{y}\partial_{z}(u^s_0+\tu_0)\partial_{y}\partial_z\mu_\epsilon |_{z=0}
	-2\epsilon\partial_{y}K\partial_{y}\partial_{z}\mu_\epsilon\partial_{y}\partial_z\tu_0 |_{z=0}
	-2\epsilon^2\partial_{y}K\partial_{y}\partial_{z}\mu_\epsilon\partial_{y}\partial_z\mu_\epsilon |_{z=0}
	\\
	&\quad+\underline{2 \partial_{x}\partial_{y}\partial_{z}\tilde{{u}}_{0} \partial_{y}\partial_{z}\tilde{u}_{0}|_{z=0}}
	+2\epsilon \partial_{x}\partial_{y}\partial_{z}\tilde{{u}}_{0} \partial_{y}\partial_{z}\mu^{\epsilon}|_{z=0}
	+2\epsilon \partial_{x}\partial_{y}\partial_{z}\mu^{\epsilon} \partial_{y}\partial_{z}\tilde{u}_{0} |_{z=0}
	+2\epsilon^{2} \partial_{x}\partial_{y}\partial_{z}\mu^{\epsilon} \partial_{y}\partial_{z}\mu^{\epsilon}|_{z=0}
	\\
	&\quad
	+\underline{2 \partial_{y}^{2}\partial_{z}\big(K(u^s_{0}+\tilde{{u}}^{\epsilon}_{0})\big)\partial_{y}\partial_{z}\tilde{u}_{0} |_{z=0}}
	+2\epsilon \partial_{y}^{2}\partial_{z}\big(K(u^s_{0}+\tilde{{u}}^{\epsilon}_{0})\big)\partial_{y}\partial_{z}\mu^{\epsilon} |_{z=0}
	+2\epsilon \partial_{y}^{2}\partial_{z}\mu^{\epsilon}\partial_{y}\partial_{z}\tilde{u}_{0} |_{z=0}
	+2\epsilon^{2} \partial_{y}^{2}\partial_{z}\mu^{\epsilon}\partial_{y}\partial_{z}\mu^{\epsilon}|_{z=0}
	\\
	&\quad
+\underline{
	\partial_{y}\partial_{z}\big(\partial_{y}^{2}K(u^s_0+\tu_0)\big) \partial_{z}(u^s_0+\tilde{u}_0)|_{z=0}
}
+
\epsilon \partial_{y}\partial_{z}\big(\partial_{y}^{2}K(u^s_{0}+\tilde{{u}}_{0})\big) \partial_{z}\mu_\epsilon|_{z=0}
+
\epsilon \partial_{y}\partial_{z}(\partial_{y}^2K\mu_\epsilon ) \partial_{z}(u^s_0+\tilde{u}_0)
+\epsilon^2 \partial_{y}\partial_{z}
(\partial_{y}^2K\mu_\epsilon )\partial_{z}\mu_\epsilon
\\
&\quad
+\underline{
	\partial_{y}\partial_{z}\big(\partial_{y}K\partial_{y}(u^s_0+\tu_0)\big) \partial_{z}(u^s_0+\tilde{u}_0)|_{z=0}
}
+
\epsilon \partial_{y}\partial_{z}\big(\partial_{y}K\partial_{y}(u^s_{0}+\tilde{{u}}_{0})\big) \partial_{z}\mu_\epsilon|_{z=0}
+
\epsilon \partial_{y}\partial_{z}(\partial_{y}K\partial_{y}\mu_\epsilon ) \partial_{z}(u^s_0+\tilde{u}_0)
+\epsilon^2 \partial_{y}\partial_{z}
(\partial_{x}K\partial_{y}\mu_\epsilon )\partial_{y}\mu_\epsilon
	\, ,
\end{align*}
}
where the underlined terms in the above equation are new and different from those in \eqref{uve 6},
thus
	\begin{align}
	\begin{split}
		&\partial^6_{z}\mu^{\epsilon}|_{z=0}\\
		&=\sum_{\beta, \gamma}C_{\beta, \gamma}\partial^{\beta_{x} }
		\partial^{\beta_{y} }\partial^{\beta_{z} +1}\big(K(u^s_0+\tilde{{u}}^{\epsilon})\big)\partial^{\gamma_{x} }
		\partial^{\gamma_{y} }\partial^{\gamma_{z} +1}\mu^{\epsilon}\big|_{z=0}\\
		&\quad+
		\sum_{\beta, \gamma}C_{\beta, \gamma}\partial^{\beta_{x} }
		\partial^{\beta_{y} }\partial^{\beta_{z} +1}\mu^{\epsilon}\partial^{\gamma_{x} }
		\partial^{\gamma_{y} }\partial^{\gamma_{z} +1}\big(u^s_{0} + \tilde{u}_{0}\big)\big|_{z=0}\\
		&\quad+
		\sum_{\beta, \gamma}C_{\beta, \gamma}\partial^{\beta_{x} }
		\partial^{\beta_{y} }\partial^{\beta_{z} +1}\big(K\mu^{\epsilon}\big)\partial^{\gamma_{x} }
		\partial^{\gamma_{y} }\partial^{\gamma_{z} +1}\mu^{\epsilon}\big|_{z=0}\\
		&\quad 
		-\underline{4\partial_{x}\partial_{z}\tu_{0}\partial_{x}^{2}\partial_{z}\tilde{u}_{0}|_{z=0}}
		-\underline{4\partial_{x}\partial_{z}\big(K(u^s_{0}+\tilde{{u}}^{\epsilon}_{0})\big)\partial_{x}\partial_{y}\partial_{z}\tilde{u}_{0}|_{z=0}}
		-\underline{2\partial_{x}^2K\partial_{z}(u^s_0+\tu_0)\partial_{y}\partial_z\tu_0 |_{z=0}}
		-\underline{2\partial_{x}K\partial_{x}\partial_{z}(u^s_0+\tu_0)\partial_{y}\partial_z\tu_0 |_{z=0}}
		\\
		&\quad
		+\underline{2 \partial_{x}\partial_{y}\partial_{z}\tilde{{u}}_{0} \partial_{y}\partial_{z}\tilde{u}_{0}|_{z=0}}
		+\underline{2\epsilon \partial_{x}\partial_{y}\partial_{z}\big(K(u^s_{0}+\tilde{{u}}^{\epsilon}_{0})\big) \partial_{x}\partial_{z}\tilde{u}_{0}|_{z=0}}
		\\
		&\quad
			+\underline{
			\partial_{y}\partial_{z}\big(\partial_{x}^{2}K(u^s_0+\tu_0)\big) \partial_{z}(u^s_0+\tilde{u}_0)|_{z=0}
		}
		+\underline{
			\partial_{y}\partial_{z}\big(\partial_{x}K\partial_{x}(u^s_0+\tu_0)\big) \partial_{z}(u^s_0+\tilde{u}_0)|_{z=0}
		}
		\\
		&\quad
		+\underline{4\partial_{y}\partial_{z}\tu_{0}\partial_{x}\partial_{y}\partial_{z}\tilde{u}_{0}|_{z=0}}
		-\underline{4\partial_{y}\partial_{z}\big(K(u^s_{0}+\tilde{{u}}^{\epsilon}_{0})\big)\partial_{y}^{2}\partial_{z}\tilde{u}_{0}|_{z=0}}
		-\underline{2\partial_{y}^2K\partial_{z}(u^s_0+\tu_0)\partial_{y}\partial_z\tu_0 |_{z=0}}
		-\underline{2\partial_{y}K\partial_{y}\partial_{z}(u^s_0+\tu_0)\partial_{y}\partial_z\tu_0 |_{z=0}}
		\\
		&\quad
+\underline{2 \partial_{x}\partial_{y}\partial_{z}\tilde{{u}}_{0} \partial_{y}\partial_{z}\tilde{u}_{0}|_{z=0}}
		+\underline{2 \partial_{y}^{2}\partial_{z}\big(K(u^s_{0}+\tilde{{u}}^{\epsilon}_{0})\big)\partial_{y}\partial_{z}\tilde{u}_{0} |_{z=0}}
		\\
		&\quad
		+\underline{
			\partial_{y}\partial_{z}\big(\partial_{y}^{2}K(u^s_0+\tu_0)\big) \partial_{z}(u^s_0+\tilde{u}_0)|_{z=0}
		}
		+\underline{
			\partial_{y}\partial_{z}\big(\partial_{y}K\partial_{y}(u^s_0+\tu_0)\big) \partial_{z}(u^s_0+\tilde{u}_0)|_{z=0}
		}
		\, .
	\end{split}
\end{align}
where the summation is for the index $0 \leq\beta(\beta_{x}, \beta_{y}, \beta_{z}) \leq 3;  0 \leq\beta(\gamma_{x}, \gamma_{y}, \gamma_{z}) \leq 3; \beta+\gamma \leq 3.$ The new underlined term
means that the regularizing $\epsilon\partial_{x}^{2}(\tu, \tv)+\epsilon\partial_{y}^{2}(\tu, \tv)$
 term has an affect on the boundary. This is why we add a corrector.

More generally, for $6 \leq 2p \leq m$, we have that $\partial_{z}^{2p+2}\mu^{\epsilon}\big|_{z=0}$ is a liner combination of the terms of the form 
\begin{align*}
&	\prod\limits_{i=1}^{q_{1}}  \partial^{\beta^{i}}\partial_{z}\Big(  u^s_{0} + \tilde{u}_{0}  \Big) \Big|_{z=0}  \times 
	\prod\limits_{j=1}^{q_{2}}  \partial^{ \gamma^{i}}\partial_{z}\mu^{\epsilon}\big|_{z=0},
\\
&	\prod\limits_{i=1}^{q_{1}}  \partial^{ \beta^{i}}\partial_{z} \mu^{\epsilon} \Big|_{z=0}  \times 
	\prod\limits_{j=1}^{q_{2}}  \partial^{ \gamma^{i}}\partial_{z}\Big( K(u^s_{0} + \tilde{u}_{0}) \Big) \Big|_{z=0},
\\
&	\prod\limits_{i=1}^{q_{1}}  \partial^{ \beta^{i}}\partial_{z}\Big( K\mu^{\epsilon} \Big) \Big|_{z=0}  \times 
	\prod\limits_{j=1}^{q_{2}}  \partial^{ \gamma^{i}}\partial_{z}\mu^{\epsilon} \Big|_{z=0},
\end{align*}
and
\begin{align*}
	\prod\limits_{i=1}^{q_{1}}  \partial^{ \beta^{i}}\partial_{z}\Big( u^s_{0} + \tilde{u}_{0}  \Big) \Big|_{z=0}  \times 
	\prod\limits_{j=1}^{q_{2}}  \partial^{ \gamma^{i}}\partial_{z}\Big( K(u^s_{0} + \tilde{u}_{0}) \Big) \Big|_{z=0},
\end{align*}
where $\beta^{i}+\gamma^{j} \leq 2p-1$, and $\partial_{z}^{2p+2}\mu^{\epsilon}\big|_{z=0}$ is determined by $u^s_{0}+\tilde{{u}}^{\epsilon}_{0}, K(u^s_{0}+\tilde{{u}}^{\epsilon}_{0})$, and the low order derivatives of $ K\mu^{\epsilon}$ and  $\mu^{\epsilon}$.

We now construct a polynomial function $\tilde{\mu}^{\epsilon}$ on $z$ by the following Taylor expansion
\begin{align*}
\tilde{\mu}^{\epsilon}(x, y, z)=\sum^{\frac{m}{2}+1}_{p=3}\tilde{\mu}^{\epsilon, 2p}(x, y)\frac{z^{2p}}{(2p)!},
\end{align*}
where 		
\begin{align*}		
&\tilde{\mu}^{\epsilon, 6}(x, y)\\
&=			-\underline{4\partial_{x}\partial_{z}\tu_{0}\partial_{x}^{2}\partial_{z}\tilde{u}_{0}|_{z=0}}
-\underline{4\partial_{x}\partial_{z}\big(K(u^s_{0}+\tilde{{u}}^{\epsilon}_{0})\big)\partial_{x}\partial_{y}\partial_{z}\tilde{u}_{0}|_{z=0}}
-\underline{2\partial_{x}^2K\partial_{z}(u^s_0+\tu_0)\partial_{y}\partial_z\tu_0 |_{z=0}}
-\underline{2\partial_{x}K\partial_{x}\partial_{z}(u^s_0+\tu_0)\partial_{y}\partial_z\tu_0 |_{z=0}}
\\
&\quad
+\underline{2 \partial_{x}\partial_{y}\partial_{z}\tilde{{u}}_{0} \partial_{y}\partial_{z}\tilde{u}_{0}|_{z=0}}
+\underline{2\epsilon \partial_{x}\partial_{y}\partial_{z}\big(K(u^s_{0}+\tilde{{u}}^{\epsilon}_{0})\big) \partial_{x}\partial_{z}\tilde{u}_{0}|_{z=0}}
\\
&\quad
+\underline{
	\partial_{y}\partial_{z}\big(\partial_{x}^{2}K(u^s_0+\tu_0)\big) \partial_{z}(u^s_0+\tilde{u}_0)|_{z=0}
}
+\underline{
	\partial_{y}\partial_{z}\big(\partial_{x}K\partial_{x}(u^s_0+\tu_0)\big) \partial_{z}(u^s_0+\tilde{u}_0)|_{z=0}
}
\\
&\quad
+\underline{4\partial_{y}\partial_{z}\tu_{0}\partial_{x}\partial_{y}\partial_{z}\tilde{u}_{0}|_{z=0}}
-\underline{4\partial_{y}\partial_{z}\big(K(u^s_{0}+\tilde{{u}}^{\epsilon}_{0})\big)\partial_{y}^{2}\partial_{z}\tilde{u}_{0}|_{z=0}}
-\underline{2\partial_{y}^2K\partial_{z}(u^s_0+\tu_0)\partial_{y}\partial_z\tu_0 |_{z=0}}
-\underline{2\partial_{y}K\partial_{y}\partial_{z}(u^s_0+\tu_0)\partial_{y}\partial_z\tu_0 |_{z=0}}
\\
&\quad
+\underline{2 \partial_{x}\partial_{y}\partial_{z}\tilde{{u}}_{0} \partial_{y}\partial_{z}\tilde{u}_{0}|_{z=0}}
+\underline{2 \partial_{y}^{2}\partial_{z}\big(K(u^s_{0}+\tilde{{u}}^{\epsilon}_{0})\big)\partial_{y}\partial_{z}\tilde{u}_{0} |_{z=0}}
\\
&\quad
+\underline{
	\partial_{y}\partial_{z}\big(\partial_{y}^{2}K(u^s_0+\tu_0)\big) \partial_{z}(u^s_0+\tilde{u}_0)|_{z=0}
}
+\underline{
	\partial_{y}\partial_{z}\big(\partial_{y}K\partial_{y}(u^s_0+\tu_0)\big) \partial_{z}(u^s_0+\tilde{u}_0)|_{z=0}
}.
\end{align*}
Taking $\mu^\epsilon = \chi(z)\tilde{\mu}^{\epsilon}$ with $\chi \in C^\infty([0, +\infty[); \chi(z)=1, 0 \leq z \leq 1; \chi(z)=0, z\geq 2$, we complete the proof of the Corollary. 
\end{proof}

\section{Derivation of formal transformations}
In this appendix,  we will derive the formal transformations of system \eqref{regular-shear-prandtl} for $g^{n}$.     Define $g^{n}$ and the other quantities as follows
\begin{align*}
	&g^{n}=\bigg(\frac{\partial_{xy}^{n} u}{u_{z}^{s}+u_{z}}\bigg)_{z}, \qquad 
	\eta_{xz}=\frac{ u_{xz}}{u_{z}^{s}+u_{z}}, \qquad
	 \eta_{yz}
	=\frac{ u_{yz}}{u_{z}^{s}+u_{z}}, \qquad 
	\eta_{zz}
	=\frac{ u_{zz}^{s}+u_{zz}}{u_{z}^{s}+u_{z}}.
\end{align*}
Taking $\partial_{xy}^{n}$ in $\eqref{regular-shear-prandtl}_{1}$, we obtain
\begin{equation}
	\begin{cases}
		\partial_{t}\partial_{xy}^{n}u+(u^s+u)\partial_{x}\partial_{xy}^{n}u+K(u^s+u)\partial_{y}\partial_{xy}^{n}u+\partial_{xy}^{n}w\partial_{z}(u^s+u)
		-\partial_{z}^{2}\partial_{xy}^{n} u-\epsilon\partial_{x}^{2}\partial_{xy}^{n} u-\epsilon\partial_{y}^{2}\partial_{xy}^{n} u
		\\
		=
		-\sum_{i=1}^{n}C^i_n\partial_{xy}^i {u} \, \partial^{n  -i}_{xy} \partial_{x}{u}
		-\sum_{i=1}^{n}C^i_n\partial_{xy}^i (K(u^s+u)) \, \partial^{n  -i}_{xy} \partial_{y}{u}
		-\sum_{i=1}^{n}C^i_n\partial_{xy}^i {\varphi} \,\partial^{n  -i}_{xy} {w},
	\end{cases}
	\label{C1}
\end{equation}
where the notation tilde $\sim$ and the superscript $\epsilon$ are dropped.
Dividing \eqref{C1} with $(u^{s}_{z}+u_{z})$ ,  taking $\partial_{z}$ on the resulting equations, we have
\begin{align}
\begin{split}
	&\partial_{z}\bigg(\frac{\partial_{t}\partial_{xy}^{n} u}{u_{z}^{s}+u_{z}}\bigg)
	+(u+u^{s})\partial_{z}\bigg(\frac{\partial_{x}\partial_{xy}^{n} u}{u_{z}^{s}+u_{z}}\bigg)
	+K(u^s+u)\partial_{z}\bigg(\frac{\partial_{y}\partial_{xy}^{n} u}{u_{z}^{s}+u_{z}}\bigg)
	\\
	&=-K\partial_{y}\partial_{xy}^{n} u+\partial_{y}\partial_{xy}^{n}\big(K(u^s+u)\big)+\partial_{z}\bigg(\frac{\partial_{z}^{2}\partial_{xy}^{n} u+\epsilon\partial_{x}^{2}\partial_{xy}^{n} u+\epsilon\partial_{y}^{2}\partial_{xy}^{n} u}{u_{z}^{s}+u_{z}}\bigg)
	\\
	&\quad+
	\partial_{z}\left\lbrace \bigg(-\sum_{i=1}^{n}C^i_n\partial_{xy}^i {u} \, \partial^{n  -i}_{xy} \partial_{x}{u}
	-\sum_{i=1}^{n}C^i_n\partial_{xy}^i (K(u^s+u)) \, \partial^{n  -i}_{xy} \partial_{y}{u}
	-\sum_{i=1}^{n}C^i_n\partial_{xy}^i {\varphi} \,\partial^{n  -i}_{xy} {w}\bigg) \bigg{/} (u_{z}^{s}+u_{z})\right\rbrace .
\end{split}
\label{C2}
\end{align}
Directly compute some terms in equality \ref{C2} as follows,
\begin{align*}
	\partial_{z}\bigg(\frac{\partial_{t}\partial_{xy}^{n} u}{u_{z}^{s}+u_{z}}\bigg) 
	&= \partial_z\bigg(\partial_t \frac{{\partial^n_{xy}u} }{u^s_z + {u}_z} + \partial_z^{-1} g^{n} \, \frac{\partial_t  {u}_z + \partial_t u^s_z }{u^s_z + {u}_z} \bigg)\\
	&= \partial_t g^{n} + \partial_z\bigg( \partial_z^{-1} g^{n}  \frac{ \partial_t u^s_z+\partial_t  {u}_z}{u^s_z + \tilde{u}_z} \bigg),\\
	(u^s+u)\partial_{z}\bigg(\frac{\partial_{x}\partial_{xy}^{n} u}{u_{z}^{s}+u_{z}}\bigg) 
& 
= (u^s + {u})\bigg\{\partial_x\partial_z\left(\frac{{\partial^n_{xy} u}}{u^s_z + {u}_z}\right) +\partial_z \left(\frac{{\partial^n_{xy}u}}{u^s_z +{u}_z}\right)  \frac{{u}_{xz}}{u^s_z+ {u}_z}  + \left(\frac{\partial^n_{xy}{u}}{u^s_z + {u}_z}\right)\partial_z\left( \frac{{u}_{xz}}{u^s_z + {u}_z}\right)\bigg\}\\
& =  (u^s + {u})(\partial_x g^{n} + g^{n} \eta_{xz} + \partial_z^{-1} g^{n}  \partial_z \eta_{xz}),\\
	K(u^s+u)\partial_{z}\bigg(\frac{\partial_{y}\partial_{xy}^{n} u}{u_{z}^{s}+u_{z}}\bigg) 
& =  K(u^s + {u})(\partial_y g^{n} + g^{n} \eta_{yz} + \partial_z^{-1} g^{n}  \partial_z \eta_{yz}),
\\
	\partial_z\left(\frac{{\partial^2_{z} {\partial^n_{xy} u}}}{u^s_z + {u}_z}\right)
	&=\partial^2_z g^{n}  + 2 \partial_z g^{n} \eta_{zz} +2 g^{n} \partial_z \eta_{zz} -  4g^{n}\eta_{zz}^2- 8 \partial_z^{-1} g^{n} \eta_{zz} \partial_z \eta_{zz}
	\\
	&\quad + \partial_z\bigg(\partial_z^{-1} g^{n}\, \frac{u_{zzz}^s + {u}_{zzz}}{u^s_z + {u}_z}\bigg),
\\
	\partial_z\left(\frac{{\partial^2_{x} {\partial^n_{xy} u}}}{u^s_z + {u}_z}\right)
	&=\partial^2_x g^{n}  + 2 \partial_x g^{n} \eta_{xz} +2 \partial_{x}\partial_{z}^{-1}g^{n} \partial_z \eta_{xz} -  2g^{n}\eta_{xz}^2- 4 \partial_z^{-1} g^{n} \eta_{xz} \partial_z \eta_{xz}
	\\
	&\quad + \partial_z\bigg(\partial_z^{-1} g^{n}\, \frac{u_{xxz}^s + {u}_{zzz}}{u^s_z + {u}_z}\bigg),
\\
	\partial_z\left(\frac{{\partial^2_{y} {\partial^n_{xy} u}}}{u^s_z + {u}_z}\right)
	&=\partial^2_y g^{n}  + 2 \partial_y g^{n} \eta_{yz} +2 \partial_{y}\partial_{z}^{-1}g^{n} \partial_z \eta_{yz} -  2g^{n}\eta_{yz}^2- 4 \partial_z^{-1} g^{n} \eta_{yz} \partial_z \eta_{yz}
	\\
	&\quad + \partial_z\bigg(\partial_z^{-1} g^{n}\, \frac{u_{yyz}^s + {u}_{zzz}}{u^s_z + {u}_z}\bigg).
\end{align*}
Next, Combining the above  estimations yields the following
formal transformations of equation $\eqref{regular-shear-prandtl}_{1}$ for $g^{n}$,
\begin{align*}
	\partial_{t}g^{n}+(u^{s}+u)\partial_{x}g^{n}+K(u^s+u)\partial_{y}g^{n}-\partial_{z}^{2}g^{n}-\epsilon\partial_{x}^{2}g^{n}-\epsilon\partial_{y}^{2}g^{n}-2\epsilon\partial_{x}\partial_{z}^{-1}g^{n}\partial_{z}\eta_{xz}-2\epsilon\partial_{y}\partial_{z}^{-1}g^{n}\partial_{z}\eta_{yz}
	=M(g^{n}),
\end{align*}
with $K(g^{n})=\sum_{i=1}^{7}K_{i}(g^{n})$, where
\begin{align*}
	M_{1}(g^{n})&=-\big\lbrace (u^s + {u})(g^{n} \eta_{xz} + \partial_z^{-1} g^{n}  \partial_z \eta_{xz}) +K(u^s + {u})(  g^{n} \eta_{yz} + \partial_z^{-1} g^{n}  \partial_z \eta_{yz})\big\rbrace ,
\\
	M_{2}(g^{n})&= 2 \partial_z g^{n} \eta_{zz} +2 g^{n} \partial_z \eta_{zz} -  4g^{n}\eta_{zz}^2- 8 \partial_z^{-1} g^{n} \eta_{zz} \partial_z \eta_{zz},
\\
	M_{3}(g^{n})&=\epsilon\big(
	2 \partial_x g^{n} \eta_{xz}  -  2g^{n}\eta_{xz}^2- 4 \partial_z^{-1} g^{n} \eta_{xz} \partial_z \eta_{xz}
	\big) ,
\\
M_{4}(g^{n})&=\epsilon\big(
 2 \partial_y g^{n} \eta_{yz}  -  2g^{n}\eta_{yz}^2- 4 \partial_z^{-1} g^{n} \eta_{yz} \partial_z \eta_{yz}
\big) ,
	\\
	M_{5}(g^{n})&=-K\partial_{y}\partial_{xy}^{n} u+\partial_{y}\partial_{xy}^{n}\big(K(u^s+u)\big),
\\
	M_{6}(g^{n})&=\partial_z\left\lbrace \partial_z^{-1} g^{n} \bigg(  \frac{(u^s+u)\partial_{x}\varphi+K(u^s+u)\partial_{y}\varphi+w(u_{zz}^s+\partial_{z}\varphi)}{u^s_z + \tilde{u}_z}-\partial_{y}K(u^s+u)\bigg)\right\rbrace ,
	\\
	M_{7}(g^{n})&=\partial_{z}	\left\lbrace \bigg(-\sum_{i=1}^{n}C^i_n\partial_{xy}^i {u} \, \partial^{n  -i}_{xy} \partial_{x}{u}
	-\sum_{i=1}^{n}C^i_n\partial_{xy}^i {(K(u^s+u))} \, \partial^{n  -i}_{xy} \partial_{y}{u}
	-\sum_{i=1}^{n}C^i_n\partial_{xy}^i {\varphi} \,\partial^{n  -i}_{xy} {w}\bigg) \bigg{/} (u_{z}^{s}+u_{z})\right\rbrace ,
\end{align*}
where we have used the relation for $K_{6}(g^{n})$,
\begin{align*}
	&\partial_t u_{z}^{s} +\partial_{t} u_{z}
	- u^{s}_{zzz}-u_{zzz} -\epsilon u_{xxz}+\epsilon u_{yyz}
	\\
	&=- (u^s + u) \partial_x\varphi - K(u^s + u) \partial_y\varphi-{w} (u^s_{zz} +  \partial_{z}\varphi) 
	+(u^{s}_{z}+u_{z})\partial_{y}K(u^s+u)
	.
\end{align*}

Finally, we only need to verify the boundary condition $\partial_{z} g^{n}\big|_{z=0}=0$ in the above equation. Noticing that 
\begin{align*}
	\left(\frac{{\partial^2_{z} {\partial^n_{xy} u}}}{u^s_z + {u}_z}\right)
	&=\partial_z g^{n}  + 2  g^{n} \eta_{zz}  -4 \partial_z^{-1} g^{n} \eta_{zz}^2
	+\partial_z^{-1} g^{n}\bigg( \frac{u_{zzz}^s + {u}_{zzz}}{u^s_z + {u}_z}\bigg),
\end{align*}
and 
\begin{align*}
&\eta_{zz}|_{z=0}\frac{ u_{zz}}{u_{z}^{s}+u_{z}}\bigg|_{z=0}=0, \\
&g^{n}|_{z=0}=\frac{\partial_{xy}^{n} u}{u_{z}^{s}+u_{z}}
\bigg|_{z=0}=0,
\end{align*}
then we have
\begin{align*}
\partial_{z} g^{n}\big|_{z=0}=0.
\end{align*}
\end{appendices}

\section*{Statement about conflicting interests}
The authors declare that there are no conflicts of interest.

\section*{Acknowledgements}
Yuming Qin was supported  by the NNSF of China with contract number 12171082, the fundamental
research funds for the central universities with contract numbers 2232022G-13, 2232023G-13, 2232024G-13.

	\phantomsection

\end{document}